\documentclass[11pt, a4paper]{article}
\usepackage{amsthm,amsmath,amssymb,enumerate}
\title{\textbf{Vorticity internal transition layers \\ for the Navier-Stokes equations}}

\newcommand{\ep}{\nu}
\providecommand{\Phi}{\textbf{\Phi}}

\providecommand{\rot}{\mathrm{curl} \,}
\providecommand{\dive}{\mathrm{div} \,}
\providecommand{\sup}{\mathrm{sup}}
\providecommand{\inf}{\mathrm{inf}}

\providecommand{\R}{\mathbb{R}}

\providecommand{\N}{\mathbb{N}}

\renewcommand{\phi}{\varphi}

\newtheorem{theo}{Theorem}[section]

\newtheorem{lem}{Lemma}[section]

\newtheorem{prop}{Proposition}[section]
\newtheorem{rem}{Remark}[section]

\newtheorem{defi}{Definition}[section]

\author{Franck Sueur\footnote{Laboratoire Jacques-Louis Lions
Universit\'e Pierre et Marie Curie - Paris 6;
175 Rue du Chevaleret
75013 Paris,
FRANCE }}
\date{\today}

\begin{document}

\maketitle


\begin{abstract}
We deal with the incompressible
Navier-Stokes equations, in two and three dimensions, 
when some vortex patches are prescribed as initial data
 i.e. when there is an internal boundary across which the vorticity is discontinuous.
We show -thanks to an asymptotic expansion- 
that  there is a sharp but smooth
variation of the fluid vorticity into a internal layer
moving with  the flow of the Euler equations;
 as long as this later exists and as $  t  << 1/\ep$, where $\ep  $ is the viscosity coefficient.
\end{abstract}

\section{Introduction and overview of the results}

In this paper we deal with the equations:
\begin{align}
\label{R1} \partial_t v^\ep + v^\ep \cdot \nabla v^\ep + \nabla p^\ep
&= \ep \Delta v^\ep
\\ \label{R2} \text{div} \ v^\ep 
&=0,
\end{align}
where the spatial derivative $x$ runs into the whole space $\R^d$ for $d = 2$ or $3$, 
$v^\ep $ and $p^\ep$ respectively denote  the velocity and the pressure of a fluid and
$\ep \geqslant 0 $ is the viscosity coefficient.
The derivative fields $v^\ep \cdot \nabla$ and $\Delta$ are with respect to $x$ and applied component-wise to the vector-valued function $v^\ep$ whereas $\nabla p^\ep $ denotes the gradient (also with respect to $x$) of the scalar function $p^\ep$.
The equations  (\ref{R1})-(\ref{R2}) are respectively
 Newton's law of motion and mass
conservation for a homogeneous (constant density) 
fluid so that (\ref{R2}) reads as an incompressibility
condition. When $\ep = 0 $ the equations  (\ref{R1})-(\ref{R2}) are referred as  the {\bf Euler} equations whereas  $\ep  >0 $ corresponds to the  {\bf Navier-Stokes} equations.

The study of these equations naturally
depends on the function space in which initial data is  provided, 
and in which solutions are sought.
In this paper we deal with  the class of initial data 
of the  vortex patches.
Roughly speaking such data correspond to Lipschitz initial velocity $v_I$ whose vorticity $\omega_0 := \text{curl}\  v_0$ is discontinuous across a hypersurface of $\R^d$, say given by the equation $ \phi_0 (x)=0 $.
To be more precise we recall first the definition of Holder spaces: 

\begin{defi}[Holder spaces]
\label{defi0}
For an open subset $\mathcal{O}$ of $\R^d$, 
for  $s$ in the set $ \N$ of all natural numbers including $0$ and  a real $0  < r<1$ the Holder space  $C^{s,r}  (\mathcal{O} )$ of the functions of class $C^{s}  (\mathcal{O} )$ such that 
\begin{align*}
  \| u   \|_{  C^{s,r} (\mathcal{O}) } := \text{sup}_{ |\alpha| \leqslant s}   \big(  \|   \partial^\alpha u  \|_{L^\infty ( \mathcal{O}) }  
+ \text{sup}_{ x \neq y \in  \mathcal{O}} \frac{ |\partial^\alpha u (x) -  \partial^\alpha u (y)| }{  |x - y|^r }  \big) <  + \infty  ,
\end{align*}
and  $C^{s,r}_c  (\mathcal{O} )$ the subset of the functions in  $C^{s,r}  (\mathcal{O})$ which are compactly supported.
 \end{defi}
 
 We now define precisely what we mean in this paper by vortex patches as initial data.
 
\begin{defi}[Vortex patches]
\label{defi}
\begin{enumerate}
\item {\bf Smoothness of the boundary.} 
\label{VP1}
Let be given  an  open subset  $\mathcal{O}_{0,+}  $ of   $ \R^d$  of class  $C^{s+1,r}$. 
 This means that there exists a function $\phi_0 \in
C^{s+1,r}(\R^d, \R)$ such that an equation of the boundary  $\partial \mathcal{O}_{0,+}$ 
 is given by  $\partial \mathcal{O}_{0,+} = \{ \phi_0 =0  \}$, such that  $ \mathcal{O}_{0,\pm} = \{ \pm \phi_0 > 0  \}$, where $\mathcal{O}_{0,-}  $ is  the interior of the complementary of  $\mathcal{O}_{0,+}  $, and $  \nabla \phi_0 \neq 0$ in a neighborhood of  $\partial \mathcal{O}_{0,+}$.
 \item {\bf Boundedness of the boundary.} 
\label{VP2}
We assume that the boundary  $\partial \mathcal{O}_{0,+}$  is bounded.
 \item {\bf Smoothness of the velocity and of the vorticity.} 
\label{VP3}
We consider  an initial velocity   $v_0$  in  $L^2_\sigma (\R^d)$ -the space of  $L^2_{loc} (\R^d)$ divergence free vector fields-  whose vorticity  $\omega_0 := \text{curl}\  v_0$ is  in the Holder space  $C^{s,r}_c  (\mathcal{O}_{0,\pm}  ) $, that is a compactly supported vorticity which is 
$C^{s,r}$ on each side of the internal  boundary  $\partial \mathcal{O}_{0,+} = \{ \phi_0 =0  \}$.
 \end{enumerate}
 \end{defi}
 
Such an initial configuration is referred as a vortex patch since in the two dimensional case $d=2$, such an initial vorticity  $\omega_0$ can be for example the characteristic function of a  bounded domain of class  $C^{s+1,r}$, 
 as it was  introduced by Majda in  \cite{majda} (however the more special case of the ellipses was already considered by Kirchoff).

 \begin{rem}
 \rm
One should notice that in the $3$d case the  initial velocity   $v_0$ is automatically in  $L^2_\sigma (\R^d)$ whereas in the $2$d case it is in  $L^2_\sigma (\R^d)$  if and only if the total mass of the vorticity vanishes. 
\end{rem}

Numerous results about existence of solutions for such data (and even for more general ones) are available either for the case of the Euler equations and for the 
Navier-Stokes equations.
The main goal of this paper  is to obtain an expansion for the solutions of the 
Navier-Stokes equations which describes as well as possible their behaviour with respect not only to the variables $t,x$ but also to the viscosity coefficient $\ep $, in the  limit  $\ep t \rightarrow 0$.
 However it is useful to reexamine first the case of the  Euler equations.
 

\subsection{The inviscid case}
\label{compendium}

We first gather from the literature the following compendium of results regarding the inviscid case: 

\begin{theo}[Chemin,  Gamblin and X. Saint-Raymond, P. Zhang, C. Huang, P. Serfati ]
\label{litt}
For  initial velocities  $v_0$  as described in the definition \ref{defi} the following holds true:
\begin{enumerate}

\item {\bf Existence and uniqueness.} 
\label{litt1}
There exists $T>0$ (which can be taken arbitrarily large when $d=2$)
 and a unique solution $v^0  \in L^\infty (0,T;Lip (\R^d))  $ to the  {\bf Euler} equations:
\begin{align}
\label{S1} \partial_t v^0 + v^0 \cdot \nabla v^0 &= \nabla p^0,
\\ \label{S2} \text{div} \ v^0 &=0,
\end{align}
with $v_0$  as initial velocity. From now on we denote $D$ the vectorfield:
\begin{align}
\label{D}
D:= \partial_t  + v^0 \cdot \nabla .
\end{align}
\item {\bf Propagation of smoothness of the vorticity.}
\label{litt2}
Moreover for each $t \in (0,T)$ the vorticity
\begin{align}
 \label{V1}
  \omega^0  :=  \text{curl}_x\  v^0  (t)
 \end{align}
   is  $C^{s,r}_c  (\mathcal{O}_{\pm} (t) )$  where  $\mathcal{O}_{\pm}  (t)$
   are respectively the  transported by the flow of  $ \mathcal{O}_{0,\pm} $ at time $t$ that is
   $\mathcal{O}_{\pm}  (t) :=  \mathcal{X}^0 (t, \mathcal{O}_{0,\pm} ) $
 where $\mathcal{X}^0$ is the flow of particle trajectories defined by
 $  \partial_t  \mathcal{X}^0(t,x) =  v^0 (t, \mathcal{X}^0 (t,x))$ with initial data $ \mathcal{X}^0 (0,x) = x$. 
 
 \item {\bf Propagation of smoothness of the boundary.}
 \label{litt3}
  For each $t \in (0,T)$ the boundary  $\partial \mathcal{O}_{+} (t)$ of the domain  $\mathcal{O}_{\pm}  (t)$   is  $C^{s+1,r}$.
  
\item {\bf Propagation of smoothness of the level function.}
 \label{litt4}
For each $t \in (0,T)$ the boundary  $\partial \mathcal{O}_{+} (t)$ is given by the equation  $\partial \mathcal{O}_{+}(t) = \{ \phi^{0} (t,.)=0  \}$, 
where
 $$ \phi^{0} \in L^\infty (0,T;   C^{1,r}  (\R^d)   ) \cap L^\infty (0,T;   C^{s+1,r}  ( \mathcal{O}_{\pm}  (t)  ))  $$
  verifies 
\begin{eqnarray}
\label{eik1} D  \phi^{0}  =0 ,
\\ \label{eik2}  \phi^{0}|_{ t =0 }  = \phi_0 .
\end{eqnarray}
Moreover  $ \mathcal{O}_{\pm} (t)= \{ \pm \phi^{0} (t,.) > 0  \}$ and there exists $ \eta >0$ such that for 
 $0  \leqslant t \leqslant T$, and $x$ such that $| \phi^{0} (t,x) | < \eta $
 the vector $n(t,x) :=  \nabla_x \phi^{0} (t,x)$ satisfies 
 $n(t,x) \neq 0$.  

\item {\bf Incompressible Rankine-Hugoniot.}
 \label{litt5}
For each $t \in (0,T)$ the function $( \omega^0  . n)(t,.)$ is $C^{0,r}$ on $ \{ | \phi^{0} (t,.) | < \eta \} $.

 \item {\bf Smoothness in time.}
  \label{litt6}
Finally the internal boundary $\partial \mathcal{O}_{+} (t)$ is analytic with respect to time and the restrictions on each side of the boundary of the  flow $\mathcal{X}^0$ are also analytic with respect to time with values in $C^{s+1,r}$.
\end{enumerate}
\end{theo}

\begin{rem}
We stress that the vector  $n $  is not an unit vector even if it is so at $t=0$, since it is stretched when time proceeds according to the equation:
\begin{eqnarray}
\label{neq}
D n  = -  {}^{t} (\nabla v^0) . n    .
  \end{eqnarray}
\end{rem}

Let us start a bunch of comments about Theorem  \ref{litt} with the two-dimensional case.
In this case the existence and
uniqueness of a solution of
the Euler equations in the more general case of an initial vorticity
which is a bounded function with compact
support was proved by V. I. Yudovich  in $1963$ cf.  \cite{yudo}.
 The corresponding velocity field is log-lipschitzian and admits
 a  bicontinuous flow $\mathcal{X}$.
 Since in the two-dimensional
case the vorticity is simply transported by the flow: 
 $$ \omega^0 (t, \mathcal{X}^0 (t,x)) =  \omega_0 (x) $$ 
 this implies that in the case of a  vortex patch as initial data
the vorticity $ \omega^0$ at time $t$
remains a vortex patch relative to a domain which is
homeomorphic to the initial domain.
  However if one only use that the initial vorticity
 is a bounded function with compact
support, one can only get that the smoothness of Yudovich's flow $\mathcal{X}$
 is exponentially decreasing: it is in $C^{\exp{- \alpha t }  }$
  where $\alpha$ depends on the initial vorticity  (cf. for example  \cite{bcd}, Theorem $7.26$).
Therefore in the case of a vortex patch, Yudovich's approach only provide that the boundary of
 a vortex patch  is in  $C^{\exp{- \alpha t }}$.
The point is to know how the smoothness of the boundary of the patch really evolves.
Numerical experiments of  Zabusky  \cite{zabu} in $1979$
suggested that singularities of the boundary of the patches would develop, presumably in finite time whereas the ones of Buttke \cite{buttke} suggested a loss of smoothness.
 In $1979$ Majda  \cite{majda} theoretically studied the evolution of  the boundary of piecewise constant vortex patches, by a contour dynamic approach, announcing  local-in-time existence and  conjecturing that there are smooth initial curves such that the curve becomes nonrectifiable in finite time.
   Constantin and Titi   \cite{ct}  studied a quadratic approximation of the equation governing the evolution of the boundary for which  S. Alinhac \cite{alinhacpatch} found some evidence of finite time breakdown.

 As a consequence the  global-in-time persistence of the  initial  $C^{s+1,r}$ smoothness of the boundary was very surprising when proved  by Chemin  \cite{cheminens}  (see also the notes  \cite{che1},  \cite{che2}, \cite{che3} and the survey  \cite{chemin}).
  Actually Chemin's approach (which already yielded  local-in-time existence 
  in  \cite{cheinventiones}, \cite{che4})
is  based on a $2$-local paradifferential  analysis that is on the propagation of smoothness with respect to non-smooth vectorfields.
  The idea of using regularity properties with respect to a family of vector fields goes back to the work of Bony cf.  \cite{sbony}
 and to the work of Alinhac \cite{salinhac} and Chemin \cite{schemin} in the case of non-smooth vector fields.
  The vortex patch problem falls in this latter case since it holds smoothness properties with respect to vector fields  tangential to the boundary of the patch which moves with the velocity of the fluid.
 \begin{rem}
 \label{strat}
This approach  even allows to deal with  more general cases
 since it could apply for example to  initial data which are irregular with respect not only to one hypersurface 
  $ \{  \phi^{0} (t,.) = 0  \}$ but to the whole foliation of the hypersurfaces $ \{ \phi^{0} (t,.) = a  \}$ for $a \in  \R$.
   \end{rem}

 In \cite{bertozzi} Bertozzi and Constantin  succeed to recover global-in-time persistence of the  $C^{s+1,r}$ smoothness of the boundary 
  in the special of constant vortex patches by the contour dynamic approach.
  
  The persistence of piecewise smoothness of the vorticity (Holder regularity up to the boundary) was proved later first by Depauw in the case $s=0$  \cite{depauw},  \cite{depauw2}
 and by Huang \cite{huangremarks} in the general case $s$ in  $ \N$ with a Lagrangian approach.
   \\
 \\  Chemin's approach  was extended to the three dimensional case $d=3$ by Gamblin and X. Saint-Raymond  in  \cite{gamblin2}, but it is only a short time result since 
  in the $3$d case the vorticity $ \omega^0$ is stretched along particle trajectories  
according to the formula 
  $$ \omega^0 (t, \mathcal{X}^0 (t,x)) := D_x \mathcal{X}^0 (t,x) .  \omega_0 (x) =  \omega_0 (x) \cdot \nabla_x \mathcal{X}^0 (t,x) $$ 
 which "solves" the equation
\begin{align}
\label{taneq} D   \omega^0  =  \omega^0  \cdot \nabla v^0 ,
\\ \label{taneq2}   \omega^0 |_{ t =0 }  =  \omega_0 .
\end{align}
A rough  estimate  of the  vorticity stretching is given by
\begin{eqnarray}
\label{stretching} 
    \|   \omega^0 (t)  \|_{L^\infty } & \leqslant&   e^{ \int_0^t  \|  v^0 \|_{Lip }} \  \|   \omega_0  \|_{L^\infty }
  \end{eqnarray}
  Of course in some specific situations the  vorticity stretching vanishes so that the existence and regularity results are in fact global in time  i.e. $T>0$ can be taken arbitrarily large as in the $2d$ case. 
  For instance a particular situation is the one of an axisymmetric initial velocity, as considered by  Gamblin and  Saint-Raymond (and others). We will not specifically consider this case here.

Indeed the statement of the  $3d$ case in Theorem  \ref{litt} is not strictly given in  \cite{gamblin};
roughly speaking  Gamblin and  Saint-Raymond deal with tangential smoothness  in the case $s=0$.
The tangential smoothness in the general case $s \in \N$ was  alluded in the comment $(ii)$ of section $1.d$ of Gamblin and  Saint-Raymond and rigorously proved by P. Zhang in the couple of papers  \cite{zhang},  \cite{zhang2} written in Chinese. 
The persistence of piecewise  $C^{0,r}$ smoothness (the case $s=0$) was proved by Huang in \cite{huangsingular}  by mean of a Lagrangian approach (see also \cite{dutrifoy} section $3.1$).
We didn't found any proof of the persistence of higher order piecewise  $C^{s,r}$ smoothness in the literature. 
This is one reason which leads us to now give a few details of the whole proof of the $3$d case, for sake of completeness, including a proof of the persistence of piecewise  $C^{s,r}$ smoothness in the Eulerian framework of Chemin,  Gamblin and Saint-Raymond.
Another reason is that we will use some extensions of the subsequent methods in the following sections regarding the case of the Navier-Stokes equations.
\\
 \\  Theorem \ref{litt} gathers the results we will need regarding the inviscid case but many other works have been done around the case strictly addressed in Theorem \ref{litt}.
Let us  mention 
      some papers 
 by Chemin and  Danchin about the two-dimensional case for vortex patches with singular boundary. 
By using the pseudo-locality of the Paley-Littlewood theory and transport equations by Log-Lipschitz velocities
 they show that if the initial boundary is regular
      ($C^{0,r}$ in Chemin's works  \cite{che5},  \cite{cheminasterisque} chapter $9$, generalized into 
      $C^{s,r}$ in  Danchin's paper  \cite{danchinsingular})
      apart from a closed subset, it remains regular for all time apart from the closed subset transported by the flow.
 When the singularity of the boundary is a cusp 
 the corresponding velocity is
 Lipschitz what allows 
 Danchin  
 cf. \cite{danchincusp1},  \cite{danchincusp2},  \cite{danchincusp3} to prove
the  global stability of the cusp with conservation of the order (actually he gives a global result of persistence of conormal regularity with respect to vector fields vanishing at a singular point, which generalizes the structure of a cusp).
 
  On the other hand several numerical simulations, based on an adaptive multi-scale
algorithm using wavelet interpolation, are
presented, showing that  corners are unstable  either immediately becoming a cusp in the case of an initial sharp
corner or immediately becoming flat in the case of an initial obtuse corner.

In  \cite{depauw}  (see also  \cite{depauw2}) N. Depauw  addresses the case when the domain of the fluid  is a bounded subset of $\bold R^2$ with smooth boundary. When the boundary of the vortex patch is away from the domain boundary, it will remain so under the evolution, and Chemin's result remains valid without restriction. 
When the vortex patch is tangent to the material boundary (for transversal intersection there is a counterexample of Bahouri and Chemin), the author proves that if the initial vortex patch is of class $C^{1,r}$ there exists a unique solution in this class of vortex patches at least up to some time $T_*>0$. 
He also proves local-in-time existence and uniqueness for several mutually tangent vortex patches in $\bold R^2$.

 In \cite{dutrifoy} Dutrifoy proves local-in-time existence and uniqueness in 
  the case when the domain of the fluid  is a bounded subset of $\bold R^3$ with smooth boundary,
   when the boundary of the vortex patch is away from the domain boundary, without restriction
   and when the vortex patch is tangent to the material boundary, under a technical condition.
   In this later case, his method
   allowed to complete the previously mentioned $2$d local-in-time result by Depauw into a global one (with a slight loss of smoothness).

Finally the  $2$d  result by Depauw was recently  extended to global in time results by Huang in \cite{huangglobal} without loss of smoothness.
 

\subsection{The viscous case}
\label{viscous}

Let us now consider  the Navier-Stokes equations. Numerous works (see for example \cite{sawada}, \cite{dudong},  \cite{germain}, \cite{carlen})
show how the additional term $\ep \Delta v^\ep$  regularizes the initial data.\footnote{It is even crucial in existence theories both  in the approach of Leray and in the one of Kato.}
  Still to the author's knowledge  the smoothing effect has so far been analyzed in a global and isotropic way, and this is precisely one
 main goal of this paper to describe how 
  the smoothing effect is (micro-)localized in the case where vortex patches are prescribed as initial data.
In effect we will show a  {\em conormal} smoothing of the  initial vorticity discontinuity
  into a layer of width $\sqrt{\ep t}$ around
 the hypersurface $  \{  \phi^{0} (t,.) =0 \}$ where the discontinuity has been transported at the time $t$  by the flow of the Euler equations.
 Hence  the fluid vorticity
\begin{eqnarray}
\label{curlep}
 \omega^\ep :=  \text{curl}\  v^\ep
\end{eqnarray}
 depends -locally (innerly)- on an extra "fast" scale: $\frac{\phi^{0}(t,x) }{\sqrt{\ep t}}$ (cf. section \ref{fastscale}) and  will be described by 
an expansion of the form
 \begin{align}
 \label{ex}
  \omega^\ep (t,x) \sim   \Omega (t,x,\frac{\phi^{0}(t,x) }{\sqrt{\ep t}} ) ,
\end{align}
where the {\em viscous profile}  $\Omega  (t,x,X) $ admits some limits when the "fast" variable $X$ goes to the infinities: $X \rightarrow \pm \infty$.

The idea to associate a viscous profile to an inviscid discontinuity seems to date back to 
Rankine  \cite{rankine}  and is widely known when the  discontinuity
is a shock as for instance in compressible fluid mechanics  (see the recent achievements by Gu\`es, M\'etivier, Williams and Zumbrun in  \cite{GMWZ1},   \cite{GMWZ2},  \cite{GMWZ3}, \cite{GMWZ4}, \cite{GMWZ5}).
However since they are
characteristic and conservative
  the vortex patches are very different from  the shocks of the compressible fluid mechanic (which are noncharacteristic and dissipative cf. for instance \cite{metivier}). 
 We therefore would like to precise that we borrow the words "viscous profile" to the setting of shocks profile but that our setting is quite different. For instance extra scales involved are not the same in the two cases.
 
 Still we hope that the approach developped here should be extended to some other setting where smoothing of characteristic conormal singularities occurs included  the gradients jumps of compressible fluids, which are studied in the inviscid case in  \cite{alinhac1}, \cite{alinhac2}, \cite{alinhac3}, \cite{alinhac4}, \cite{tougeron}, \cite{sonic1}, \cite{sonic2}, in meteorology cf. \cite{charve}, the domain walls in ferromagnetism cf. \cite{gues}...

These precisions done let us go back to our vortex patch problem and describe the construction of the viscous profile $\Omega $ involved in the expansion (\ref{ex}).
We will look for a  viscous profile $\Omega $ of the form 
$$ \Omega   (t,x,X) =  \omega^{0}  (t,x) +  \tilde \Omega  (t,x,X),$$
where $ \tilde\Omega (t,x,X)  $ denotes 
  a  perturbation local with respect to the extra scale $X$ so that 
\begin{eqnarray}
    \label{sharp1intro} \lim_{X\rightarrow \pm \infty}    \tilde\Omega   (t,x,X) &=& 0.
 \end{eqnarray}
Hence the Navier-Stokes vorticities  $ \omega^\ep (t,x)$ will be described by  an expansion of the form
 \begin{align*}
  \omega^\ep (t,x) \sim   \omega^0 (t,x) + \tilde \omega^\ep (t,x)
  \text{  where  }
   \tilde \omega^\ep (t,x) :=
    \tilde \Omega (t,x,\frac{\phi^{0}(t,x) }{\sqrt{\ep t}}).
\end{align*}
 The dependence of the perturbation $ \tilde \omega^\ep$ on $\frac{\phi^{0}(t,x) }{\sqrt{\ep t}}$ encodes the "conormal self-similarity" of the layer. 
 Pragmatically the consequences of the condition (\ref{sharp1intro}) on the profile  $ \tilde\Omega (t,x,X)  $ at the level of the function $ \tilde \omega^\ep$ are threefold:
\begin{enumerate}

\item For any $(t,\ep )  \in (0,T) \times \R^*_+ $, $  \tilde \omega^\ep (t,x) \rightarrow 0 $ when $\phi^{0} \rightarrow \pm \infty$. This was actually our motivation to impose the condition (\ref{sharp1intro}) on the profile  $ \tilde\Omega (t,x,X)  $: it sounds natural that the viscous layer is confined to the neighbourhood  of the hypersurface where the inviscid discontinuity occurs.

\item For any $t  \in (0,T)$, for any $x  \in  \R^d  \setminus  \partial \mathcal{O}_{+}  (t)$, 
$  \tilde \omega^\ep (t,x) \rightarrow 0 $ when $\ep \rightarrow 0^+$.
This consequence is directly linked with  another strong underlying motivation to this work that is the issue of the inviscid limit 
 of the Navier-Stokes equation to the Euler ones.
The "strength" of this inviscid limit (that is the functional space where it holds) does not depend only on the presence or not of material boundaries but also 
 on the smoothness of the initial data. 
Basically the more the initial data is smooth the more the convergence is strong.
For smooth data the Navier-Stokes solutions are {\it regular} perturbations of the corresponding Euler solutions  in the inviscid limit (see for instance Swann  \cite{Swann}, Kato  \cite{kato1}, \cite{kato2}) and converge say in any Holder spaces with a rate of order $\ep t$.
We also quote here Masmoudi \cite{masmoudi} for a slight improvement.
At the other end in $2$D when  vortex sheets are prescribed 
as initial data one only knows the weak $L^2$ convergence (cf.  \cite{delort}).
The vortex patches are an intermediary case, first  studied  in
  $2$D by Constantin and Wu in  \cite{constantinwu}.  
In \cite{abidi} Abidi and Danchin found the optimal rate in  $ L^\infty (0,T; L^2 (\R^d)) $, and recently  \cite{masmoudi} extends this result to the $3$D case. 
We also refer to the papers of Hmidi \cite{hmidiholder}, \cite{hmidising} to the study of  $2$D vortex patches (including with singular boundary).
In this direction the novelty here is again to describe locally what happens
(actually the analysis performed here gives optimal estimates of convergence rates in any spaces as a simple byproduct).

\item For any $(x, \ep )  \in  \R^d \times  \R^*_+ $, $  \tilde \omega^\ep (t,x) \rightarrow 0 $ when $t \rightarrow  0^+ $. This yields that the Navier-Stokes vorticities 
  $\omega^\ep $ have the same initial value than the Euler one $\omega^0$.
  Let us mention here that the analysis can be simplified if on the contrary we allow ourselves to choose the initial data for the Navier-Stokes vorticities 
  $\omega^\ep $ since there exists some well-prepared data for which the viscous smoothing is already taken into account  (cf.  section  \ref{bienprep}). 
  
 \end{enumerate}

We will argue (cf.  section  \ref{amplitudes}) that the corresponding expansion of the velocity is  of the form 
\begin{eqnarray}
 v^\ep (t,x) &\sim & v^0  (t,x) + \sqrt{\ep t} \, V (t,x,\frac{\phi^{0}(t,x)}{\sqrt{\ep t} } ),
  \end{eqnarray}
with $V(t,x,X)$
 satisfying
   \begin{eqnarray}
\label{heu24r}   V (t,x,X)   \rightarrow 0 \quad \text{when }  \quad  X \rightarrow \pm \infty. 
   \end{eqnarray}

We will be led  (cf. section \ref{looking})  to consider for the profile $V (t,x,X)$ the  {\em linear} partial  differential equation:
 \begin{equation}
  \label{heu2r}
  L V = 0
  \end{equation}
where the differential operator $L$ is given by
  \begin{eqnarray}
   \label{defL}
  L=  \mathcal{E}  -t (D + A)
   \end{eqnarray}
where $D$ is the vectorfield in (\ref{D}), and
$\mathcal{E}$ and  $A$ are some operators of respective order $2$ and $0$   acting formally on functions $V(t,x,X)$ as follows:
 \begin{eqnarray}
 \mathcal{E} V := a  \partial_X^2  V  +  \frac{X}{2}  \partial_X V  - \frac{1}{2}  V ,
\\ A V :=   V \cdot \nabla_x v^0 
  - 2  \frac{( V \cdot\nabla_x v^0). n }{a  } n .
  \end{eqnarray}
Here  $a $ denotes a function in the space
 \begin{eqnarray}
  \label{defB}
   \mathcal{B} :=
 L^\infty([0,T], C^{0, r} (  \R^d ))  \cap L^\infty (0,T;   C^{s,r}  ( \mathcal{O}_{\pm}  (t)  )) 
 \end{eqnarray}
  such that 
 \begin{equation}
  \label{ell}
 \inf_{[0,T] \times  \R^d} a = c > 0
 \end{equation}
  and such that  
$a = | n |^2$  when  $| \phi^{0}  | < \eta $.

We expect a continuous transition of the viscous fluid velocity $ v^\ep$ and  of the viscous fluid vorticity $\omega^\ep$ (these are the Rankine-Hugoniot conditions), instead of the discontinuity of the inviscid vorticity $\omega^0$.
These continuity conditions would be translated into the following  Dirichlet-Neumann type transmission conditions  for the profile $V (t,x,X)$ on the internal boundary $\{ X=0 \}$ (cf.  section  \ref{wp1}): for any $(t,x) \in (0,T),  \times  \R^d$
 \begin{eqnarray}
 \label{heu21r}  \lbrack V  \rbrack = 0
  \text{ and }    \lbrack  \partial_X V    \rbrack   = - \frac{n \wedge  (\omega^0_+ -  \omega^0_-  )}{a} ,
  \end{eqnarray}
  where  the brackets denote the jump across  $\{ X=0 \}$ that is $ \lbrack f(t,x,X)  \rbrack :=  f  |_{X=0^+  } - f |_{X=0^- }$ and
 $ \omega^0_\pm$ are some well-chosen extensions
 of $ \omega^0 |_{\mathcal{O}_{\pm}(t)  }$.

The transmission conditions  (\ref{heu21r}) are normal for the operator $\mathcal{E} $ which is elliptic  with respect to $X$ thanks to the condition (\ref{ell}). Actually we point out here that because of its unbounded coefficient $X$ the 
operator  $\mathcal{E} $
does not strictly enter in the classical theory of elliptic operators  (with $t,x$ as parameter through the coefficient $a$). 
However we will see that it shares their main features, at least for our purposes.
For instance omitting, to simplify, the dependence  on $t$ (which is here only parametric) we have the following result  (cf.  section  \ref{wp4}):
\begin{prop}
\label{lm}
For any  $f \in   L^2 (\R^d , H^{-1} ( \R))$  and $g  \in    L^2 (\R^d )$ there is exactly one solution   $V  \in L^2 (\R^d , H^1 ( \R))$  of the equation $ \mathcal{E} V = f$ with the transmission conditions  $\lbrack V  \rbrack = 0$
 and $   \lbrack  \partial_X V    \rbrack   = g$  across  $\{ X=0 \}$.
 \end{prop}
 One can verify that it is actually the kind of equation of the layer created in the very particular case of the stationary $2$d circular vortex patches (see \cite{abidi}). 
 In this case because of the symmetry there are  neither convection nor stretching, and  the norm of the normal vector is conserved so that the profile equation is simply an ODE, whose solutions involve a Gaussian function.

 Of course the full equation  (\ref{heu2r}) is much more intricate. Roughly  speaking 
  for $t>0$ the equation (\ref{heu2r}) is hyperbolic in $t,x$ and parabolic in $t,X$; but degenerates
 for $t=0$ precisely  into the previous elliptic equation. 
 However we will show that the equation  (\ref{heu2r}), with the transmission conditions (\ref{heu21r}) and the conditions (\ref{heu24r}) at infinity are well-posed.
We stress that since the hypersurface $\{t=0\}$ is characteristic for the operator $L$ none initial condition at $t=0$ has to be prescribe for the equation (\ref{heu2r}).
We will use here a $L^2$ setting, for two reasons: first in view to future extensions we want to give a claim hopefully robust. In particular it is well-known since  \cite{brenner} that in (multi-dimensional) compressible fluid mechanics the inviscid system should be tackled in $L^2$-type spaces. This will to robustness is also the reason why we choose to put the emphasis on the velocity in this presentation, more than on the vorticity.
The second reason for  a $L^2$ setting is linked to the degeneracy at $t=0$ of the equation  (\ref{heu2r}), which leads to the existence of parasite solutions. 
For instance if we look for solutions $V$ not depending on $X$ and neglecting the term involving $A$ the equation (\ref{heu2r}) simplifies into the Fuschian  differential equation $t \partial_t V = -  \frac{V}{2}$, which admits an infinity of solutions i.e. 
$V(t) = \frac{C}{ \sqrt{t} }$, for $C \in \R$. 
However only one is in $L^2 (0,T)$, corresponding to $C=0$; and 
we expect that the scaling is enough relevant to  have a solution with $L^2 (0,T)$ smoothness, even in the case of the full equation (\ref{heu2r}).
Let us give a precise statement:  denoting $E_1$ 
  the space $$E_1 := L^2 ( (0,T) \times \R^d , H^1 ( \R) )$$
   we will prove:
\begin{theo}
\label{weakeqi}
For any $f \in E'_1$, for any $g \in  L^2 ((0,T) \times \R^d )$
there exists exactly one solution $V(t,x,X)  \in E_1$ of  $L V = f$  with the transmission conditions $ ( \lbrack  V \rbrack , \lbrack \partial_X V \rbrack ) = (0,g)$ on $\Gamma  :=  (0,T) \times \R^d  \times \{0 \} $. 
In addition the function 
$ \sqrt{t} \|  V(t,.,.)  \|_{L^2 ( \R^d \times \R) }$ is continuous on $(0,T)$.
\end{theo}
The  equation  $L V = f$  is satisfied  in the sense of distributions on both sides  $\mathcal{D}_\pm := (0,T) \times \R^d  \times \R^*_\pm$ of the hypersurface $\Gamma  :=  (0,T) \times \R^d  \times \{0 \} $. Since  $V$ is in $E_1$ the jump $ \lbrack  V \rbrack $ is in $L^2  (\Gamma)$. The sense given to the jump of the derivative $ \lbrack \partial_X V \rbrack$ is actually a part of  the problem.
The idea is to give some sense  by using the equation put in a weak form thanks to Green's formula.
We will explain this  in details in section \ref{wp}.

In the case of the transmission conditions (\ref{heu21r}) the source terms are orthogonal to $n$. It is then possible to use the uniqueness part of the previous theorem to prove that  the function $  V (t,x,X)\cdot  n (t,x) $ vanishes identically. This  orthogonality condition is self-consistent with the incompressibility condition  (see section \ref{other}) and with the linearity of the equation (see section \ref{Transparency}).

We are now interested in the smoothness of the solution $V$ given by Theorem \ref{weakeqi}. It is judicious to look again at the associated elliptic problem first.
We will prove that the solution inherits the  smoothness with respect to the usual variables $t,x$  from the coefficients; and which are piecewise smooth with respect to the fast variable $X$. 
 To be more precise, let us denote  $p-\mathcal{S}  (\R)$  the space of  the functions $f(X)$ whose restrictions to the half-lines  $ \R_\pm$ are in the Schwartz space of rapidly decreasing functions, and $\mathcal{A}$ the space (of  the functions $f(t,x,X)$):
 \begin{eqnarray*}
     \mathcal{A} := L^\infty \Big((0,T), C^{0, r} \big(  \R^d , p-\mathcal{S} (  \R) \big)\Big)
  \cap L^\infty \Big(0,T;   C^{s,r}   \big(  \mathcal{O}_{\pm}  (t)   , p-\mathcal{S} (  \R) \big)\Big).
 \end{eqnarray*}
  In section \ref{wp4} we will prove:
\begin{theo}
\label{propva0}
The solution $ V  (t,x, X)$ of  the equation $ \mathcal{E} V = f$ with the transmission conditions (\ref{heu21r})
 is in $ \mathcal{A}$. 
 \end{theo}
The main idea of the proof is to use a spectral localization with respect to $x$, which is here parametric. The point is that this process is compatible with some classical elliptic arguments used to get smoothness with respect to $X$.

 We will be able to prove the same for the full equation (\ref{heu2r}):
\begin{theo}
\label{propva}
The solution $ V  (t,x, X)$ of  the equation (\ref{heu2r})  with the transmission conditions (\ref{heu21r})
 is in $ \mathcal{A}$. 
 \end{theo}
On the opposite to Theorem  \ref{weakeqi} we will use here the particular properties of the 
  equation (\ref{heu2r}) throught the point \ref{litt6} of Theorem \ref{litt}. To explain this let us define for any Frechet space $E$ of functions depending on $t,x$ and possibly on $X$ the space
  \begin{eqnarray*}
  E_D := \{ f \in E / \ \exists C > 0 /  \   ( \frac{ D^k f}{C^k k! }   )_{ k  \in \mathbb{N} }   \text{ is bounded in } E \} .
  \end{eqnarray*}
Thanks to  the point \ref{litt6} of Theorem \ref{litt} we will be able to construct (see section \ref{wp1})
 the  extensions $\omega^0_\pm$ of the  vorticities  and the function $a$ in the space
$ \mathcal{B}_D$.
As a consequence we will actually prove in section \ref{wp'} that $V$ is even in $ \mathcal{A}_D$.

  The vorticity profile $ \Omega$ in the expansion (\ref{ex}) is then constructed  as
   \begin{eqnarray}
 \Omega (t,x,X) &:=&
  \omega^{0}_\pm  (t,x) +    n (t,x)   \wedge \partial_X V  (t,x,X)  \text{ for }  \pm X > 0.
\end{eqnarray}

If piecewise smoothness of the initial data is sufficient it is possible to go on with the expansion  with respect to $\ep t$ of the solutions of the Navier-Stokes equations. At the extreme limit 
  if the initial data is piecewise smooth on each side of the interface  $\{ \phi^{0} = 0  \}$
-that is if  $s = + \infty$- then it  is possible to write a complete  formal asymptotic expansion of the vorticity of the form:
\begin{eqnarray}
\label{complete1intro}
 \omega^\ep (t,x) =   \sum_{j \geqslant 0 }  \sqrt{\ep t}^j   \  \Omega^{j  }  (t,x,\frac{\phi^{0}(t,x)}{\sqrt{\ep t} } )  +O(\sqrt{\ep t}^\infty)  ,
 \end{eqnarray}
where the first  profile $\Omega^{0 } $ is the one previously
 constructed:  $\Omega^{0 } := \Omega $.
 This construction will be achieved in section
 \ref{hoprofile}. 

The stability of these expansions will be tackled in  section
 \ref{stab}.
To describe the results we introduce the set $\mathcal{F}$ of the families $(f^\ep (t,x) )_\ep $ of the smooth functions such that  for any $s'  \in \N$, for any $r'  \in (0,1)$, the sequence $$( \sqrt{\ep t}^{s'+r'}  \| f^{\ep }  \|_{L^\infty ( (0,T)  , C^{s',r'} (\R^3 )}     )_{0 < \ep  < \ep_0}$$ is bounded. 

\begin{theo}
\label{TheoStab}
There exists $\ep_0 > 0$ such that for $0 < \ep   < \ep_0$ for all $k \in \N$ for any $(t,x)  \in (0,T) \times \R^d $
 \begin{eqnarray}
 \label{vortexstab}
 \omega^\ep (t,x) =  \sum_{j = 0 }^k  \sqrt{\ep t}^j   \  \Omega^{j  }  (t,x,\frac{\phi^{0}(t,x)}{\sqrt{\ep t} } ) +  \sqrt{\ep t}^{k+1 }  \omega^{\ep }_R
 \end{eqnarray}
 with  $(\omega^{\ep }_R )_{0 < \ep  < \ep_0}$ in $\mathcal{F}$. 
 \end{theo}
 
Let us point out that if it is well-known that for any $\ep \geq  0$ there exists $T^\ep > 0$ and a unique solution 
$v^\ep  \in L^\infty (0,T^\ep;Lip (\R^d)) $ solution of the equations (\ref{R1})-(\ref{R2}) with $v_0$  as initial velocity, Theorem \ref{TheoStab} proves that the lifetime $T^\ep > 0$ can be extended such that $T^\ep \geq T$. 
The lifetime of such expansions is the one of the solution of the Euler equation
 ("the ground state") which traps the main part of the nonlinearity of the problem.

\section{Compendium on the Littlewood-Paley theory}
\label{paley}

 In this section we gather  some usual results of  the Littlewood-Paley theory that we will need. 
 We begin with the case of the whole space for which we
 refer for example to  the books  \cite{bcd} and \cite{TriebelBook} for a much more detailed expository.
 
 \subsection{Dyadic  decomposition}
 
 We first recall the existence of a smooth dyadic partition of unity:  there exist two radial functions $\phi $ and $\chi$  valued in the interval $\lbrack 0,1 \rbrack$, belonging respectively to 
 $\mathcal{D} (B(0,4/3))$ and to  $\mathcal{D} (C(3/4,8/3))$ (where $B(0,4/3)$ and 
  $C(3/4,8/3 )$ denote respectively the ball $B(0,4/3) := \{ \| \xi  \|_{\R^3 } < 4/3 \}$ and the annulus  $ C(3/4,8/3) := \{ 3/4 < \| \xi  \|_{\R^3 } < 8/3 \}$; and $\mathcal{D}(U) $ denotes the space of the smooth functions whose support is a compact included in $U$)
 such that 
  \begin{eqnarray*}
  \forall \xi \in \R^3 , \quad \chi (\xi) + \sum_{ j \geqslant 0 } \phi  (2^{-j} \xi ) = 1
  \\ | j-j' |  \geqslant 2 \Rightarrow \text{ supp }  \phi  (2^{-j} . ) \cap   \text{ supp } \phi  (2^{-j'} . ) =  \emptyset ,
   \\  j \geqslant 1  \Rightarrow   \Rightarrow \text{ supp }  \chi  (2^{-j} . ) \cap   \text{ supp } \phi  (2^{-j} . ) =  \emptyset .
  \end{eqnarray*}
  The Fourier transform $\mathcal{F}$
  is defined on the space of integrable functions $f \in L^1 (\R^3  ) $ by
  \begin{eqnarray*}
  \mathcal{F} f := \int_{\R^3  } e^{-x.\xi } f(x) dx
   \end{eqnarray*}
  and extended in an automorphism of the space $\mathcal{S}' (\R^3  ) $  of  the tempered distributions, which is the dual of  the Schwartz space $\mathcal{S} (\R^3  ) $  of rapidly decreasing functions.
  We will use the non-homogeneous  Littlewood-Paley decomposition (in  $\mathcal{S}'(\R^3)$)
 \begin{eqnarray*}
 Id =  \sum_{j  \geqslant -1} \Delta_j   ,
    \end{eqnarray*}
    where the so-called dyadic blocks $\Delta_j $ corresponds to the Fourier multipliers:
     \begin{eqnarray*}
  \Delta_{-1} := \chi (D)  \text{ and }   \Delta_j :=  \phi (2^{-j} D)  \text{ for }  j  \geqslant 0,
    \end{eqnarray*}
   that is
     \begin{eqnarray}
     \label{convol}
  \Delta_{-1} u (x):=  \int_{\R^3  }  \tilde{h} (y) u(x-y) dy 
    \text{ and }   \Delta_j u (x):=  2^{3j} \int_{\R^3  }  \tilde{h} (2^{j} y) u(x-y) dy  \text{ for }  j  \geqslant 0 ,
    \end{eqnarray}
    where $h:=  \mathcal{F}^{-1}  \phi$ and $ \tilde{h} :=  \mathcal{F}^{-1}  \chi$.
    We also introduce the low frequency cut-off operator $S_j :=  \sum_{k  \leqslant j-1} \Delta_j  $.

     \subsection{Besov spaces}
     
     We now recall the definition of the 
    Besov spaces $B^\lambda_{p,q}  $ on the whole space $\R^3 $
    which are, for  $\lambda  \in \R$ (the smoothness index), $p,q  \in \lbrack 1,+  \infty \rbrack$ (respectively the integral-exponent and the sum-exponent), 
    some Banach spaces defined by
     \begin{eqnarray*}
    B^\lambda_{p,q}  (\R^3  ) := \{ f \in \mathcal{S}' (\R^3  ) / \
     \| f \|_{ B^\lambda_{p,q}  (\R^3  )} :=  \|  (2^{j \lambda} \| \Delta_j  f  \|_{ L^p (\R^3  ) })_{j  \geqslant -1}  \|_{ l^q } < \infty \}.
     \end{eqnarray*}
        These spaces do not depend of the choice of the dyadic partition above.
        When $p=q= + \infty $ the  Besov spaces   $B^\lambda_{p,q} $ are simply the  Holder-Zygmund spaces   $C_*^\lambda $ and in particular
for  $\lambda    \in \R_+ - \N$ they coincide with the $C^{s,r} $ spaces of the introduction in the sense that 
  $B^\lambda_{\infty ,\infty }  (\R^3  )=  C^{s,r}  (\R^3  )$ where $s$ is the entire part of 
    $\lambda$ and $r:= \lambda - s$.
    It is  worth mentioning the the space $L^{\infty }  (\R^3  )$ is continuously embedded in the Besov space  $B^0_{\infty ,\infty }  (\R^3  )$:
      \begin{eqnarray}
       \label{inclu}
 L^{\infty }  (\R^3  )   \hookrightarrow B^0_{\infty ,\infty }  (\R^3  ) .
     \end{eqnarray}
     Conversely the spaces for $\lambda > 0$ the spaces $B^\lambda_{\infty ,\infty }  (\R^3  )$
     are continuously embedded in the  space $L^{\infty }  (\R^3  )$, and the $L^{\infty } $ norm can also be estimated by the following logarithmic interpolation inequality:
 \begin{eqnarray}
  \label{log}
\|  f \|_{ L^\infty   (\R^3)} 
&\lesssim&  L(\| f \|_{B^0_{\infty, \infty} (\R^3 )}, \| f \|_{B^\lambda_{\infty ,\infty }  (\R^3 )}) ,
 \end{eqnarray}
 where we define  for $a$ and $b$ strictly positive
  $L(a,b) := a  \ln (e + \frac{a }{ b} )$ -which is notably increasing both with respect with $a$ and $b$.
 We also recall the way Fourier multipliers act on Besov spaces: if $f$ is a smooth function such that for any multi-index $\alpha$ there exists 
 an integer $m \in \N$ 
 such that 
 \begin{eqnarray}
 \forall \xi \in \R^3 , \ |    \partial^\alpha f(\xi)  |  &\leqslant& C_\alpha  (1 +   | \xi | )^{m- | \alpha | }   ,
  \end{eqnarray}
 then for all $\lambda  \in \R$ and $p,q  \in \lbrack 1,+  \infty \rbrack$, the operator $f(D)$ is continuous from   $B^{\lambda}_{p,q} $ to  $B^{\lambda - m}_{p,q}$.
 In particular  introducing  the function $ \lambda (\zeta) := (\chi (\zeta) + |  \zeta  |^{2})^\frac{1}{2}  $, 
 where $\chi$ is in $C^\infty_0 (\R^3 ) $ positive  equal to $1$ near $0$, and  the corresponding  Fourier multiplier $\Lambda:=  \lambda (D)$ we get an
  one-parameter group of elliptic operators  $\Lambda^t$, for $t  \in  \R$,  continuous from 
  $B^{\lambda}_{p,q}  (\R^3  ) $ to  $B^{\lambda+t}_{p,q}  (\R^3  )$.

    \subsection{Bony's paraproduct}
    
    When $v$ and $w$ are two Holder distributions, we denote by $T_v w$  
Bony's paraproduct  of $w$ by $v$:
  \begin{eqnarray}
   \label{parap0}
  T_v w :=  \sum_{j \geqslant 1} S_{j-1} v \Delta_j  w 
   \end{eqnarray}
 for which
 we have the following tame estimates, for  $\lambda  \in \R$, denoting   $\| . \|_\lambda$ the  norm  in the space $B^\lambda_{\infty ,\infty }$
 \begin{eqnarray}
 \label{parap1}
\| T_v w \|_\lambda  &\lesssim&  \| v \|_{ L^\infty }   \| w \|_{\lambda}   \  \text{for}  \   \lambda  \in \R,
  \\   \label{parap2}  \| (v-T_v) w \|_\lambda  &\lesssim&  \| v \|_{ \lambda }   \| w \|_{ L^\infty}  \  \text{for}  \   \lambda  >0.
    \end{eqnarray}
    
We will also use the following commutator estimate (cf. \cite{bcd} Lemma $2.92$):
if $f$ is a smooth function homogeneous of degree $m$ away from a neighborhood of $0$ then
the  commutator 
$\lbrack T_a , f(D)  \rbrack  := T_a  f(D) -  f(D)  T_a$
between a paraproduct $T_a$ and the Fourier multiplier $f(D)$ can be estimated for any  $\lambda  \in \R$ and for any $r \in (0,1) $ by
 \begin{eqnarray}
 \label{commu}
 \| \lbrack T_a , f(D)  \rbrack u   \|_{\lambda -m + r } 
  &\lesssim&  \| a  \|_{ r }   \| u \|_{\lambda }  
  \end{eqnarray}
We also mention the following useful estimate
for commutators of the form $R_j := \lbrack v^0 \cdot \nabla  ,   \Delta_j  \rbrack  f$
 (cf. \cite{bcd} Lemma $2.93$):
 \begin{eqnarray}
 \label{commu2} 
\sup_{j  \geqslant -1}  2^{ j \lambda }  \ \| R_j   \|_{L^p  }
 &\lesssim&  
  \|  \nabla  v^0  \|_{ B^{d/p}_{p, \infty  }   \cap L^\infty  } 
   \| f  \|_{ B^ \lambda_{p, \infty  }  }  \  \text{for}  \  0 <  \lambda < 1 + d/p  ,
   \\    \label{commu3} 
   \sup_{j  \geqslant -1} 2^{ j \lambda }  \ \| R_j   \|_{L^p }  
 &\lesssim&  
  \|  \nabla  v^0  \|_{  L^\infty  } 
   \| f  \|_{ B^ \lambda_{p, \infty  }   } + 
    \|  \nabla   f  \|_{L^p  }       \|  \nabla   v^0 \|_{B^ {\lambda-1 }_{p, \infty  }   }    \  \text{for}  \  0 <  \lambda.
 \end{eqnarray}

 \subsection{Transport estimates}
 
We recall  some useful transport estimates (cf. \cite{bcd} Theorem $3.11$): for $s=-1$ or $s=0$ 
\begin{eqnarray}
\label{TEE}
 \| f  (t)   \|_{C^{s,r} (\R^3 )} 
   \leqslant
   \big(  \| f|_{t= 0}   \|_{C^{s,r} (\R^3 )} 
   +  \int_0^t  \|   (\partial_t f + v^0 \cdot \nabla f)  (\tau)     \|_{ C^{s,r} (\R^3 ) } 
   e^{-C V (\tau)  } d\tau   \big)
    e^{C  V (t)  } ,
  \end{eqnarray}
   where
\begin{eqnarray*}
V (t) :=
 \int_0^t   \|  \nabla v^0    \|_{ L^{\infty} (\R^3  ) } ds
   \end{eqnarray*}
and for any $s \geqslant 1$ 
\begin{eqnarray}
\label{TE}
 \| f  (t)   \|_{C^{s,r} (\R^3 )} 
   \leqslant
   \big(  \| f|_{t= 0}  \|_{C^{s,r} (\R^3 )} 
   +  \int_0^t  \|  (\partial_t f + v^0 \cdot \nabla f)  (\tau)     \|_{ C^{s,r} (\R^3 ) } 
   e^{-C V_{s-1,r} (\tau)  } d\tau   \big)
    e^{C  V_{s-1,r} (t)  },
  \end{eqnarray}
   where
\begin{eqnarray*}
V_{s-1,r} (t) :=
 \int_0^t   \|  \nabla v^0    \|_{ C^{s-1,r} (\R^3  ) } ds
   \end{eqnarray*}
 We refer for example to \cite{bcd} Theorem $3.11$
for a  proof of the estimates (\ref{TEE})-(\ref{TE}) by a spectral localization  that is by applying the dyadic block  $\Delta_j$ to the equation  and then making use of an energy method, taking care of the commutators thanks to Bony's paraproduct.

 \subsection{Besov spaces on Lipschitz domains}
 
  When $\Omega$ is a Lipschitz domain one usually define the Besov spaces 
      $B^\lambda_{p,q}  ( \Omega )$ (see for instance  \cite{Triebel}) as the set of restrictions of all elements of  $B^\lambda_{p,q}  (\R^3  )$ in the sense of $\mathcal{D}' (\Omega)$,   the space of the distributions on $\Omega$.
      In other words the spaces  $B^\lambda_{p,q}  ( \Omega )$ consist of
exactly those distributions $ f \in \mathcal{D}' (\Omega)$ which have extensions belonging to 
      $B^\lambda_{p,q}   (\R^3  ) $.
      Endowed with the quotient space norms they become Banach spaces.

We will make use of extension operators. 
We give the following nice general result.
\begin{theo}[Rychkov \cite{rychkov}]
\label{rychkov}
Let $\Omega$ be a Lipschitz domain in $\R^n$ with a bounded boundary.
There exists a so-called universal extension operator that is a linear operator $ext$ such that
\begin{eqnarray*}
\text{ext } : B^{s}_{p,q}     (\Omega ) \rightarrow B^{s}_{p,q}     (\R^n )
 \end{eqnarray*}
is continuous for any $s \in \R$,  $0 < p \leqslant \infty$, $0 < q \leqslant \infty$
and satisfies $(\text{ext }  u)|_{ \Omega  } = u$.
\end{theo}
 
  One can use this result to give intrinsic characterizations of Besov spaces on Lipschitz domains.
 Dispa  \cite{dispa} have shown  in particular that 
 -as in the case of the whole space-  the  Besov spaces $B^\lambda_{ \infty, \infty}    (\Omega ) $ still coincide with the  Holder spaces: for  $\lambda    \in \R_+ - \N$,  $B^\lambda_{\infty ,\infty }  ( \Omega )=  C^{s,r}  ( \Omega )$ where $s$ is the entire part of 
    $\lambda$ and $r:= \lambda - s$.

 We will also use the following result about characteristic functions as pointwise multipliers in Besov spaces.
\begin{theo}[Frazier and Jawerth \cite{frazier}]
\label{frazier}
 The characteristic function  $\chi_\Omega $   of a Lipshitz domain $\Omega  \subset \R^n$ with a bounded boundary is a pointwise multiplier in $ B^{s}_{p,q}$ (that is the map $u \mapsto \chi_\Omega  u $ is bounded from $ B^{s}_{p,q}   (\R^n )$ into  itself)  
 iff 
\begin{eqnarray*}
\text{max } \big( (\frac{1} {p} - 1) , n (\frac{1} {p} - 1)  \big) < s < \frac{1} {p} .
\end{eqnarray*}
 \end{theo}
Let us stress that this theorem says in particular that the characteristic function of a  Lipshitz domain  is a pointwise multiplier in the Holder spaces $C^{-1,r}$.

\section{On the proof of Theorem  \ref{litt}}

In this section we give a unified  proof of the $3$d case of  Theorem  \ref{litt}  in the Eulerian framework of Chemin,  Gamblin and Saint-Raymond including
 the persistence of piecewise  $C^{s,r}$ smoothness.

  \subsection{Existence from a-priori estimates}
  
We first recall that  local existence and uniqueness of solutions of the Euler equations when the initial velocity field is assumed to be in  $C^{s+1,r}  (\R^3) \cap L^{2}  (\R^3)$ (that is without a hypersurface of discontinuity) was proved by Chemin in the paper  \cite{cheminsmoothness}.  Moreover Bahouri and Dehman in \cite{BD} prove that  the criterion  obtained by  Beale,  Kato and  Majda (singularities will not form in the flow as long as there is no rapid accumulation of vorticity, i.e., as long as the integral, in time, of the $L\sp{\infty}$-norm of vorticity remains bounded) remains valid in this setting.
Actually Gamblin and X. Saint-Raymond  begin their paper \cite{gamblin} by collecting these results so that we also to refer to \cite{gamblin} Theorem $2.8$ (uniqueness) and Theorem $2.9$ (existence).
We also refer to the book  \cite{bcd} Theorem $7.1$ and  $7.20$  for a complete expository about local existence and uniqueness of solutions of the Euler equations  in supercritical Besov spaces.

To construct solutions of Theorem  \ref{litt} we proceed by regularization of the initial data
taking limits thanks to some a-priori estimates -uniform with respect to the regularization parameter- on the smooth solutions given by the results above.
We refer to
\cite{gamblin} Proposition $5.2$ for this passage to the limit, and focus now to the way to get 
the  a-priori estimates.
Actually we will precisely study in the following sections very carefully what happens in the passage to the limit when we regularize the Euler equations into slightly viscous Navier-Stokes equations.

\subsection{First smoothness property of the boundary}

As direct application of the  estimate  (\ref{TEE}) with $s=0$ we have the following estimate for the function $ \phi^{0}$ which satisfy the transport equation
 (\ref{eik1})-(\ref{eik2})
\begin{eqnarray}
\label{ordrezero}
 \| \phi^{0}   (t)   \|_{C^{0,r} (\R^3 )} 
   \leqslant
\| \phi_0  \|_{C^{0,r} (\R^3 )}  e^{C  V (t)  } .
  \end{eqnarray}

\subsection{Initial conormal vectorfields}
\label{dudulle}
Let us introduce the  vectorfields we will look for smoothness with respect to.
 We first observe that at the initial time $t=0$ the vectorfields
   \begin{eqnarray*}
  w^1_0  := \begin{bmatrix} 0 \\ -\partial_3 \phi_0  \\ \partial_2 \phi_0 \end{bmatrix} , \ w^2_0  := \begin{bmatrix} \partial_3 \phi_0  \\ 0   \\ - \partial_1 \phi_0 \end{bmatrix} , \ w^3_0  := \begin{bmatrix}- \partial_2 \phi_0  \\   \partial_1 \phi_0  \\ 0  \end{bmatrix} 
\end{eqnarray*}
are tangent to  the foliation of hypersurfaces given by $ \phi_0$.
As in our case  we deal with only one initial hypersurface   $ \{ \phi_0  = 0  \}$ of singularity we also add the vectorfields
   \begin{eqnarray*}
 w^4_0  := \begin{bmatrix} \partial_3 (\chi  x_3 ) \\ 0  \\ -\partial_1  ((1-\chi) x_3  )     \end{bmatrix} , \ w^5_0  := \begin{bmatrix} -\partial_2 (\chi  x_1 ) \\  \partial_1  ((1-\chi) x_1)     \\ 0  
   \end{bmatrix}.
  \end{eqnarray*}
to the previous triplet, where 
  $\chi$ is a  smooth function compactly supported
identically equal to $1$ when  $| \phi^{0}(t,.)  | < \eta $.
We therefore get a set of five
 divergence free vectorfields,  in  $C^{s,r} (\R^3)$ and  conormal to the  boundary  $\partial \mathcal{O}_{+,0}$ of the initial vortex patch: $w^i_0 . n = 0$ for any $1   \leqslant  i \leqslant 5$ on $\partial \mathcal{O}_{+,0}$.
Moreover this set is maximal in the sense that it satisfies
\begin{eqnarray*}
\lbrack W_0 \rbrack := \inf_{\R^3}  \  \sup_{0   \leqslant  i, j  \leqslant 5 } \  | w^i_0 \times w^j_0  | >0 .
  \end{eqnarray*}

We now  illustrate that the initial vorticity $ \omega_0$ has "good" derivatives in the directions that are conormal to the boundary of the patch.
To do this we first set up the notion of conormal derivatives.
For $u \in L^\infty  (\R^3)$ we define the distributions
\begin{eqnarray*}
w^i_0 \dagger  u := \dive ( w^i_{0} \otimes u ) =     \sum_{j}  \partial_j  (w^i_{0,j} .  u).
\end{eqnarray*}
where we denote the components of the $\R^3$-valued vector field $w^i_0$ by $w^i_{0,j}$, for $j=1,2,3$ and the spatial derivatives  $\partial_j = \partial_{x_j}$.
Since the $w^i_0$  are divergence free there holds 
$w^i_0 \dagger  u = w^i_0 \cdot \nabla u $ when this last quantity is defined.
To express iterated conormal derivatives we will denote for $k \in \N$ 
\begin{eqnarray*}
(w_0  \dagger )^\beta u := w^{\beta_1}_0  \dagger (w^{\beta_2}_0  \dagger (...(w^{\beta_k}_0  \dagger u))...)
\end{eqnarray*}
where $\beta :=  (\beta_1 , ..., \beta_k ) \in   \{1,..,5 \}^k$.
We are going to prove 
\begin{prop}
\label{dudu}
The $k$-th order conormal derivatives of the initial vorticity for any integer  $k$ between $1$ and $s$, that is the vector fields  $(w_0  \dagger )^\beta  \omega_0 $ for $\beta :=  (\beta_1 , ..., \beta_k) \in   \{1,..,5 \}^k$, are  in the space 
$C^{-1,r}  (\R^3)$.
  \end{prop}
  
  \begin{proof}
To do this we first use Rychkov's extension theorem \ref{rychkov} to get   from
  the restrictions  $\omega_0  |_{\mathcal{O}_{0,\pm}}  $
  on each side $ \mathcal{O}_{0,\pm}$ of the patch  two functions $\omega_{0,\pm} $ both in
  $C^{s,r}  (\R^3)$. Since the initial vorticity can be decomposed thanks to these extensions and to the characteristic functions $ \chi_{\mathcal{O}_{0,\pm }}$ of $ \mathcal{O}_{0,\pm}$ as 
  $\omega = \omega_{0,+} \chi_{\mathcal{O}_{0,+}}  + \omega_{0,-} \, \chi_{\mathcal{O}_{0,-}}$  it is sufficient to show that  the $(w_0  \dagger )^\beta  \omega_{0,\pm} \chi_{\mathcal{O}_{0,\pm}}$ are in  $C^{-1,r}  (\R^3)$. But since the vector fields  $w^i_0$ are tangent to the patch there holds 
 $(w_0  \dagger )^\beta  \omega_{0,\pm}  \chi_{\mathcal{O}_{0,\pm}} = \chi_{\mathcal{O}_{0,\pm}}\,  (w_0  \dagger )^\beta  \omega_{0,\pm}$.
 
 Moreover thanks to Theorem  \ref{frazier} the characteristic functions $ \chi_{\mathcal{O}_{0,+}}$ are some pointwise multipliers in the Holder spaces $C^{-1,r}$ so that the question of the 
 $C^{-1,r}  (\R^3)$ regularity of the $(w_0  \dagger )^\beta  \omega_0 $ finally
 reduces to the one of the  $(w_0  \dagger )^\beta  \omega_{0,\pm}$ which can be simply got from the paraproduct rules (\ref{parap1})-(\ref{parap2}).
\end{proof}

  \subsection{Time-dependent conormal vectorfields}
  
   Since the vorticity satisfy the transport-stretching equation  (\ref{taneq}) we expect that the tangential smoothness is conserved.
  However the boundary of the patch moves with the flow so that the notion of tangential smoothness depends itself of the solution. 
  Indeed we define  time-dependent
  conormal vectorfields  $w^i$  
  by the following formula which
   imitates the dynamic of the vorticity $ \omega^0$
\begin{eqnarray}
\label{imitating}
  w^i  (t,x) :=  ( w^i_0   \cdot \nabla_x \mathcal{X}^0 ) (t, (\mathcal{X}^0 (t,.) )^{-1} (x)) 
    \end{eqnarray}
     so that they verify the equations
\begin{align}
\label{taneqbis} \partial_t w^i + v^0  \cdot \nabla  w^i  = w^i    \cdot \nabla v^0 ,
\\ \label{taneq2bis}  w^i |_{ t =0 }  = w^i_0 .
\end{align}

By a straightforward  estimate we can show that for any $t$ the set of the vector fields $w^i (t)$ remains maximal. Indeed  (cf. \cite{gamblin}  Corollary $4.3$) they satisfy the estimate
\begin{eqnarray}
\label{maximal}
  \lbrack W (t) \rbrack^{-1}_{L^\infty  (\R^3)} & \leqslant&   e^{C  \int_0^t  \|  v^0 \|_{Lip  (\R^3)} } \ \lbrack W_0 \rbrack^{-1}_{L^\infty  (\R^3)} .
  \end{eqnarray}
 Their divergence satisfy the equations 
\begin{eqnarray*}
\label{imitatingdiv}
 \partial_t \dive w^i + v^0  \cdot \nabla  \dive w^i  = \dive w^i    \cdot \nabla v^0 
\end{eqnarray*}
 and since they are divergence free at $t=0$ the vector fields $w^i$ remain divergence free when time proceeds.
 
Let us denote for any circular permutation of $(i,j,k)$ of $(1,2,3)$ the vector
 $u^i := w^j \wedge w^k$.  %
It comes out that the vector $n$ satisfies
\begin{eqnarray}
 \label{normalrecovered}
n(t,\mathcal{X}^0 (t,x))  =  \sum_{1   \leqslant  i  \leqslant 3 }
 \ \frac{\partial_i   \phi_0 (x) }{ |  \nabla_x      \phi_0  (x) | } u^i (t,X^0 (t,x)) .
  \end{eqnarray}
  In particular  the vectorfields $w^i (t)$ remain  conormal to the moving boundary  $\partial \mathcal{O}_{+} (t)$ when time proceeds:  $w^i . n = 0$ for any $1  \leqslant  i \leqslant 5$ on $\partial \mathcal{O}_{+} (t)$.

The conormal vectorfields satisfy the transport-stretching equation  (\ref{taneqbis}) so that we should expect that their smoothness persists as long as we have good estimates of the velocity.
In particular we need Lipschitz estimate as well as conormal and piecewise estimates of the velocity.
Next section is devoted to the way to deduce such estimates from the ones on the vorticity.

 \subsection{Static estimates}
   \label{secstaticesti}
 
Let us start with the  Lipschitz estimate of the velocity.
As the Marcinkiewicz-Calderon-Zygmund theorem fails to give a  Lipschitz estimate 
 of the velocity simply by $\|   \omega^0  \|_{L^\infty  \cap L^2}$,
 we need to consider a larger norm in the r.h.s. 
 But using normal derivatives (i.e. along $n$) would lead to estimates non-uniform with respect to the understood regularization parameter (the vorticity is expected to be only bounded not discontinuous across the patch boundary).
 The  argument  -dating back from Chemin for the $2$d case (see \cite{chemin}, \cite{che1},   \cite{che2},  \cite{che3}, \cite{che4})- consists to use the ellipticity of the 
 $\dive - \rot$ system to establish a  Lipschitz estimate for the velocity field with respect to the $L^\infty$ norm and to conormal derivatives only of the vorticity .
 Moreover one can manage to involve the conormal derivatives only through a  log, that is  widely known helpful in a subsequent Gronwall lemma since an inequality of the form $y'  \leqslant y \log y$ does not lead to a blow-up.
  In the $3$d case Gamblin and  Saint-Raymond  (cf. \cite{gamblin} Lemma $3.5$) also succeed to use the ellipticity  to get rid of the derivatives normal to the boundary proving:
  \begin{prop}
 There holds 
   \begin{eqnarray}
   \label{staticesti}
   \|   v^0 \|_{\text{Lip}(\R^3)} &\lesssim& 
  \|   \omega^0   \|_{L^2 \cap L^\infty  (\R^3 )} 
  + L(  \|  \omega^0  \|_{L^\infty (\R^3 )}  , \|   \omega^0  \|_{ C^{0,r}_ \text{co} }),
  \end{eqnarray}
  where we define  for $a$ and $b$ strictly positive
  $L$ is the function defined in (\ref{log}) and 
  the conormal $C^{0,r}_ \text{co}$  norm denotes:
  \begin{eqnarray*}
  \|  u  \|_{C^{0,r}_ \text{co} } :=   \|   u   \|_{ L^\infty  (\R^3 )} 
  +   \lbrack W  \rbrack^{-1}_{L^\infty  (\R^3)}
  \sum_{1   \leqslant i  \leqslant 5}  (    \|  w^i   \|_{0,r}    \|  u \|_{L^\infty (\R^3 )}   + \|  w^i \dagger  u  \|_{-1,r}) .
   \end{eqnarray*}
   \end{prop}  
  
   \begin{proof} [Scheme of proof]
 Actually since  $\|  v^0  \|_{L^\infty }$ can easily be estimated by 
 $$ \|  v^0  \|_{L^\infty } \lesssim  \|   \omega^0  \|_{L^\infty  \cap L^2}  $$
 (thanks to an appropriated splitting into near and far field in the integral formula of the Biot-Savart law)
the delicate problem is to estimate 
 the
 $\| \partial_i v^0  \|_{L^\infty }$.
 Let us continue to reduce the problem by
introducing  a function  $\chi$  in $C^\infty_0 (\R^3 ) $ positive and equal to $1$ near $0$ 
 and  $\chi (D)$ the highly smoothing corresponding Fourier multiplier.
Let $ \lambda (\zeta) := (\chi (\zeta) + |  \zeta  |^{2})^\frac{1}{2}  $ and $\Lambda:=  \lambda (D)$ the corresponding elliptic Fourier multiplier.
Since the low frequencies part can simply be estimated by 
\begin{eqnarray}
\label{lf}
 \|   \partial_j   \chi (D)  \Lambda^{ -2}   v^0 \|_{-1,r} &\lesssim & \|  v^0  \|_{L^\infty }
 \end{eqnarray}
 we are left to estimate the high frequencies part
 \begin{eqnarray}
  \label{left1}
 \| \partial_j  (1 -  \chi (D)  \Lambda^{ -2} ) v^0  \|_{ L^\infty} 
  \leqslant   \sum_{1   \leqslant k  \leqslant 3}  \|  \Lambda^{ -2} \partial_j  \partial_k  \omega^0  \|_{ L^\infty} .
\end{eqnarray}
  By  quadratic forms technics  Gamblin and  Saint-Raymond  (cf. \cite{gamblin} Lemma $3.5$ ) succeed to write
 \begin{eqnarray}
  \label{yael0}
  \partial_j  \partial_k =  \Delta a_{j,k} .  +  \sum_{l,i,p}  \partial_l \partial_p    b_{j,k}^{l,i}   w^i_p  ,
  \end{eqnarray}
  where the  functions $a_{j,k}$ and  $b_{j,k}^{l,i} $ are build by partition of unity from local expressions of the form
    \begin{eqnarray}
  \label{yael1}
  a_{j,k}  &=&   \frac{(w^m \times w^n )_{j}  (w^m \times w^n )_{k} }{  | w^m \times w^n  |^2 }  ,
\\   \label{yael2}
 b_{j,k}^{l,i} &=& \frac{ P_{j,k}^{l,i} (w^m , w^n ) }  {  | w^m \times w^n  |^2 }  ,
  \end{eqnarray}
where  the  $P_{j,k}^{l,i} (w^m , w^n ) $  are homogeneous polynomials of $w^m$, $w^n$ of degree $7$ and $w^m \times w^n  $ doesn't vanish, such that
   \begin{eqnarray*}
   \|   a_{j,k}   \|_{L^\infty }  &\leqslant& 1 ,
\\   \|  b_{j,k}^{l,i}  \|_{C^{0,r}  (\R^3) }  &\leqslant& C ( \|  \lbrack W (t) \rbrack^{-1}_{L^\infty  (\R^3)} .
\sum_{1   \leqslant i  \leqslant 5}    \|  w^i   \|_{0,r}  )^{19} . 
  \end{eqnarray*}
  
  We then rewrite  the function $\Lambda^{ -2} \partial_j  \partial_k  \omega^0 $ in
  (\ref{left1}) as 
    \begin{eqnarray}
      \label{statique}
  \Lambda^{ -2} \partial_j  \partial_k  \omega^0 = f_1 + f_2   \text{ where }   f_1 :=  (1 -  \chi (D)  \Lambda^{ -2} ) a_{j,k}  \omega^0 \text{ and } 
    f_2 :=  \Lambda^{ -2} ( \partial_j  \partial_k -   \Delta a_{j,k} ) \omega^0 .
   \end{eqnarray}
Since
  $\|  f_1 \|_{ L^\infty   (\R^3)}   \lesssim  \| \omega^0 \|_{L^\infty  (\R^3)} 
 $   we are left with the estimate of
 $   \|  f_2 \|_{L^\infty  (\R^3)} $.    
    We then apply the logarithmic interpolation inequality:
 \begin{eqnarray}
  \label{static1}
\|  f_2 \|_{ L^\infty   (\R^3)} &\lesssim&  L(  \|   f_2 \|_{B^0_{\infty, \infty} (\R^3 )}  , \|  f_2   \|_{ 0,r} ) .
 \end{eqnarray}
and observe that
  \begin{eqnarray*}
  \|  f_2   \|_{B^0_{\infty, \infty} (\R^3 )}  &\lesssim&   \|  \omega^0   \|_{ L^\infty   (\R^3)} ,
  \\   \|  f _2  \|_{ 0,r}   &\lesssim&   ( \|  \lbrack W (t) \rbrack^{-1}_{L^\infty  (\R^3)} 
\sum_{1   \leqslant i  \leqslant 5}    \|  w^i   \|_{0,r}  )^{19} \|   \omega^0  \|_{ C^{0,r}_ \text{co} }.
    \end{eqnarray*}
 which allow to conclude to the estimate  (\ref{staticesti}).
  
  \end{proof}

In the same vein we have that conormal estimates on the vorticity imply conormal estimates on the velocity. More precisely there holds:

 \begin{eqnarray}
  \label{coco}
  \|  w^i \dagger   v^0   \|_{0,r}  &\lesssim&
\|w^i \dagger  \omega^0    \|_{-1,r} + \|   v^0 \|_{\text{Lip}(\R^3)}  \|  w^i  \|_{0,r}  .
 \end{eqnarray}
 
   \begin{proof} [Scheme of proof]
To prove the estimate (\ref{coco}), we make use of Bony's paraproduct writing
\begin{eqnarray}
 \label{cocobis}
w^i \dagger   v^0    = (w^i - T_{w^i} ) \nabla   v^0  
+ T_{w^i}  \nabla   v^0 
 \end{eqnarray}
 and estimating the first term by $ \|   v^0 \|_{\text{Lip}(\R^3)}  \|  w^i  \|_{0,r}$
 thanks to the paraproduct rule (\ref{parap2}).
Then to estimate the second term 
we split once again the velocity $ v^0$ into 
low and high frequencies parts 
so that we get on one hand the term $T_{w^i}  \dagger  \chi (D)  \Lambda^{ -2}  v^0 $, which is easily estimated by $\|   v^0 \|_{\text{Lip}(\R^3)}  \|  w^i  \|_{0,r}$,
and on the other hand  a combination, with coefficients valued in $ \{ -1,0,1 \}$, of the terms 
 $T_{w^i}     \Lambda^{ -2} \partial_k \nabla    \omega^0 $.
 To control these last terms we commute the paraproduct $T_{w^i}$ with the operator $  \Lambda^{ -2} \partial_k \nabla  $:
  \begin{eqnarray}
  \label{coco2}
 T_{w^i}   \Lambda^{ -2} \partial_k \nabla    \omega^0  = 
 \lbrack T_{w^i} ,  \Lambda^{ -2} \partial_k \nabla \rbrack    \omega^0
+   \Lambda^{ -2} \partial_k    w^i   \dagger  \omega^0
+   \Lambda^{ -2} \partial_k  \nabla   (T_{w^i}- w^i )    \omega^0 .
 \end{eqnarray}
Now the first term in the r.h.s. of (\ref{coco2}) can be estimated by 
$\|   v^0 \|_{\text{Lip}(\R^3)}  \|  w^i  \|_{0,r}$
 thanks to the commutator estimate (\ref{commu}) with $m=\lambda = 0$.
The second one is estimated simply by $\|w^i \dagger  \omega^0    \|_{-1,r}$ and the third one by  $\|   v^0 \|_{\text{Lip}(\R^3)}  \|  w^i  \|_{0,r}$  thanks to the paraproduct rule (\ref{parap2}).
 \end{proof}

 \subsection{$C^{0,r}$ Dynamic estimates}
 \label{secdynaesti}
 
 With the previous estimates on the velocity in hand we can now prove some persistence of smoothness 
 of the vorticity.
Actually Gamblin and  Saint-Raymond  (cf. \cite{gamblin} Proposition $4.1$)
prove   the estimate
\begin{eqnarray}
 \label{dynaesti}
 \|  w^i  (t)   \|_{C^{0,r}}
  + \frac{ \| (w^i  \dagger \omega^0 )(t) \  \|_{C^{-1,r}}}{ \|   \omega^0  (t)   \|_{L^\infty } } 
  &\leqslant& e^{C V(t)} \
 \big(  \|  w^i_0   \|_{C^{0,r}}
  + \frac{ \| w^i_0  \dagger \omega_0  \  \|_{C^{-1,r}}}{ \|   \omega_0  \|_{L^\infty } }
   \big) ,
  \end{eqnarray}
  where $V(t)$ denotes $V(t) := \int_0^t  \|  v^0 \|_{Lip } $, 
  the constant $C$ depends only on $r$.
All the norms above are related to the whole space $\R^3$.

 \begin{proof} [Scheme of proof]

  We apply the transport estimate  (\ref{TEE})
  with  $s=0$ 
  to   the conormal vectorfields  $f=w^i$ which
 satisfy the transport-stretching equation  (\ref{taneqbis})-(\ref{taneq2bis})
 and estimate 
\begin{eqnarray}
\| \partial_t  w^i + v^0 \cdot \nabla w^i  \|_{ C^{0,r} (\R^3 ) }  = \| w^i \cdot \nabla v \|_{ C^{0,r} (\R^3 ) } 
 \end{eqnarray}
  by $(\ref{coco})$.
  This yields
\begin{eqnarray}
\label{555}
 \|   w^i  (t)   \|_{C^{0,r} (\R^3 )} 
   \leqslant
   \big(  \|  w^i_0  \|_{C^{0,r} (\R^3 )} 
   +  \int_0^t  (\|w^i \dagger  \omega^0    \|_{-1,r} + \|   v^0 \|_{\text{Lip}(\R^3)}  \|  w^i  \|_{0,r} )  (\tau)
   e^{-C V (\tau)  } d\tau   \big)
    e^{C  V (t)  } .
  \end{eqnarray}
  By using the equations (\ref{taneq})-(\ref{taneq2}) of the vorticity $\omega^0$ and the 
  equations (\ref{taneqbis})-(\ref{taneq2bis}) of the conormal vectorfields $w^i$
 we next observe that the field $w^i  \dagger \omega^0 $ involved in the inequality above satisfies the equation
\begin{align}
\label{taneq7} \partial_t  (w^i  \dagger \omega^0 )+ v^0  \cdot \nabla ( w^i  \dagger   \omega^0 )  = w^i  \dagger ( \omega^0  \cdot \nabla v^0 ),
\\ \label{taneq8}  (w^i  \dagger \omega^0 )|_{ t =0 }  =  w^i_0  \dagger \omega_0 .
\end{align}
We apply  the estimate  (\ref{TEE}) with $s=-1$ to the field $f = w^i  \dagger \omega^0 $ 
\begin{eqnarray}
\label{44}
 \| w^i  \dagger \omega^0  (t)   \|_{C^{-1,r} (\R^3 )} 
   \leqslant
   \big(  \|  w^i_0  \dagger \omega_0 \|_{C^{-1,r} (\R^3 )} 
   +  \int_0^t  \|  w^i  \dagger ( \omega^0  \cdot \nabla v^0 ) (\tau)     \|_{ C^{-1,r} (\R^3 ) } 
   e^{-C V (\tau)  } d\tau   \big)
    e^{C  V (t)  } .
  \end{eqnarray}
  Let us recall that the initial data $\| w^i_0  \dagger \omega_0 \|_{C^{-1,r} (\R^3 )}$ should be estimated by Proposition \ref{dudu}.
   Now to estimate the integral we first notice that 
\begin{eqnarray}
\label{costretch}
  w^i  \dagger ( \omega^0  \cdot \nabla v^0 ) := \dive (\omega^0 \otimes ( w^i  \dagger \omega^0 )) +  \dive ( \zeta \otimes v^0 )
    \end{eqnarray}
  with $   \zeta :=  \dive (  \omega^0 \otimes w^i - w^i  \otimes   \omega^0 )$.
  Then we make use of the paraproduct to commute the divergence operator in the first term of the r.h.s. of (\ref{costretch})
\begin{eqnarray*} 
 \dive (\omega^0 \otimes ( w^i  \dagger \omega^0 )) = 
\lbrack \partial_k , T_{  w^i  \dagger \omega_0 } \rbrack \omega^0_k
+   \partial_k \big( (w^i  \dagger \omega_0 - T_{  w^i  \dagger \omega_0 }) \omega^0_k \big),
 \end{eqnarray*}
 the sum over $k$ being understood. 
 This gives -thanks to the estimate (\ref{parap2}) and  (\ref{commu})
\begin{eqnarray*} 
 \| \dive (\omega^0 \otimes ( w^i  \dagger \omega^0 )) \|_{ C^{-1,r} (\R^3 ) }
  &\lesssim&  \| w^i \dagger v^0   \|_{0,r}  \| \omega^0 \|_{ L^\infty  } .
  \end{eqnarray*}
  Now to estimate the second term in the r.h.s. of (\ref{costretch}) we first notice that $\zeta$ is divergence free so that 
\begin{eqnarray*} 
 \| \dive ( \zeta \otimes v^0 )  \|_{ C^{-1,r} (\R^3 ) }
  &\lesssim& 
   \|   v^0   \|_{\text{Lip}(\R^3)} \|  \zeta \|_{ C^{-1,r} (\R^3 ) } .
  \end{eqnarray*}
  Since the tame estimates    (\ref{parap1}) and  (\ref{parap2}) yield
\begin{eqnarray*} 
 \|  \zeta \|_{ C^{-1,r} (\R^3 ) }
 &\lesssim& \| w^i \dagger v^0   \|_{0,r}  
 +  \| \omega^0 \|_{ L^\infty  }  \|  w^i  \|_{0,r}   
    \end{eqnarray*}
we finally get -using one more time the estimate  by $(\ref{coco})$- the following 
estimate of the conormal derivative of the stretching term
\begin{eqnarray}
\label{111}
\|   w^i  \dagger (  \omega^0  \cdot \nabla v^0   )  \|_{ C^{-1,r} (\R^3 ) } 
 &\lesssim&  \|   v^0 \|_{\text{Lip}(\R^3)} 
( \| w^i \dagger  \omega^0   \|_{-1,r} + \|   \omega^0 \|_{L^\infty (\R^3)}  \|  w^i  \|_{0,r} ).
\end{eqnarray}
Now we observe that for any $0 \leqslant s \leqslant t\leqslant T$ the vorticity  $\omega^0$ obeys the following $L^\infty $ estimate:
\begin{eqnarray}
\label{333}
 e^{V(s) -  V (t)}  \leqslant \frac{ \|\omega (s) \|_{L^\infty (\R^3)} }{\|\omega (t) \|_{L^\infty (\R^3)}  }  \leqslant 
 e^{V(t) -  V (s)} .
\end{eqnarray}
Combining the estimates (\ref{44}), (\ref{111}) and (\ref{333}) leads to 
\begin{eqnarray}
\label{444}
\frac{ \| w^i  \dagger \omega^0  (t)   \|_{C^{-1,r} (\R^3 )}  }{ \|\omega (t) \|_{L^\infty (\R^3)}  }
   \leqslant
   \big(  \frac{\|  w^i_0  \dagger \omega_0 \|_{C^{-1,r} (\R^3 )} }{ \|   \omega_0   \|_{L^\infty } }
  \\ \nonumber +  \int_0^t  \|   v^0  (\tau)  \|_{\text{Lip}(\R^3)}  (   \| w^i  (\tau )   \|_{C^{0,r} (\R^3 )}    + \frac{    \| w^i  \dagger \omega^0   (\tau) \|_{C^{-1,r} (\R^3 )}  }{ \|\omega (\tau  ) \|_{L^\infty (\R^3)}  }  )
   e^{-C V (\tau)  } d\tau   \big)
    e^{C  V (t)  } .
  \end{eqnarray}
We add the estimates (\ref{444}) and  (\ref{555})
and conclude  by a Gronwall argument.
  \end{proof}

 \subsection{Gamblin and  Saint-Raymond's result }
 \label{222}
 
Putting together the estimates (\ref{stretching})-(\ref{maximal})-(\ref{staticesti})-(\ref{dynaesti}),  Proposition \ref{dudu} and thanks to the relation  (\ref{normalrecovered})
we conclude that
there exists $T>0$ 
 and a unique solution $v^0  \in L^\infty (0,T;Lip (\R^d))  $ to the Euler equations
 such that for each $t \in (0,T)$  the boundary  $\partial \mathcal{O}_{+} (t)$ is  $C^{1,r}$, and can be given  by the equation  $\partial \mathcal{O}_{+}(t) = \{ \phi^{0} (t,.)=0  \}$, 
where $ \phi^{0} \in L^\infty (0,T;   C^{1,r}  (\R^d)   ) $
  verifies  (\ref{eik1})-(\ref{eik2}).
   The transport properties yield   $ \mathcal{O}_{\pm} (t)= \{ \pm \phi^{0} (t,.) > 0  \}$ and that there exists $ \eta >0$ such that for 
$| \phi^{0}  | < \eta $, $t \leqslant T$, there holds $\nabla \phi^{0} \neq 0$.  
We deduce the continuity of the normal component of the vorticity  $( \omega^0  . n)(t,.)$ including across $\partial \mathcal{O}_{+}(t) = \{ \phi^{0} (t,.)=0  \}$ from the divergence free of the vorticity and its conormal regularity.

 \subsection{Iterated tangential regularity}
 \label{ita}

In this section we briefly explain how to prove the propagation of the higher order smoothness- that is the  $C^{s+1,r}$  smoothness- of the boundary $\partial \mathcal{O}_{+} (t)$
(the point $\ref{litt3}$ of Theorem \ref{litt}).
 This was done  by P. Zhang in the couple of papers  \cite{zhang}, \cite{zhang2} written in Chinese, by adapting the method used by Chemin in the $2$d case  \cite{cheinventiones}.
By local inversion there exists a parametrization of the initial boundary $\partial \mathcal{O}_{+, 0} $ given by 
a function 
\begin{eqnarray}
\label{param1}
  \mathbf{x}_0 (\tau_1 , \tau_2 )  \in C^{s+1,r}  (S^1 \times  S^1 ,\R^3).
  \end{eqnarray}
Then for  each $t \in (0,T)$  the boundary $\partial \mathcal{O}_{+} (t)$ is given by the function
\begin{eqnarray}
\label{param2}
  \mathbf{x} (t, \tau_1, \tau_2) :=   \mathcal{X} (t,   \mathbf{x}_0 ( \tau_1, \tau_2))    
    \end{eqnarray}
and our goal in this section is to show that this function $\mathbf{x}$ is in $L^\infty (0,T;   C^{s+1,r}  ( S^1 \times  S^1 ,\R^3  ))  $, what means that the boundary  $\partial \mathcal{O}_{+} (t)$ is $L^\infty (0,T;   C^{s+1,r})$.

Since for $k = 1,2$ 
the vector field  $\partial_{\tau_k } \, \mathbf{x}_0$ is tangent to $\partial \mathcal{O}_{+, 0} $ at the point $ \mathbf{x}_0 (\tau_1 , \tau_2)$ and since the set of the vector fields $\{w^1_0 ,w^2_0 ,w^3_0 \}$ is maximal on $\partial \mathcal{O}_{+, 0} $  there exists some  functions 
$$ a^k ( \tau_1 , \tau_2 ) :=  (a^k_1 (\tau_1 , \tau_2),a^k_2 (\tau_1 , \tau_2) ,a^k_3 (\tau_1 , \tau_2) )
 \in C^{s,r}  (S^1 \times  S^1 ,\R^3)    $$
 such that
 $$ \partial_{\tau_k } \mathbf{x}_0   ( \tau_1 , \tau_2 )  = \sum_{i=1}^3 a^k_i   ( \tau_1 , \tau_2 ) w^i_0  ( \mathbf{x}_0  ( \tau_1 , \tau_2 ))  .$$

As a consequence by  applying the derivatives $\partial_{\tau }^\alpha :=  \partial_{\tau_1 }^{\alpha_1} \partial_{\tau_2 }^{\alpha_2}$ to the function $\mathbf{x}$ in (\ref{param2}) for $\alpha  := (\alpha_1 , \alpha_2 )  \in  \N \times  \N$ with 
$|\alpha |:= \alpha_1 + \alpha_2  \leqslant s+1$ we get 
 $$ \partial_{\tau }^\alpha   \mathbf{x}  (t, \tau_1 , \tau_2 ) = 
  \sum_{l=1}^{|\alpha |}  \sum_{ \beta  \in  \{1,2,3 \}^l}
  a_ \beta  ( \tau_1 , \tau_2 ) \Big(  (w_0 \dagger )^\beta  \mathcal{X}  \Big) (t,   \mathbf{x}_0 (\tau_1 , \tau_2 )) 
 $$
with $\beta :=  (\beta_1 , ..., \beta_l )$,
$(w_0  \dagger )^\beta := (w^{\beta_1}_0  \dagger )...(w^{\beta_l}_0  \dagger )$
and where the functions $a_ \beta$ are in $C^{0,r}  ( S^1 \times  S^1 ,\R )$.

Going back to the definition of the $w^i$ in
(\ref{imitating}) we see that  
$$(w^i _0  \dagger   \mathcal{X}) (t,   \mathbf{x}_0 (\tau_1 , \tau_2 ))
= w^i ( t,   \mathbf{x} (t,\tau_1 , \tau_2 )) $$
and by iteration that the iterated derivatives  $\Big(  (w_0 \dagger )^\beta  \mathcal{X}  \Big) (t,   \mathbf{x}_0 (\tau_1 , \tau_2 )) $ can be transformed into
the iterated derivatives  $\Big(  (w \dagger )^{\beta'}  w^{\beta_l}  \Big) (t,   \mathbf{x} (t,\tau_1 , \tau_2 )) $, where we denote  $\beta' :=  (\beta_1 , ..., \beta_{l-1} )$.
We are thus led to look for $L^\infty (0,T;   C^{0,r})$ smoothness for 
the   $  (w \dagger )^{\beta'}  w^{i}  $.

Since each $w^i \dagger$ commute with the material derivative
$\partial_t  + v^0 \cdot \nabla $ the  $  (w \dagger )^{\beta'}  w^{i}  $ satisfy
\begin{align}
\label{taneq79} \partial_t  (w \dagger )^{\beta'}  w^{i} + v^0  \cdot \nabla  (w \dagger )^{\beta'}  w^{i}  =  (w \dagger )^{\beta'}      w^{i}   \nabla v^0 ,
\\ \label{taneq89}  (  (w \dagger )^{\beta'} w^{i} )|_{ t =0 }  =   (w_0  \dagger )^{\beta'} w^{i}_0 .
\end{align}

Applying  the transport estimate  (\ref{TE}) with $s=0$ to the equations  (\ref{taneq79})-(\ref{taneq89})
yields for any  $t \in (0,T)$
\begin{eqnarray}
\label{TEEE}
 \|  (w \dagger )^{\beta'}  w^{i}  (t)   \|_{0,r} 
   \leqslant
  \big(  \|  (w_0  \dagger )^{\beta'} w^{i}_0   \|_{0,r} 
   +  \int_0^t  \|  (w \dagger )^{\beta'}      w^{i}   \nabla v^0  (\tau)     \|_{ 0,r } 
   e^{-C V (\tau)  } d\tau   \big)
    e^{C  V (t)  } ,
  \end{eqnarray}
   where
$V (t) :=
 \int_0^t   \|  \nabla v^0    \|_{ L^{\infty} (\R^3  ) } ds $ is now under control.
 Moreover since their definition in section
\ref{dudulle} we know that the initial conormal vectorfields 
$w^{i}_0$ are in $C^{s,r}$. 
Since  $|\beta' |  \leqslant s$
the paraproduct rules yield that the initial data 
$(w_0  \dagger )^{\beta'} w^{i}_0$ is actually in $C^{0,r}$.
As a consequence the estimate (\ref{TEEE}) now simply reads
\begin{eqnarray}
\label{TEEEE}
 \|  (w \dagger )^{\beta'}  w^{i}  (t)   \|_{0,r} 
   \leqslant
  C \big( 1 +  \int_0^t  \|  (w \dagger )^{\beta'}      w^{i}   \nabla v^0  (\tau)     \|_{ 0,r } d\tau  \big).
  \end{eqnarray}

Now we can estimate the r.h.s. of (\ref{TEEEE}) by following the approach of Gamblin and Saint-Raymond mentioned in section \ref{secstaticesti} proving on one hand that that iterated conormal regularity for the vorticity implies  iterated  conormal regularity for the velocity and on the other hand that iterated  conormal regularity for the velocity is preserved when time proceeds. 
For the first step we refer to  \cite{zhang2} p$387$ which shows that for any time  the $ C^{0,r} $ norm of the  $ (w \dagger )^{\beta'}      w^{i}   \nabla v^0$ can be estimated by  the $C^{-1,r}$ norm of the iterated conormal derivatives  $  (w \dagger )^{\beta}   \omega^0 $ of the vorticity.
More precisely when proceeding by iteration on $|\beta' |  \leqslant s$
and understooding  the quantities already known from previous steps
we have
\begin{eqnarray*}
 \|  (w \dagger )^{\beta}  v^0     \|_{0,r } 
 \lesssim  \|  (w \dagger )^{\beta}    \omega^0       \|_{-1,r } 
 +  \|  (w \dagger )^{\beta'}  w^i    \|_{0,r } .
 \end{eqnarray*}
Now in order to prove the second step we notice that  applying the derivatives  $  (w \dagger )^{\beta}   $ to the equation (\ref{taneq}) and using again that the  $w^i \dagger$ commute with the material derivative
$\partial_t  + v^0 \cdot \nabla $ yields that  the $  (w \dagger )^{\beta}   \omega^0 $ satisfy 
\begin{align}
\label{taneq799} \partial_t  (w \dagger )^{\beta}   \omega^0 + v^0  \cdot \nabla  (w \dagger )^{\beta}    \omega^0  =  (w \dagger )^{\beta}  \omega^0  \nabla v^0 ,
\\ \label{taneq899}  (  (w \dagger )^{\beta}  \omega^0 )|_{ t =0 }  =  (w_0 \dagger )^{\beta}  \omega_0  
\end{align}
We use again the transport estimate  (\ref{TE}) but this time with $s=-1$, which 
yields for any  $t \in (0,T)$
\begin{eqnarray}
\label{TEEbis}
 \|  (w \dagger )^{\beta}   \omega^0   (t)   \|_{-1,r} 
   \leqslant
  \big(  \|    (w_0 \dagger )^{\beta}  \omega_0    \|_{-1,r} 
   +  \int_0^t  \|  (w \dagger )^{\beta}    \omega^0     \nabla v^0  (\tau)     \|_{-1,r } 
   e^{-C V (\tau)  } d\tau   \big) 
    e^{C  V (t)  } .
  \end{eqnarray}
   Let us recall that the estimate of the initial data is performed in Proposition \ref{dudu}.
   The estimate of the  r.h.s. of (\ref{taneq799}) is done in  \cite{zhang2} p$388$.
It can be seen as
an extension in the setting of iterated conormal derivatives of the 
estimate  (\ref{111}) of the conormal derivatives of the stretching term
and reads as follows
\begin{eqnarray*}
 \|  (w \dagger )^{\beta}    \omega^0     \nabla v^0     \|_{-1,r } 
 \lesssim  \|  (w \dagger )^{\beta}    \omega^0       \|_{-1,r } 
 +  \|  (w \dagger )^{\beta'}  w^i    \|_{0,r } .
 \end{eqnarray*}
Plugging this in the inequality  (\ref{TEEbis}) yields an inequality of the form
\begin{eqnarray*}
 \|  (w \dagger )^{\beta}   \omega^0   (t)   \|_{-1,r} 
   \leqslant
 C \big(  1
   +  \int_0^t   \|  (w \dagger )^{\beta}    \omega^0       \|_{-1,r } 
 +  \|  (w \dagger )^{\beta'}  w^i    \|_{0,r }     \big)
   \end{eqnarray*}
so that  Gronwall argument leads to
\begin{eqnarray*}
 \|  (w \dagger )^{\beta}   \omega^0    \|_{-1,r}  
 \lesssim 
\|  (w \dagger )^{\beta'}  w^i    \|_{0,r }   
 \end{eqnarray*}
Now the estimate
 (\ref{TEEEE}) reads 
\begin{eqnarray*}
 \|  (w \dagger )^{\beta'}  w^{i}  (t)   \|_{0,r} 
   \leqslant
  C \big( 1 +  \int_0^t  \|  (w \dagger )^{\beta'}  w^{i}    \|_{0,r}  (\tau) d\tau  \big),
  \end{eqnarray*}
summations over $\beta'$ and $i$ being understood.
Hence applying  a Gronwall argument to the previous
estimate  yields
that  the  $  (w \dagger )^{\beta'}  w^{i}  $ are in  $L^\infty (0,T;   C^{0,r})$.
As a consequence for any $\alpha $ with 
$|\alpha |  \leqslant s+1$ we get 
 $ \partial_{\tau }^\alpha   \mathbf{x}  $ is in  $L^\infty (0,T;   C^{0,r})$, which proves that the boundary  $\partial \mathcal{O}_{+} (t)$ is $L^\infty (0,T;   C^{s+1,r})$.

\subsection{Piecewise transport estimates}
\label{pte}
 
 In order to  prove  the propagation of piecewise regularity (in the next section)
 we will  use piecewise estimate for the transport equations of the form:
\begin{align}
\label{T1} \partial_t f + v^0 \cdot \nabla f &= g,
\\ \label{T2} f |_{t=0} &= f_0,
\end{align}
 where  $f_0$, $v^0$ and $g$ are assumed to be given, 
 the two last ones being time dependent.
 To obtain these piecewise transport estimates we will need the following result about extensions of functions preserving the
Lipschitz property:
\begin{theo}[McShane  \cite{mcshane}, Whitney \cite{whitney}]
\label{extlip}
Let $\Omega$ be a subset of $\R^n$.
There exists an extension operator that is a linear operator $ext$ such that
\begin{eqnarray*}
\text{ext } : \text{Lip }     (\Omega ) \rightarrow \text{Lip }  (\R^n )
 \end{eqnarray*}
is continuous and satisfies $(\text{ext }  u)|_{ \Omega  } = u$ for any $u$ in $\text{Lip }     (\Omega )$.
\end{theo}
Let us now claim our result of piecewise transport estimates:
\begin{prop}
\label{ptransport}
There holds the following piecewise estimate: for $k=-1$ or $k=0$ 
\begin{eqnarray}
\label{TE1}
 \| f  (t)   \|_{C^{k,r} (\mathcal{O}_{\pm} (t)   )} 
   &\leqslant& C(t)
   \big(  \| f_0   \|_{C^{k,r} (\mathcal{O}_{0,\pm}  )} 
   +  \int_0^t  \|  g(\tau)     \|_{ C^{k,r} ( \mathcal{O}_{\pm} (\tau ) ) } 
   e^{-C  V_{\pm} (\tau)} d\tau  \big)
    e^{C V_{\pm}  (t)},
  \end{eqnarray}
   where
$V_{\pm}  (t) :=
 \int_0^t   \|  v^0    \|_{\text{Lip }   ( \mathcal{O}_{\pm}(s)  ) } ds$
 and for any $k \geqslant 1$
\begin{eqnarray}
\label{TE2}
 \| f  (t)   \|_{C^{k,r} (\mathcal{O}_{\pm} (t)   )} 
   &\leqslant& C(t)
   \big(  \| f_0   \|_{C^{k,r} (\mathcal{O}_{0,\pm}  )} 
   +  \int_0^t  \|  g(\tau)     \|_{ C^{k,r} ( \mathcal{O}_{\pm} (\tau ) ) } 
   e^{-C  V_{k,r,\pm} (\tau)} d\tau  \big)
    e^{C V_{k,r,\pm} (t)},
  \end{eqnarray}
  where $V_{k,r,\pm} (t) :=
 \int_0^t   \| v^0    \|_{ C^{k,r} (\mathcal{O}_{\pm} (s) ) } ds$
where the constant $C(t)$ depend only on the sup for $s$ running over $(0,t)$ of  the Lipschitz norm of the domain $\mathcal{O}_{+} (s) $.
\end{prop}

\begin{proof}
   Actually we will prove  an estimate on the side $\mathcal{O}_{+}$  totally
 independent of the one on the side $\mathcal{O}_{-}$.
  To fix the idea let us consider the case of the estimate on $\mathcal{O}_{+}$.
    We make use of Theorem \ref{rychkov}  to get  from
  the restrictions  $f_0 |_{\mathcal{O}_{0,+}}  $ and $g|_{\mathcal{O}_{+} (t)}  $ of the initial data $f_0$ and of the source term $g$  some  extensions $f_{0,+}  $ and $g_{+}   $ defined in the  whole space $\R^3$, with -for any $k \geqslant -1$-
\begin{eqnarray}
\label{TE3}
 \| f_{0,+}   \|_{C^{k,r} (\R^3 )}   \leqslant C
\| f_0 |_{\mathcal{O}_{0,+}}  \|_{C^{k,r} (\mathcal{O}_{0,+} ) } ,
\quad
\| g_{+}   \|_{C^{k,r}(\R^3 )}   \leqslant C(t)
\| g  |_{\mathcal{O}_{+} (t)   )}  \|_{C^{k,r} (\mathcal{O}_{+} (t))} ,
\end{eqnarray}
where the constant $C$ and $C(t)$ depend only on the Lipschitz norm of the domains $\mathcal{O}_{0,+}$ and $\mathcal{O}_{+} (t) $.

To extend the velocity field $v^0$ we distinguish two cases.
When  $k=-1$ or $k=0$  we make
use of McShane-Whitney extension Theorem \ref{extlip} to get a Lipschitz extension  $v^0 _{+}  $  defined in the  whole space $\R^3$ of the restriction  $v^0 |_{\mathcal{O}_{+} (t)} $.
When $k \geqslant 1$ we use again Rychkov extension Theorem \ref{rychkov}  to get  from
  the restriction $v |_{\mathcal{O}_{+} (t)}  $ an extension $v^0 _{+}  $ defined in the  whole space $\R^3$, with  %
\begin{eqnarray}
 \label{TE4}
 \| v^0_{+}   \|_{C^{k,r} (\R^3 )}   \leqslant C (t)
\| v^0   |_{\mathcal{O}_{+} (t)   )}  \|_{C^{k,r} (\mathcal{O}_{+} (t) )} ,
\end{eqnarray}

Then we apply the estimates  (\ref{TEE}) and (\ref{TE}) to the transport equation
\begin{align}
\label{T1+} \partial_t f_{+} + v^0_{+} \cdot \nabla f_{+} &= g_{+} ,
\\ \label{T2+} f_{+} |_{t=0} &= f_{0,+} .
\end{align}
This yields for $k=-1$ or $k=0$ :
\begin{eqnarray}
\label{TEEP}
 \| f_{+}  (t)   \|_{C^{k,r} (\R^3 )} 
   \leqslant
   \big(  \| f_{0,+}   \|_{C^{k,r} (\R^3 )} 
   +  \int_0^t  \| g_{+}   (\tau)     \|_{ C^{k,r} (\R^3 ) } 
   e^{-C  \tilde{V}_{+} (\tau)  } d\tau   \big)
    e^{C   \tilde{V}_{+} (t)  } ,
  \end{eqnarray}
   where
$ \tilde{V}_{+} (t) := \int_0^t   \|  v^0_{+}    \|_{ \text{Lip }  (\R^3  ) } ds$ and for any $k \geqslant 1$ 
\begin{eqnarray}
\label{TEP}
 \| f_+ (t)   \|_{C^{k,r} (\R^3 )} 
   \leqslant
   \big(  \| f_{0,+}   \|_{C^{k,r} (\R^3 )} 
   +  \int_0^t  \|  g_{+}  (\tau)     \|_{ C^{s,r} (\R^3 ) } 
   e^{-C  \tilde{V}_{k,r,+} (\tau)  } d\tau   \big)
    e^{C   \tilde{V}_{k,r,+} (t)  },
  \end{eqnarray}
  where
$\tilde{V}_{k,r,+} (t) :=
 \int_0^t   \|   v^0_+  \|_{ C^{k,r} (\R^3  ) } ds$.

Next we observe that since the domain $\mathcal{O}_{+} (t)$ is the transported by the flow of the initial domain  $\mathcal{O}_{+,0} $,  the transport equation  (\ref{T1})-(\ref{T2}) is well-posed in $  \cup_{t \in (0,T)} \{t\} \times \mathcal{O}_{+} (t)$ without any boundary condition. 
 By restriction of  the transport equation  (\ref{T1+})-(\ref{T2+}) the function $ f_{+} |_{\mathcal{O}_{+} (t)}$
also satisfies the  transport equation  (\ref{T1})-(\ref{T2})  in $\mathcal{O}_{+,0} $.
Now 
we can conclude that
$ f_{+} |_{\mathcal{O}_{+} (t)} = f|_{\mathcal{O}_{+} (t)} $ by uniqueness 
(let say in $L^\infty$ since it is sufficient for our purpose. 
Of course it is possible to prove uniqueness  in a much wider context, and without use of Fourier analysis, see the work of L. Ambrosio).
Making use of the estimates (\ref{TE3})-(\ref{TE4})
and doing the same thing on the other side $\mathcal{O}_{-}$ we get 
the desired estimates.
\end{proof}

\subsection{Propagation of piecewise regularity}
 
 In this section we investigate the propagation of piecewise regularity for the Euler equations.
 Indeed we  finish to prove 
 the points \ref{litt2}, \ref{litt4} and \ref{litt5} of Theorem \ref{litt}.
 For the commodity of the reader we recall here the decomposition (\ref{statique}):
 \begin{eqnarray*}
  \Lambda^{ -2} \partial_j  \partial_k  \omega^0 = f_1 + f_2   \text{ where }   f_1 :=  (1 -  \chi (D)  \Lambda^{ -2} ) a_{j,k}  \omega^0 \text{ and } 
    f_2 :=  \Lambda^{ -2}    \sum_{l,i,m}  \partial_l \partial_p   b_{j,k}^{l,i}   w^i_p \omega^0 .
   \end{eqnarray*}
 Let us denote by  $\Omega$  the collection of the $w^i$ and of the coefficients $a_{j,k}$ and $b_{j,k}^{l,i}$ which appear in the decomposition above.
 For any $m  \in  \N^*$  we introduce the set $Op_m$ (respectively the set $\tilde{Op}_m$) the collection of the operators $T$ (resp. $ \tilde{T}$)  of the form
  \begin{eqnarray*}
Tf := \Omega_0 \partial_{i_1} \Omega_1 \partial_{i_2} ... \partial_{i_m} \Omega_m f
\  (\text{resp.} \  \tilde{T} f := (w^{i_1}   \dagger) \Omega_1  (w^{i_2}   \dagger) \Omega_2 ... (w^{i_m}   \dagger) \Omega_m f )
   \end{eqnarray*}
Our strategy is to prove by iteration for $0  \leqslant  k  \leqslant s$ that
 \begin{eqnarray}
  \label{1053} 
  \text{For } 1  \leqslant i  \leqslant 5,  \quad
  w^i  \in L^\infty (0,T; C^{k,r}  ( \mathcal{O}_{\pm}  (t) ))
     \text{  and  }
     w^i \dagger \omega^0  \in  L^\infty (0,T; C^{k-1,r}  ( \mathcal{O}_{\pm}  (t) )), 
 \\   \label{1054}   v^0  \in   L^\infty (0,T; C^{k+1,r}  ( \mathcal{O}_{\pm}  (t) )),
  \\ \label{1055} 
    \text{For } | \alpha | \leqslant k,  \quad 
  \partial^ \alpha   \nabla v^0 = 
 S f 
   +  \sum_{ m \leqslant k } T_m \omega^0 
   +  \Lambda^{ -2} \partial_j  \sum_{ m \leqslant k +1 }    \tilde{T}_m    \omega^0 ,
   \end{eqnarray}
where $S$ is infinitely smoothing, $T_m$  (respectively  $\tilde{T}_m$) is in  $Op_m$ (resp. $\tilde{Op}_m$).
 Before to begin the proof by iteration let us explain briefly
the choice of such a formulation. 
Of course  the estimate (\ref{1054}) implies that  
$ \omega^0 $ is in
$L^\infty (0,T; C^{k,r}  ( \mathcal{O}_{\pm}  (t) ))$. 
Still to prove  the estimate (\ref{1054}) we will use transport features of the vorticity 
and it will turn out that it is useful to get some extra static estimates
 to deduce piecewise estimates on $ \nabla v^0$
 from piecewise estimates on $ \omega^0$.
For an operator the property to act continuously between some spaces of piecewise regularities is usually referred as a transmission property
 cf. Boutet de Monvel  and followers \cite{bdm1}, \cite{bdm2},  \cite{gh}, \cite{rempel}.
In particular the references above prove that 
 the operator $T: \omega \mapsto  \nabla v$ 
  satisfies the transmission property  in the case of smooth boundaries. 
 In the present case the smoothness of the boundary is limited 
 and  we need to estimate carefully how the smoothness of the boundary is involved in the constant of continuity of  the operator $T$ above in $C^{0,r}( \mathcal{O}_{\pm} (t) )$.
In particular as a byproduct  what follows will extend the previous results 
giving  that for an open subset $\mathcal{O}_+$ of class $C^{s+1,r}$ the operator $T$ is bounded from $C^{s,r}  ( \mathcal{O}_{\pm} )$ into itself. 
Actually the analysis would be complicated here by the  
 time-dependent setting. 
 The key point is the decomposition (\ref{1055}) which basically allows to write
 any derivatives of the velocity as the sum of a smooth term (the low frequencies part), 
 a local term involving derivatives of the vorticity, what allows direct piecewise estimate and a term which despite nonlocal  involves (iterated) conormal derivatives (up to commutators) which better behave (than normal derivatives) when estimated in the whole space.

Let us now start the proof by iteration of  (\ref{1053})-(\ref{1054})-(\ref{1055}).
 Regarding the step $k= 0$ we  already know that the assertions concerning 
  the $w^i$ and the $w^i \dagger \omega^0 $ 
  -the estimates (\ref{1053})- are  satisfied 
from the estimates  (\ref{dynaesti}) (which actually even give  that 
 the $w^i$ and the $w^i \dagger \omega^0 $ are respectively in $L^\infty (0,T; C^{0,r}  (\R^3 )  )$ and  $L^\infty (0,T; C^{-1,r}  (  \R^3 ))$) whereas the  decomposition (\ref{1055}) reduces to the  decomposition (\ref{statique}) proved in section \ref{secstaticesti} and recalled above. As a consequence to conclude the step $k=0$ it only remains to prove the estimate (\ref{1054}).
 As we have said above we are going to use transport features of the vorticity,  applying the  piecewise transport estimate  (\ref{TE1}) in $C^{0,r}( \mathcal{O}_{\pm} (t)  )$  to the vorticity
$f= \omega^0$  with  the stretching term 
$g=\omega^0 .\nabla v^0$ as right hand side (we recall that the vorticity satisfies the equation (\ref{taneq})).
 Since the functions $V_{\pm} $ are under control by the previous sections we get an estimate of the following form:
\begin{align}
 \label{2055}
  \| \omega^0  \|_{C^{0,r}( \mathcal{O}_{\pm} (t)  )} \lesssim 
1 +  \int_0^t   \|  \omega^0 .\nabla v^0    \|_{C^{0,r}( \mathcal{O}_{\pm} (\tau)  )} d\tau.
 \end{align}
Now let us use  the  decomposition (\ref{statique}) to get some extra static estimates
 to deduce piecewise estimates on $ \nabla v^0$
 from piecewise estimates on $ \omega^0$.
 From the low/high frequencies splitting of section \ref{secstaticesti} estimates (\ref{lf})-(\ref{left1})
 the question reduces to prove  that the operators
 $ \Lambda^{-2}   \, \partial_j   \, \partial_k $
      act continuously in $C^{0,r}( \mathcal{O}_{\pm}) (t) $. 
Indeed -since the operator $\chi (D)  \Lambda^{ -2} $ is smoothing- we get
 \begin{eqnarray*}
\|  f_1 \|_{C^{0,r}( \mathcal{O}_{\pm} (t)  )} \lesssim \|  a_{j,k} \|_{C^{0,r}( \mathcal{O}_{\pm} (t)  )} .  \|  \omega^0  \|_{C^{0,r}( \mathcal{O}_{\pm} (t)  )} .
   \end{eqnarray*}
   Now we use the local representation of the coefficients  $a_{j,k}$ given in (\ref{yael1}) to estimate the norms $ \|  a_{j,k} \|_{C^{0,r} } $ by means of 
 the norms  $ \|  w^i \|_{C^{0,r} } $ and 
 $  \lbrack W \rbrack$ which have already  been estimated
 (cf.  section \ref{secdynaesti} and estimate  (\ref{maximal})).
  In the same way than for the $a_{j,k}$ that is
 by using  the local representation of the coefficients  $b_{j,k}^{l,i}$ given in (\ref{yael2}) we can estimate the norms $ \|  b_{j,k}^{l,i} \|_{C^{0,r} }$ 
 by some quantities already   under control.
 As a consequence we simply  estimate  $\|  f_2 \|_{C^{0,r}( \mathcal{O}_{\pm} (t)  )}$
 by $\|  f_2 \|_{C^{0,r}( \R^3  )}$ which
 reduces to the ones of
$\|   w^i   \dagger  \omega^0  \|_{C^{-1,r}( \R^3  )}$
 which was already controlled in  section \ref{secdynaesti} Estimate  (\ref{dynaesti}).
Finally we get 
\begin{eqnarray*}
      \| \nabla v^0  \|_{C^{0,r}( \mathcal{O}_{\pm} (t)  )} \lesssim 
  \| \omega^0   \|_{C^{0,r}( \mathcal{O}_{\pm} (t)  )} 
 \end{eqnarray*}
 and since we have the tame estimates
\begin{align*}
 \| \omega^0 .\nabla v^0  \|_{C^{0,r}( \mathcal{O}_{\pm} (t)  )} \lesssim 
  \| \omega^0   \|_{C^{0,r}( \mathcal{O}_{\pm} (t)  )} .  \| \nabla v^0  \|_{ L^{\infty} }
+  \| \omega^0   \|_{ L^{\infty}} .   \|   \nabla v^0   \|_{C^{0,r}( \mathcal{O}_{\pm} (t)  )} 
\end{align*}
 we  infer the following  piecewise estimate of the stretching term:
\begin{eqnarray*}
      \| \omega^0 .\nabla v^0  \|_{C^{0,r}( \mathcal{O}_{\pm} (t)  )} \lesssim 
  \| \omega^0   \|_{C^{0,r}( \mathcal{O}_{\pm} (t)  )} 
 \end{eqnarray*}
   then by applying a  Gronwall argument to (\ref{2055})
 we  conclude that  the  vorticity 
$$  \omega^0  \in L^\infty (0,T; C^{0,r}  ( \mathcal{O}_{\pm}  (t) )).$$
Using once more the static estimates yields  the estimate (\ref{1054}) and the proof of the step $k=0$ is complete.
 \\
\\  At this point we can use that  the normal component of the vorticity  $( \omega^0  . n)(t,.)$ including across $\partial \mathcal{O}_{+}(t) = \{ \phi^{0} (t,.)=0  \}$ is continuous (cf. section \ref{222}) to get 
the point $\ref{litt5}$ of Theorem \ref{litt}:
for each $t \in (0,T)$ the function $( \omega^0  . n)(t,.)$ is $C^{0,r}$ on $ \{ | \phi^{0} (t,.) | < \eta \} $. 
\\
\\  Let us now assume that 
  (\ref{1053})-(\ref{1054})-(\ref{1055}) hold
 for  $0  \leqslant  k  \leqslant s-1$.
We are going to prove that  (\ref{1053})-(\ref{1054})-(\ref{1055}) hold at the order $k+1$.
  To do so we first apply the piecewise transport estimate  (\ref{TE2}) 
 to the conormal vectorfields $f= w^i$ with as respective r.h.s. 
 $g= w^i .\nabla v^0$ (we recall that the conormal vectorfields $ w^i$ satisfy the equations
 (\ref{taneqbis})-(\ref{taneq2bis})). 
 Since the functions $V_{k+1,r,\pm} $ are under control by the previous step we get an estimate of the following form:
\begin{align}
 \label{2032}  \| w^i   \|_{C^{k+1,r}( \mathcal{O}_{\pm} (t)  )} \lesssim 
1 +  \int_0^t   \|   w^i .\nabla   v^0       \|_{C^{k+1,r}( \mathcal{O}_{\pm} (\tau)  )} d\tau.
 \end{align}
We consider the  extensions (to  the  whole space)
$w^i_{\pm} $ and $v^0_{\pm} $ given
by   Theorem \ref{rychkov}   from
  the restrictions  $w^i |_{\mathcal{O}_{\pm}}  $ and $v^0  |_{\mathcal{O}_{\pm} (t)}  $
  on each side of the patch of $w^i$ and $v^0$.
  We denote by $ \omega^0_{\pm}$ the curl of $v^0_{\pm} $.
  We have therefore 
\begin{eqnarray*}
 \|   w^i .\nabla   v^0       \|_{C^{k+1,r}( \mathcal{O}_{\pm} (t)  )} 
\lesssim 
 \|   w^i_{\pm} .\nabla   v^0_{\pm}       \|_{C^{k+1,r}( \R^3)} .
 \end{eqnarray*}
Now that we are dealing with functions in the full space we can use the paraproduct.
Indeed we can adapt the estimate  (\ref{coco}) into the following one (with only modifications of the indexes in the proof of (\ref{coco})):
 \begin{eqnarray}
\nonumber  \|  w^i_{\pm} \dagger   v^0_{\pm}   \|_{k+1,r}  &\lesssim&
\|w^i_{\pm} \dagger  \omega^0_{\pm}    \|_{k,r} + \|   v^0_{\pm} \|_{\text{Lip} (\R^3)}  \|  w^i_{\pm}  \|_{k+1,r}  .
 \end{eqnarray}
Since we do not know if the universal extension operator given by Theorem \ref{rychkov} acts continuously in $L^\infty$ we first use the embedding  (\ref{inclu}) before to go back to piecewise estimates through Besov-type spaces:
 \begin{eqnarray}
 \|  w^i_{\pm} \dagger   v^0_{\pm}   \|_{k+1,r}   \nonumber &\lesssim&
\|w^i_{\pm} \dagger  \omega^0_{\pm}    \|_{k,r} + \|   v^0_{\pm} \|_{1,r}  \|  w^i_{\pm}  \|_{k+1,r}  
\\  &\lesssim&
\|w^i \dagger  \omega^0  \|_{C^{k,r}( \mathcal{O}_{\pm} (t)  ) }
 + \|   v^0 \|_{ C^{1,r}( \mathcal{O}_{\pm} (t)  )}  \|  w^i  \|_{C^{k+1,r}( \mathcal{O}_{\pm} (t)  ) }
  \end{eqnarray}
 and since $\|   v^0 \|_{ C^{1,r}( \mathcal{O}_{\pm} (t)  )} $ is already controled 
the estimate (\ref{2032}) becomes
\begin{align}
 \label{2033}  \| w^i   \|_{C^{k+1,r}( \mathcal{O}_{\pm} (t)  )} \lesssim 
1 +  \int_0^t   \|   w^i   \|_{C^{k+1,r}( \mathcal{O}_{\pm} (\tau )  )} +
    \|  w^i \dagger  \omega^0       \|_{C^{k,r}( \mathcal{O}_{\pm} (\tau)  )}  d\tau.
 \end{align}
To estimate the conormal derivatives $w^i \dagger  \omega^0$ in  $C^{k,r}( \mathcal{O}_{\pm})$ we apply the transport estimate (\ref{TE2}) to the equations 
(\ref{taneq7})-(\ref{taneq8}). 
At initial time the  conormal derivatives $w^i_0 \dagger  \omega_0$ is in $C^{k,r}( \mathcal{O}_{\pm})$ since  $C^{s,r}( \mathcal{O}_{\pm})$ is an algebra, and the amplification factor involves 
only  $V_{k,r,\pm} $ so that we infer an estimate of the form:
 \begin{eqnarray}
 \label{1328}
  \|  w^i \dagger  \omega^0  \|_{C^{k,r}( \mathcal{O}_{\pm} (t))}
  \lesssim  1 +  \int_0^t   \|   w^i \dagger (\omega^0  \dagger \nabla v^0 )  \|_{C^{k,r}( \mathcal{O}_{\pm} (\tau))} .
 \end{eqnarray}
Now to estimate the integrals above we proceed as follows. We consider once again the extensions $w^i_{\pm} $ and $v^0_{\pm} $ so that
 \begin{eqnarray*}
\|   w^i \dagger(  \omega^0  \dagger \nabla   v^0 )  \|_{C^{k,r}( \mathcal{O}_{\pm}  )}
  \lesssim
  \|   w^i_{\pm}  \dagger(  \omega^0_{\pm}   \dagger \nabla   v^0_{\pm}  )  \|_{C^{k,r}  (\R^3)} .
\end{eqnarray*}
Now we recall the decomposition (\ref{costretch}) which for the extensions reads
\begin{eqnarray}
\label{pcostretch}
  w^i_{\pm}  \dagger ( \omega^0_{\pm}  \cdot \nabla v^0_{\pm} ) := \dive (\omega^0_{\pm} \otimes ( w^i_{\pm}  \dagger \omega^0_{\pm} )) +  \dive ( \zeta_{\pm} \otimes v^0_{\pm} )
    \end{eqnarray}
  with $   \zeta_{\pm} :=  \dive (  \omega^0_{\pm} \otimes w^i_{\pm} - w^i_{\pm}  \otimes   \omega^0_{\pm} )$.
 The first term of the r.h.s. of (\ref{pcostretch}) can be estimated proceeding in the same way than in section \ref{secdynaesti}. This yields
\begin{eqnarray*}
 \| \dive (\omega^0_{\pm} \otimes ( w^i_{\pm}  \dagger \omega^0_{\pm} ))   \|_{C^{k,r}  (\R^3)}
  \lesssim
  \|    w^i_{\pm}  \dagger v^0_{\pm}    \|_{C^{k+1,r}  (\R^3)}  \|   \omega^0_{\pm}  \|_{L^\infty} .
  \end{eqnarray*}
Now for the second term we have the following tame estimates:
\begin{eqnarray*}
\| \dive (\zeta_{\pm} \otimes v^0_{\pm} )  \|_{C^{k,r}  (\R^3)} \lesssim
  \|    w^i_{\pm}  \dagger \omega^0_{\pm}    \|_{C^{k,r}  (\R^3)} +  \|  w^i_{\pm}  \|_{C^{k+1,r}  (\R^3)} .
\end{eqnarray*}
Thus using the continuity properties of the extension operator  the estimate (\ref{1328}) now becomes
 \begin{eqnarray}
 \label{1328bis}
  \|  w^i \dagger  \omega^0       \|_{C^{k,r}( \mathcal{O}_{\pm} (t)  )}
  \lesssim  1 +  \int_0^t   \|   w^i \dagger \omega^0   \|_{C^{k,r}( \mathcal{O}_{\pm} (\tau )  )} 
  +  \|   w^i  \|_{C^{k+1,r}( \mathcal{O}_{\pm} (\tau )  )}  .
 \end{eqnarray}
Combining with the estimate (\ref{2033}) and up to  a  Gronwall argument we conclude that  the conormal vector fields $w^i$ and the conormal derivatives of the vorticity $w^i \dagger \omega^0 $ are respectively   in $L^\infty (0,T; C^{k+1,r}  ( \mathcal{O}_{\pm}  (t) ))$  and in $L^\infty (0,T; C^{k,r}  ( \mathcal{O}_{\pm}  (t) ))$. Thus we have already proved that the estimate (\ref{1053}) holds at the order $k+1$.

We are now going to  show that 
$ \omega^0$ is in
$L^\infty (0,T; C^{k+1,r}  ( \mathcal{O}_{\pm}  (t) ))$.
In order to prove this
 we apply the estimate  (\ref{TE2}) to get
\begin{eqnarray}
 \label{1423}
  \|  \omega^0   \|_{C^{k+1,r}( \mathcal{O}_{\pm} (t)  )}
  \lesssim  1 +  \int_0^t   \|    \omega^0  \nabla v^0   \|_{C^{k+1,r}( \mathcal{O}_{\pm} (\tau )  )} d\tau  .
     \end{eqnarray}
Now we estimate the integral above thanks to piecewise tame estimates.
Let us mention a technical point:
despite the zero-th order estimate 
\begin{eqnarray}
 \|  ab  \|_{C^{0,r}( \mathcal{O}_{\pm} (t)  )} 
  \lesssim 
   \|  a  \|_{C^{0,r}( \mathcal{O}_{\pm} (t)  )}    \|  b  \|_{L^\infty } + 
   \|  a  \|_{L^\infty }  \|  b  \|_{C^{0,r}( \mathcal{O}_{\pm} (t)  )}  
  \end{eqnarray}
is a straightforward consequence of Definition \ref{defi0}, 
we do not know if the higher order  estimates  ($k \in \N$)
\begin{eqnarray}
 \|  ab  \|_{C^{k,r}( \mathcal{O}_{\pm} (t)  )} 
  \lesssim 
   \|  a  \|_{C^{k,r}( \mathcal{O}_{\pm} (t)  )}    \|  b  \|_{L^\infty } + 
   \|  a  \|_{L^\infty }  \|  b  \|_{C^{k,r}( \mathcal{O}_{\pm} (t)  )}  
  \end{eqnarray}
are correct. One should think to use the extension Theorem \ref{rychkov}  to go back to the full plane case but we do not know if the universal extension operator given by Theorem \ref{rychkov} acts continuously in $L^\infty$.
Still we can deduce from Definition \ref{defi0} and Leibniz rule the estimates
\begin{eqnarray}
 \|  ab  \|_{C^{k,r}( \mathcal{O}_{\pm} (t)  )} 
  \lesssim 
   \|  a  \|_{C^{k,r}( \mathcal{O}_{\pm} (t)  )}    \|  b  \|_{ C^{k-1,r}( \mathcal{O}_{\pm} (t)  )   } + 
   \|  a  \|_{C^{k-1,r}( \mathcal{O}_{\pm} (t)  ) }  \|  b  \|_{C^{k,r}( \mathcal{O}_{\pm} (t)  )}  
  \end{eqnarray}
and these estimates would be sufficient here, since the lower derivatives would be always estimated in some previous steps of the iteration process.
Here we get -still understanding any controled quantities-
\begin{eqnarray}
\label{2138}
 \|    \omega^0  \nabla v^0   \|_{C^{k+1,r}( \mathcal{O}_{\pm} (t)  )}
 \lesssim 
 1 +  \|  \nabla v^0   \|_{C^{k+1,r}( \mathcal{O}_{\pm} (t)  )} + \|    \omega^0    \|_{C^{k+1,r}( \mathcal{O}_{\pm} (t)  )}.
 \end{eqnarray}
It is now time to stare at $ \nabla v^0  $. 
Derivating  the decomposition  (\ref{1055}) at order $k$
and eliminating the double derivatives
thanks to the identity (\ref{yael0})
we get the decomposition  (\ref{1055}) at order $k+1$.
Now let us remark that thanks to the local representation (\ref{yael1})-(\ref{yael2})
 we can estimate  the $C^{k+1,r}( \mathcal{O}_{\pm} (t)  )$ norm
of the coefficient $\Omega$
 by a power of the $C^{k+1,r}( \mathcal{O}_{\pm} (t)  )$ norms of the $ w^i$ which have been previously controlled.
 We also know from section \ref{ita} that the iterated conormal derivatives $  (w  \dagger)^\beta \omega^0$ are in 
 $L^\infty (0,T; C^{-1,r}( \mathcal{O}_{\pm} (t)  ))$.
   As a  consequence we get
\begin{eqnarray}
\label{2007}
\|  \nabla v^0   \|_{C^{k+1,r}( \mathcal{O}_{\pm} (t)  )}
 \lesssim 
 1 +   \|   \omega^0   \|_{C^{k+1,r}( \mathcal{O}_{\pm} (t)  )}  .
  \end{eqnarray}
Plugging this in the estimate (\ref{2138}) and then the resulting estimate in the estimate (\ref{1423}) yields
\begin{eqnarray*}
 \|  \omega^0   \|_{C^{k+1,r}( \mathcal{O}_{\pm} (t)  )}
  \lesssim  1 +  \int_0^t   \|     \omega^0    \|_{C^{k+1,r}( \mathcal{O}_{\pm} (\tau )  )} d\tau  .
     \end{eqnarray*}
so that 
$ \omega^0$ is  in
$L^\infty (0,T; C^{k+1,r}  ( \mathcal{O}_{\pm}  (t) ))$.
Using again  the previous  static estimate we infer that the estimate (\ref{1054}) holds at order $k+1$.
The iteration can be done and the point $\ref{litt2}$ of Theorem \ref{litt} is therefore proved.
\\
\\ Using  the relation (\ref{normalrecovered}) we next infer that
$n$ is in $L^\infty (0,T; C^{s,r}  ( \mathcal{O}_{\pm}  (t) ))$ which together with the tangential estimates  give that $ \phi^{0}$ is in
$L^\infty (0,T; C^{s+1,r}  ( \mathcal{O}_{\pm}  (t) ))$ and finish to prove 
the point $\ref{litt4}$ of Theorem \ref{litt}.

\subsection{Analyticity}
 
 In this section we say a few words about 
 the final statement regarding analyticity w.r.t. time (the point $\ref{litt6}$ of Theorem \ref{litt}).
 Actually the study of smoothness of the boundary w.r.t. time is already done by Chemin in its pioneering work
  \cite{cheinventiones}. 
  He  incorporates the material field $ D$ into its conormal fields and carries out iterated conormal derivations. 
  He follows his approach in  \cite{cheminsmoothness} and prove  the smoothness w.r.t. time 
   for classical solutions ($v_0$ in $C^{1,r}$) in any dimensions and for Yudovich solutions (with bounded vorticity) in two dimensions. 
    This result was extended into  analyticity w.r.t. time 
   by  P. Serfati \cite{serfati} (see also its doctoral thesis)
   in the case of classical solutions and  in the present case of  vortex patches, thanks to Lagrangian methods
   (by considering the Euler equations as an ODE for the flow).
    Let us mention papers \cite{gamblinana}, \cite{gamblinana2} which recover
    by a Eulerian approach the result by Serfati of analyticity w.r.t. time of classical solutions
    and prove Gevrey-$3$ smoothness in time of Yudovich  solutions.
    Finally we mention that such results also hold in the case of solid boundaries: 
    the paper  \cite{katoana}  of T. Kato
    proves the $C^\infty$ smoothness in time for classical solutions in a smooth bounded domain in any dimension, the paper  \cite{ogfstt}  extends Kato's result to analyticity and also proves the  analyticity of the motion of a body immersed into a perfect incompressible fluid  with $C^{2,r}$ initial velocity.


\section{Internal transition layer profiles}
\label{profile}

In this section we look for an expansion for the solutions of  the Navier-Stokes equations which describes as well as possible their behaviour with respect not only to the variables $t,x$ but also to the viscosity coefficient $\ep $.
The sections  \ref{fastscale} to \ref{looking}
 give the heuristic of the derivation of the profile problem -the equations (\ref{heu2r})-(\ref{heu21r})- mentionned in the introductory section \ref{viscous}.
 More precisely in section \ref{fastscale} 
we will identify the inner fast scale as $\frac{\phi^{0}(t,x) }{\sqrt{\ep t}}$.
This means that
 the initial discontinuity of the vorticity is smoothed out into a layer of size $\sqrt{\ep t}$ around the hypersurface $  \{  \phi^{0}  =0 \}$ where  the inviscid discontinuity occurs.
 In section \ref{amplitudes}  we  pay attention to the  expected  order of  amplitudes  of  the velocity and  pressure profiles. We will see that it is natural to associate to a vorticity expansion of the form:
\begin{eqnarray*}
 \omega^\ep (t,x) &\sim & \Omega  (t,x,\frac{\psi(t,x)}{\sqrt{\ep } } )
  \end{eqnarray*}
a velocity expansion of the form:
\begin{eqnarray*}
 v^\ep (t,x) &\sim & v^0  (t,x) + \sqrt{\ep t} \, \tilde V (t,x,\frac{\phi^{0}(t,x)}{\sqrt{\ep t} } ),
  \end{eqnarray*}
and a pressure expansion of the form:
 \begin{eqnarray*}
 p^\ep (t,x) &\sim &  p^{0}  (t,x) + \ep t \tilde{P} (t,x,\frac{\phi^{0}(t,x)}{\sqrt{\ep t} } ).
 \end{eqnarray*}
In section \ref{looking} we look for the profile equations. 
We choose there to deal with the velocity formulation of the  the Navier-Stokes equations.
In section \ref{wp4} we study the problem obtained when setting
 formally  $t=0$ in  the profile problem. 
In section \ref{wp} we  prove the existence and uniqueness of the layer profile in a $L^2$ setting, and we prove the smoothness properties of  this solution in section \ref{wp'}. 
The other properties of the profile are proved in section \ref{other} what achieves the proof of Theorem  \ref{propva} given in introduction.
n section  \ref{hoprofile} we write a complete  asymptotic expansion. 
In section \ref{bienprep} (respectively  in section \ref{ws}) we mention how the analysis is simplified in a setting of
 "well-prepared"   data (resp. in the setting of conormal singularities weaker than  vortex patches).
 In section \ref{ss} we  link the present study with previous works on stronger singularities.

\subsection{The inner fast scale}
\label{fastscale}

 The goal of  this section is to explain how 
we identify the inner fast scale as $\frac{\phi^{0}(t,x) }{\sqrt{\ep t}}$.

\subsubsection{A highly simplified model}
\label{baby}

To identify the inner fast scale we first look at  the $1$d scalar heat equation: $ \partial_t  \omega^\ep 
= \ep \partial_x^2   \omega^\ep$ which plays here the role of "baby model" for the NS equations. 
We prescribe as initial data a discontinuous vorticity: $\omega^\ep |_{t= 0} = 1_{\R_+}$.
In the inviscid case $\ep=0$ -which stands for (highly) simplified Euler equations-
 the solution is simply equal to the initial data
$ \omega^0 (t,.) :=   1_{\R_+} $ for any time, whereas for $\ep > 0$ and $ t > 0$
one can explicitly compute the solutions $\omega^\ep$ by convolution:
 \begin{align}
  \omega^\ep (t,x) :=  \Omega (\frac{x}{\sqrt{\ep t}} ) 
\text{ where }
  \Omega (X) := \frac{1}{\sqrt{\pi}} \int_{-\frac{X}{2} }^{ \infty} e^{-y^2 } dy  .
\end{align}
Hence the initial discontinuity of the vorticity is smoothed out into a layer of size $\sqrt{\ep t}$ where occurs -smoothly- the transition between the values $0$ and $1$.
It is useful to rewrite 
the $\omega^\ep$ as:
 \begin{align*}
  \omega^\ep (t,x) :=    \omega^0 (t,x) +  \tilde\Omega_\pm (\frac{x}{\sqrt{\ep t}} )  \quad \text{when }  \quad  \pm x > 0 , 
\end{align*}
where
\begin{align*}
  \tilde\Omega_\pm (X) := \frac{1}{\sqrt{\pi}} \int_{-\frac{X}{2} }^{\mp \infty} e^{-y^2 } dy   \quad \text{when }  \quad  \pm X > 0 .
\end{align*}
One then see  the  "viscous" solutions $\omega^\ep$ as the sum of the "inviscid" solution $ \omega^0$ plus a "double initial-(internal) boundary layer" $\tilde\Omega_\pm $ which satisfies the double-ODE:
\begin{align}
\label{baby1}
 \partial_X^2 \tilde\Omega_\pm   + \frac{X}{2}  \partial_X \tilde\Omega_\pm  = 0   \quad \text{when }  \quad  \pm X > 0 ,
\end{align}
match the continuity conditions of $ \omega^\ep$ and $ \partial_x  \omega^\ep$  at the internal boundary $x=X=0$ (for $t>0$):
\begin{eqnarray}
\label{baby2}
  \omega^0 |_{x=0^+ } +   \tilde\Omega_+ |_{X=0^+} = 1-1/2 =0 + 1/2 =  \omega^0 |_{x=0^- } +   \tilde\Omega_- |_{X=0^-} ,
\\ \label{baby3}
 \partial_X \tilde\Omega_+ |_{X=0^+} = \frac{1}{2 \sqrt{\pi}} =  \partial_X \tilde\Omega_- |_{X=0^-} 
 \end{eqnarray}
and vanish
\begin{eqnarray}
\label{baby4}
 \tilde\Omega_\pm (X)   \rightarrow 0   \quad \text{in the limits}  \quad  X \rightarrow \pm \infty.
  \end{eqnarray}
These last limits correspond both to the limits $t > 0$, $x   \rightarrow \pm \infty$ and the limits  $\pm x > 0$, $\ep t   \rightarrow 0^+$ (we recall that $X$ is the placeholder for $\frac{x}{\sqrt{\ep t}}$).
Conversely the two second order elliptic equations (\ref{baby1}) with the four "normal" boundary conditions (\ref{baby2})-(\ref{baby3})-(\ref{baby4})  (this last one contain two conditions) determine uniquely the profiles $\tilde\Omega_\pm $.

\subsubsection{Inner scale in the NS equations}
\label{baby2bis}

Of course the case of the NS equations is really much more complicated than the previous baby model. In particular the inviscid discontinuity moves:
Theorem \ref{litt} states that  the inviscid discontinuity occurs at the hypersurface $  \{  \phi^{0}(t,.)   =0 \}$ 
given by  the eikonal equations (\ref{eik1})-(\ref{eik2}) associated to the 
the particle derivative 
$ \partial_t   + v^0 \cdot\nabla_x $.
Therefore we are led to consider the  inner fast scale  $\frac{\phi^{0}(t,x)}{\sqrt{\ep t} }$ and we expect that in the case of vortex patches as initial data the  solutions  $\omega^\ep$ of NS can be described by an expansion of the form 
\begin{eqnarray}
\label{formu1}
 \omega^\ep (t,x) &\sim &  \omega^{0}  (t,x) +  \tilde\Omega  (t,x,\frac{\phi^{0}(t,x)}{\sqrt{\ep t} } ) ,
\end{eqnarray}
where
 $   \tilde\Omega (t,x,X)  $ denotes 
  a  local  perturbation so that 
\begin{eqnarray}
    \label{sharp1} \lim_{X\rightarrow \pm \infty}    \tilde\Omega   (t,x,X) &=& 0.
 \end{eqnarray}
Actually we will consider profiles even
rapidly decreasing at infinity.
\begin{rem}
  \label{experts}
 \rm
At this point experts in mathematical geometric optics should 
argue that  the fast scale $\frac{\phi^{0}(t,x)}{\sqrt{\ep } }$ should be  more intuitive  since applying the particle derivative 
$ \partial_t   + v^0 \cdot\nabla_x $ to a function $\Omega  (t,x,\frac{\phi^{0}(t,x)}{\sqrt{\ep } } )$ produces the singular term 
$$\frac{1}{\sqrt{\ep } } ( \partial_t   + v^0 \cdot\nabla_x ) \phi^{0} . (\partial_X \, \Omega) |_{X= \frac{\phi^{0}(t,x)}{\sqrt{\ep } }  }$$
 which fortunately vanishes thanks to the eikonal equation. 
 
 The point is that such a choice of inner scale does not lead to totally satisfactory results.
 Let us show that on our baby model.
 If one look for representation of the solutions $\omega^\ep$ of the form 
 \begin{align}
  \omega^\ep (t,x) :=    \omega^0 (t,x) +  {\bf\Omega}_\pm (t,\frac{x}{\sqrt{\ep }} )  \quad \text{when }  \quad  \pm x > 0 , 
\end{align}
one then see that  the  "double (internal) boundary layer" $ {\bf\Omega}_\pm $ has  this time to satisfy the parabolic equation:
\begin{align}
\label{baby1a}
\partial_t   {\bf\Omega}_\pm  =   \partial_X^2  {\bf\Omega}_\pm      \quad \text{when }  \quad  \pm X > 0 ,
\end{align}
with the same boundary conditions:
\begin{eqnarray}
\label{baby2a}
  {\bf\Omega}_+ |_{X=0^+} -   {\bf\Omega}_- |_{X=0^-} = -1 ,
\\ \label{baby3a}
 \partial_X  {\bf\Omega}_+ |_{X=0^+} -  \partial_X  {\bf\Omega}_- |_{X=0^-} = 0,
\\  \label{baby4a}
  {\bf\Omega}_\pm (X)   \rightarrow 0   \quad \text{in the limits}  \quad  X \rightarrow \pm \infty,
  \end{eqnarray}
Since we prescribe the same  initial data for $\omega^\ep$ than for  $\omega^0$,
 one has to prescribe zero initial data for the layers: 
 \begin{eqnarray}
  \label{para5a}   {\bf\Omega}_\pm |_{t=0  } &=& 0 ,
   \end{eqnarray}
so that the compatibility condition between the transmission condition (\ref{baby2a}) and the initial condition (\ref{para5a}) on the "corner" $\{t=X=0\}$ is not satisfied even at order zero, which ruins any hope for smoothness regarding the profiles $ {\bf\Omega}_\pm $.

 Now let us stress that  if  the "phase" $ \psi(t,x) := \frac{\phi^{0}(t,x)}{\sqrt{t } }$ chosen  in the  asymptotic expansions (\ref{formu1}) does not satisfy exactly the eikonal equation it however satisfies
\begin{eqnarray}
( \partial_t   + v^0 \cdot\nabla_x ) \psi = - \frac{1}{2t } \psi 
\end{eqnarray}
 so that  applying the particle derivative 
$ \partial_t   + v^0\cdot\nabla_x $ to a function 
\begin{eqnarray}
\label{sŽpa2}
 \omega^\ep (t,x) := \Omega  (t,x,\frac{\psi(t,x)}{\sqrt{\ep } } )
  \end{eqnarray}
  produces the  term 
\begin{eqnarray}
\label{sŽpa}
\frac{1}{\sqrt{\ep } } ( \partial_t   + v^0 \cdot\nabla_x )\psi .  (\partial_X \, \Omega)  |_{X= \frac{\phi^{0}(t,x)}{\sqrt{\ep t} }  }= - \frac{1}{2t } (X \partial_X \ \Omega)  |_{X= \frac{\phi^{0}(t,x)}{\sqrt{\ep t } }  } .
 \end{eqnarray}
 which is not singular (with respect to  $\ep $) anymore.
 The point is that in the present setting of a localized profile $\Omega$ 
 the derivative $X \partial_X $ in (\ref{sŽpa})
 does not cause any difficulty
 (note that such a term even appears in our baby model see  (\ref{baby1})), 
 in particular because the prefactor $\frac{1}{t }$
echoes the one in the term
\begin{eqnarray}
  \frac{|n(t,x)|^2}{t} \partial^2_X  \, \Omega  |_{X= \frac{\phi^{0}(t,x)}{\sqrt{\ep t} }  }
  \end{eqnarray}
 which is the larger one produced by applying the Laplace operator  $\ep  \triangle_x$ to the function
 $ \omega^\ep (t,x)$  in (\ref{sŽpa}).
 
 Of course this additional   derivative $X \partial_X $ is more problematic in the traditional context of periodic oscillations of geometric optics.
 
  \end{rem}


\subsection{Amplitudes}
\label{amplitudes}

We now pay attention to the  expected  order of  amplitudes  of  velocity and  pressure profiles.
 In the full plane the Biot-Savart law has Fourier symbol  $-\frac{ \xi }{ |\xi|^2} \wedge .$
It is a pseudo-local operator of order $-1$ so that we expect that the velocity $v^\ep$ given by the Navier-Stokes equations can be described by an asymptotic expansion of the form:
\begin{eqnarray}
\label{formu2}
 v^\ep (t,x) &\sim & v^0  (t,x) + \sqrt{\ep t} \tilde V (t,x,\frac{\phi^{0}(t,x)}{\sqrt{\ep t} } ),
  \end{eqnarray}
  where the profile $\tilde V (t,x,X) $ is also expected to satisfy
\begin{eqnarray}
  \label{sharp1vitesse} \lim_{X\rightarrow \pm \infty}    \tilde V   (t,x,X) &=& 0.
 \end{eqnarray}
  Arguably since  the Euler velocity $v^{0}$ is Lipschitz we expect that its viscous perturbations $v^\ep$ are uniformly Lipschitz, what gives support to the ansatz (\ref{formu2}).
  Plugging  (\ref{formu1}) and (\ref{formu2}) into  the  relations (\ref{curlep}), taking into account (\ref{V1}) and equalling the leading order terms
 leads to
\begin{eqnarray}
\label{rotrot}
n \wedge \partial_X \tilde V = \tilde \Omega . 
\end{eqnarray}
Hence the vorticity profile $\tilde \Omega  $ has to satisfy the orthogonality condition:
\begin{eqnarray}
  \label{ortho}
\tilde \Omega . n = 0 .
\end{eqnarray}
This condition is not a surprise: since $w^0 $ is divergence free $w^0 . n$ is continuous (cf. the point \ref{litt5} of Theorem \ref{litt})) so that no (large amplitude) layer is expected on the normal component of the vorticity. 

Now the pressure $p^\ep$ can be recovered from the velocity $ v^\ep$ by applying the operator divergence to the equation  (\ref{R1}) which yields the Laplace problem:
\begin{eqnarray}
 \label{pressure}
 \triangle_x\ p^\ep = -  \partial_i v^\ep_j .  \partial_j v^\ep_i  .
 \end{eqnarray}
If  the velocity $ v^\ep$ satisfies the expansion (\ref{formu2}), the
r.h.s. of (\ref{pressure}) should admit an expansion of the form:
\begin{eqnarray*}
\label{formu2p0}
 \triangle_x\ p^\ep  &\sim &  \triangle_x\ p^0
 +  \tilde F (t,x,\frac{\phi^{0}(t,x)}{\sqrt{\ep t} } ),
  \end{eqnarray*}
  where  the function  $\tilde F$ satisfies
\begin{eqnarray*}
  \label{sharp1pression0}
   \lim_{X\rightarrow \pm \infty}  \tilde{F}  (t,x,X) &=& 0.
 \end{eqnarray*}
  Since  the Laplacian is  of order $-2$  we are lead  to consider a perturbation of order $\ep t$ on the pressure:
 \begin{eqnarray}
 \label{formu2p}
  p^\ep (t,x) &\sim &  p^{0}  (t,x) + \ep t \tilde{P} (t,x,\frac{\phi^{0}(t,x)}{\sqrt{\ep t} } ) ,
 \end{eqnarray}
where -once again- the fast scale $\frac{\phi^{0}(t,x)}{\sqrt{\ep t} }$ is expected to be a local inner scale (the  Laplacian  operator is pseudo-local) so that
\begin{eqnarray}
  \label{sharp1pression} \lim_{X\rightarrow \pm \infty}  \tilde{P}  (t,x,X) &=& 0.
 \end{eqnarray}

\subsection{Looking for a profile problem}
\label{looking}

 Now that we have an intuition of the amplitude of the profiles, we look for the profile equations. 
 We choose here to deal with the velocity formulation of the NS equations, which is believed a more robust method (in view of future adaptation to the compressible case for instance).
 We proceed in several steps.
 In section \ref{looklook}   we  plug the ansatz into the velocity equation equalling the leading order terms. We then pay attention to the divergence free condition what leads to a crucial observation in section \ref{Transparency}.  In section \ref{rid}  we get rid of the pressure in  the velocity profile equation. 
 As the vectorfield $n$ may vanish, away the patch boundary, it is useful to modify the resulting equation in order to avoid a degeneracy of the order. 
 This will be done in section \ref{avoiding}.
  In section \ref{wp1}  we study the transmission conditions on the  inner interface $X=0$. We will use several times the formulas (for any $h$)
\begin{eqnarray}
 \label{10} 
\partial_t  \Bigl[  (\ep t)^\frac{j}{2}  \   h\bigl(t,x,\frac{\phi^{0}(t,x)}{\sqrt{ \ep t}}\bigr)\Bigr]
&=&  (\ep t)^\frac{j}{2}  \ \Bigl(   \partial_t h\bigl(t,x,\frac{\phi^{0}(t,x)}{\sqrt{ \ep t}}\bigr)
+ \frac{\partial_t  \phi^{0} }{\sqrt{ \ep t} }  \partial_X h\bigl(t,x,\frac{\phi^{0}(t,x)}{\sqrt{ \ep t}}\bigr) 
\\ \nonumber &&- \frac{1}{2t} (X \partial_X h ) \bigl(t,x,\frac{\phi^{0}(t,x)}{\sqrt{ \ep t}})
      +  \frac{j}{2t}   h\bigl(t,x,\frac{\phi^{0}(t,x)}{\sqrt{ \ep t}}\bigr)   \Bigl)                         ,
\\ \vspace{0,3cm} \label{11}
\nabla_x\Bigl[h\bigl(t,x,\frac{\phi^{0}(t,x)}{\sqrt{ \ep t}}\bigr)\Bigr]
&=& \nabla_xh\bigl(t,x,\frac{\phi^{0}(t,x)}{\sqrt{ \ep t}}\bigr)
+\frac {n(t,x)}{\sqrt{ \ep t}}\partial_Xh\bigl(t,x,\frac{\phi^{0}(t,x)}{\sqrt{ \ep t}}\bigr)
\end{eqnarray}
and
\begin{eqnarray}
\label{12}
 \ep  \triangle_x\Bigl[h\bigl(t,x,\frac{\phi^{0}(t,x)}{\sqrt{ \ep t}}\bigr)\Bigr]
&=&   \frac{|n(t,x)|^2}{t} \partial^2_Xh\bigl(t,x,\frac{\phi^{0}(t,x)}{\sqrt{ \ep t}}\bigr)\\ \nonumber
&&+  \frac{\sqrt{\ep t}}{t} (\triangle\phi^{0} \ \partial_Xh\bigl(t,x,\frac{\phi^{0}(t,x)}{\sqrt{ \ep t}}\bigr)
+ 2 n(t,x)\cdot\nabla_x\partial_Xh\bigl(t,x,\frac{\phi^{0}(t,x)}{\sqrt{ \ep t}}\bigr) )\\ \nonumber
&&+ \frac{\ep t}{t} \ \triangle_x \, h\bigl(t,x,\frac{\phi^{0}(t,x)}{\sqrt{ \ep t}}\bigr).
\end{eqnarray}

\subsubsection{Velocity equation}
\label{looklook}

We plug the ansatz (\ref{formu2}) and (\ref{formu2p})   into the equation  (\ref{R1}),  equalling the leading order terms, which are of order $\sqrt{\ep t}^{0}$:
\begin{equation}
  \partial_t v^0 + v^0 \cdot \nabla v^0 + \nabla p^0 +
(\partial_t  + v^0 \cdot\nabla_x ) \phi^{0} . \partial_X  \tilde V = 0.
\end{equation}
which is satisfied since the velocity $v^{0}$ satisfies the Euler equations
and $\phi^{0}$ satisfies the eikonal equation (\ref{eik1}).
 At the following order $\sqrt{\ep t}$ we get the equality
 \begin{equation}
  \label{heu}
\partial_t  \tilde V + v^0 \cdot\nabla_x  \tilde V +  \tilde V \cdot n  \partial_X  \tilde V  +  \tilde V \cdot\nabla_x v^0 +  \partial_X \tilde{P} n = \frac{1}{t}  ( |n|^2 \partial_X^2 \tilde V  +  \frac{X}{2}   \partial_X \tilde V - \frac{1}{2}  \tilde V  ).
\end{equation}

\subsubsection{Incompressible transparency}
\label{Transparency}

We now pay attention to the divergence free condition.
 Plugging  the ansatz (\ref{formu2})  into the equation   (\ref{R2}), retaining the terms  at order $\sqrt{\ep t}^{0}$ and  taking into account that  the velocity $v^0$ given by Euler is divergence free leads to the orthogonality equation:
$ n.  \partial_X \tilde V = 0$, 
which by integration, with the condition (\ref{sharp1vitesse}) leads to the condition:
 \begin{eqnarray}
  \label{orthov}
 n.  \tilde V = 0.
\end{eqnarray}
An important consequence of the  condition  (\ref{orthov}) is to kill the third term in  (\ref{heu})
 which is the only nonlinear one.  
 Here lies an analogy with the WKB theory of the propagation of high frequency oscillations for hyperbolic systems (see for instance \cite{metivier}). 
 The  condition  (\ref{orthov}) can be seen as a polarization of the singularity on the components tangential to the "phase" $\phi^{0} $. Then the vanishing of the burgers-like term in  (\ref{heu}) can be interpreted as a transparency property: the self-interaction of the singularity vanishes because this latter is characteristic for a field which is linearly degenerate (actually this concept belongs to the hyperbolic theory but the incompressible limit is reminiscent of this fact).

 \subsubsection{Getting rid of the pressure}
\label{rid}

The equation (\ref{heu}) involves both $ \tilde V$ and $ \tilde P$.
However the pressure in the NS equations is not truly an unknown but can be recovered from the velocity (as recalled in
(\ref{pressure})) so that we expect that the same holds for the profiles.
One way to proceed is to  project normally  the equation (\ref{heu}), taking into account the equation (\ref{neq}) for $n(t,x)$  and using the condition (\ref{orthov}), to get
\begin{eqnarray}
 \label{lerayprofile}
  \partial_X \tilde{P} :=  - 2  \frac{(\tilde V \cdot\nabla_x v^0).n }{|n|^2 } .
 \end{eqnarray}
 
 For  $\pm X > 0$, we 
 integrate between $X$ and $\pm \infty$,
 taking the condition at infinity (\ref{sharp1pression}) into account,
 to find
 \begin{eqnarray}
  \label{lerayprofile2}
  \tilde{P} :=  - 2 \int_X^{\pm \infty} \frac{(\tilde V \cdot\nabla_x v^0).n }{|n|^2 } .
 \end{eqnarray}
 
   \begin{rem}
    \rm
One another way to proceed is to explicit the term $\tilde F$ which occurred  in  the expansion (\ref{formu2p}) when we were discussing the amplitude of the pressure layers by mean of the Laplace problem satisfied by   the viscous pressure $p^\ep$.
Taking  into account the condition (\ref{orthov}),  we get
 \begin{eqnarray}
  \tilde{F} = - 2 (\partial_X \tilde V \cdot\nabla_x v^0).n .
  \end{eqnarray}
Plugging on the other hand the expansion  (\ref{formu2p})  in the l.h.s. of  (\ref{formu2p})
yields 
\begin{eqnarray}
|n|^2  \partial_X^2 \tilde{P} = - 2  ( \partial_X \tilde V \cdot\nabla_x v^0).n 
 \end{eqnarray}
 and we recover (\ref{lerayprofile2}) by  integrating twice
 with vanishing conditions at infinity for  $ \tilde{P}$  and $\partial_X \tilde{P}$.
 
 It is interesting to note that
 despite the second method involves one more derivative, it has the advantage  to 
  involve less terms. Furthermore we do not need to combine with
  the equation of $n$.
  \end{rem}
  
  \begin{rem}
   \label{ppp}
      \rm
  We note that the profile  $ \tilde{P}$ is discontinuous at $X=0$.
  Actually this discontinuity is compensated by another pressure profile which depends only of $t,x$. This profile will be constructed in the section \ref{hoprofile}.
  On the other hand  we will construct a velocity profile  $ \tilde{V} $  continuous including  at $X=0$, so that   $\partial_X \tilde{P} $ will also be continuous including  at $X=0$.
 \end{rem}

 We now use  the equation (\ref{lerayprofile}) to get rid of the pressure profile into the equation  (\ref{heu}).
Inverting the two sides and dividing by $t$, we have:
\begin{equation}
  \label{heutan}
|n|^2 \partial_X^2 \tilde V  + \frac{X}{2}  \partial_X \tilde V  - \frac{1}{2}  \tilde V =
 t( \partial_t  \tilde V + v^0 \cdot\nabla_x  \tilde V +   \tilde V \cdot\nabla_x v^0   - 2  \frac{(\tilde V \cdot\nabla_x v^0).n }{|n|^2 } n) ,
\end{equation}

 \subsubsection{Avoiding a far-field degeneracy}
\label{avoiding}

The vectorfield $n$ may vanish, away the patch boundary, hence so may do the coefficient in front of the leading order in the equation  (\ref{heutan}).
To remedy to this we consider a function $a  $ in the space
 \begin{eqnarray*}
   \mathcal{B} :=
 L^\infty([0,T], C^{0, r} (  \R^d ))  \cap L^\infty (0,T;   C^{s,r}  ( \mathcal{O}_{\pm}  (t)  )) 
 \end{eqnarray*}
 satisfying the condition $ \inf_{[0,T] \times  \R^d} a =: c > 0$
  and such that  
$a = | n |^2$  when  $| \phi^{0}  | < \eta $, and we consider for the profile $V (t,x,X)$ the  {\em linear} partial  differential equation $  L V = 0$ where the differential operator $L$ is given by   $L :=  \mathcal{E}  -t (D + A)$
where
$\mathcal{E}$ and  $A$ are some operators of respective order $2$ and $0$   acting formally on functions $V(t,x,X)$ as follows:
 \begin{eqnarray*}
 \mathcal{E} V := a  \partial_X^2  V  +  \frac{X}{2}  \partial_X V  - \frac{1}{2}  V  \text{ and }
 A V :=   V \cdot \nabla_x v^0  - 2  \frac{( V \cdot\nabla_x v^0). n }{a  } n .
  \end{eqnarray*}
 Roughly speaking the equation (\ref{heutan}) is now hyperbolic in $t,x$ and parabolic in $t,X$,  for $t>0$, and degenerates into an elliptic equation in $X$ for $t=0$.

The substitution of $a$ instead of $|n |^2$ is almost harmless since their values are different 
only for $| \phi^{0}| \geqslant \eta $, so that the corresponding values of the eventual solutions  $V (t,x,\frac{\phi^{0}(t,x)}{\sqrt{\ep t} } )$ and $\tilde V (t,x,\frac{\phi^{0}(t,x)}{\sqrt{\ep t} } )$ respectively given by the equations  (\ref{heutan}) and  (\ref{heu2r}) (the equation $  L V = 0$) both tend to $0$ as $\sqrt{\ep t} $ tends to $0$, because of the vanishing condition (\ref{sharp1vitesse}) for $X$ at infinities.

\subsubsection{Transmission conditions}
\label{wp1}

Of course we hope to extirpate from the previous equation a non-trivial solution.
To this purpose an important point is that we look for a solution $ \tilde V $ with a $X$ derivative $\partial_X \tilde V $ discontinuous in $X=0$. 
  Actually because of the parabolic nature of the Navier Stokes equations, we expect that  $v^\ep$ and   $\omega^\ep$ are continuous  including through  $\phi^{0} =0$ (these are the Rankine-Hugoniot conditions associated to the problem),  
  which lead to the transmission conditions: $ \tilde V$ and 
 $\omega^0 + \tilde \Omega $ should be continuous, which (taking into account the equalities (\ref{rotrot}), (\ref{ortho}) and (\ref{orthov}))
 is equivalent to the  transmission conditions: $ \tilde V$ and 
 $n \wedge \omega^0 - |n|^2 \partial_X \tilde V  $ should be continuous.
 More precisely this means a priori that
 \begin{eqnarray}
\label{wp71} \tilde  V |_{X=0^+ , \phi^{0} = 0^+ } -  \tilde V |_{X=0^- , \phi^{0} = 0^- } &=& 0,
\\ \label{wp72}  |n|^2 \partial_X \tilde V |_{X=0^+ , \phi^{0} = 0^+ }  -  |n|^2 \partial_X \tilde V |_{X=0^- , \phi^{0} = 0^- } &=& - ( n \wedge \omega^0 |_{ \phi^{0} = 0^+ }  - n \wedge \omega^0 |_{ \phi^{0} = 0^- } ).
  \end{eqnarray}
   Since $X$ is the placeholder for $\frac{\phi^{0}(t,x)}{\sqrt{\ep t}}$ the function $ \tilde  V (t,x,X)$ needs to be defined only when $X$ and $\frac{\phi^{0}(t,x)}{\sqrt{\ep t}}$ share the same sign.
   However it is useful to look for a profile  $V (t,x,X)$ defined for $(t,x,X)$ in the wholedomain  $\mathcal{D}:= (0,T) \times \R^d  \times \R$.
  As a consequence we will actually look at the following transmission conditions:  for any $(t,x) \in (0,T) \times \R^d$: $  \lbrack  V \rbrack = 0
$ and $ \lbrack  \partial_X V   \rbrack = -  \frac{ n \wedge (\omega^0_+ -  \omega^0_-  )}{a}$,   where  the brackets denote the jump 
  $ \lbrack  V \rbrack  = V |_{X=0^+} -  V |_{X=0^-}  $ across $\{X=0\}$ and where 
  $\omega^0_\pm$ are two functions in $ L^\infty \Big((0,T), C^{s, r} \big(  \R^d )\Big)$ such that $\omega^0_\pm |_{  \mathcal{O}_{\pm}  (t) } = \omega^0$.
  More precisely we recall from the compendium (precisely the point \ref{litt6} of Theorem \ref{litt}) that the restrictions of the flow  $\chi^0$ of the Euler solution on each side of the boundary are analytic with respect to time with values in  $C^{s+1, r}$.
  Thanks to Theorem \ref{rychkov} 
  there exists some extensions $\chi^0_\pm $ 
  analytic on $(0,T)$ with values in $C^{s+1, r} (  \R^d )$
  of the restriction $\chi^0 |_{\mathcal{O}_{0,\pm}  } $.
  We define the corresponding velocities  $v^0_\pm$ by 
  $ v^0_\pm (t,.):= (\partial_t  \chi^0_\pm ) (t, (\chi^0_\pm )^{-1} (t,.))$ and the  corresponding vorticities  $\omega^0_\pm$ by 
  $ \omega^0_\pm := \text{curl } v^0_\pm$.
 As a consequence, with the notation of the introduction,  $\omega^0_\pm$ are in $ \mathcal{B}_D$.
  
     The profile equations  (\ref{heu2r})-(\ref{heu21r}) announced in the introductory section 
\ref{viscous} are therefore derived.

\subsection{At initial time}
\label{wp4}

We expect that the initial values $V_{0} (x,X) := V (0,x,X) $ satisfy the problem 
obtained 
 formally making  $t=0$ in the equations (\ref{heu2r})-(\ref{heu21r}), i.e. :
 \begin{eqnarray}
  \label{ieu20}
a_0  \partial_X^2  V_{0}  +  \frac{X}{2}  \partial_X V_{0}  - \frac{1}{2}  V_0 &=& 0 \quad \text{when }  \quad  \pm X > 0,
  \\  \label{ieu21} V_{0} |_{X=0^+  } - V_{0} |_{X=0^- } &=& 0,
\\  \label{ieu22}  a_0 \partial_X V_{0} |_{X=0^+  }  - a_0 \partial_X  V_{0} |_{X=0^-  } &=& g_{0} .
\end{eqnarray}
The functions  $a_{0}(x)$ and $g_{0}(x)$  which denote respectively the initial value of $a$ and of $n \wedge (\omega^0_+ -  \omega^0_-  )$ are in  $C^{0, r} (\R^d )  \cap C^{s,r}  (\mathcal{O}_{\pm,0}  )$.
\begin{prop}
\label{propv}
There exists a unique couple of solutions 
 \begin{eqnarray}
 \label{miracle0}
 V_{0}  (x, X) \in  C^{0,r}  \big(  \R^d , p-\mathcal{S} (\R) \big)  \cap C^{s,r}  \big(   \mathcal{O}_{\pm,0} , p-\mathcal{S} (\R) \big)
 \end{eqnarray}
   which satisfy  the problem  (\ref{ieu20})-(\ref{ieu21})-(\ref{ieu22}).
  Moreover 
\begin{eqnarray}
  \label{perp}
  V_{0} (x,X)\cdot  n_{0} (x) &=& 0  \quad \text{for}   \pm X >0 .
\end{eqnarray}
  \end{prop}
\begin{proof}
\textbf{  Step $1$: Reduction}.
  We first reduce the transmission conditions to the homogeneous case  by
  defining the functions $ \tilde{V}_{0}$ by
 \begin{eqnarray}
    \tilde{V}_{0} :=  V_{0} 
    \pm \frac{1}{2}   \frac{g_{0}}{a_{0}} (e^{\mp X} -  e^{\mp 2X} )  \quad \text{when }  \quad  \pm X > 0,
     \end{eqnarray}
   so that the problem  (\ref{ieu20})-(\ref{ieu21})-(\ref{ieu22}) is turned into the following one (dropping the tilda and the index $0$)
 \begin{eqnarray}
  \label{jeu20}
a_{0}  \partial_X^2  V  +  \frac{X}{2}  \partial_X V  - \frac{1}{2}  V &=& f \quad \text{when }  \quad  \pm X > 0,
  \\  \label{jeu21} V |_{X=0^+  } - V |_{X=0^- } &=& 0,
\\  \label{jeu22}   \partial_X V |_{X=0^+  }  -  \partial_X  V |_{X=0^-  } &=& 0 ,
\end{eqnarray}
with $f \in C^{0, r} \big(  \R^d ,  H^{-1} ( \R)\big) \cap
C^{0, r} \big(  \R^d , p-\mathcal{S} (  \R) \big)  \cap C^{s, r} \big( \mathcal{O}_{\pm,0} ,  H^{-1} ( \R)\big)  \cap  C^{s, r} \big( \mathcal{O}_{\pm,0} , p-\mathcal{S}  (\R)\big)   $. 

\textbf{Step $2$: Cutting-off the unbounded coefficient. }
Because of the unbounded coefficient $X$ in the equation (\ref{ieu20}) the previous problem 
 does not enter in the classical theory of elliptic problems (with $x$ as parameter). 
 To remedy to this we introduce a cut-off.
We consider $\sigma > 0$ and a smooth function $\chi_\sigma$ such that 
$\chi_\sigma (X) = X $ for $| X | < \sigma $, $\chi_\sigma (X) = 3\sigma /2 $ for $| X | >2 \sigma $ and $\| \chi_\sigma  '  \|_{L^\infty (\R) } < 1$.
We will work with the modified equation:
 \begin{eqnarray}
  \label{jeu20b}
a_{0}  \partial_X^2  V  +  \frac{\chi_\sigma (X) }{2}  \partial_X V  - \frac{1}{2}  V &=& f \quad \text{when }  \quad  \pm X > 0,
\end{eqnarray}
\textbf{  Step $2$: Existence and uniqueness of weak solution}. 
The variational formulation of the problem  (\ref{jeu20b})-(\ref{jeu21})-(\ref{jeu22})  reads: for $f \in   L^2 (\R^d , H^{-1} ( \R))$ find  $V  \in L^2 (\R^d , H^1 ( \R))$  such that $B_\sigma (V, W)  = - < f,W>$ for all $W \in  L^2 (\R^d , H^1 ( \R))$ where  $< .,.>$ denotes the bracket of duality between $  L^2 (\R^d , H^{-1} ( \R))$ and $L^2 (\R^d , H^1 ( \R))$, and
 $B_\sigma (V, W)$ is the following  bilinear  form on $ L^2 (\R^d , H^1 ( \R)) \times  L^2 (\R^d , H^1 ( \R))$:
 \begin{eqnarray} 
 B_\sigma (V, W) :=  \int_{   \R^d   \times \R } a_{0}   \partial_X V .  \partial_X W
 -  \frac{\chi_\sigma (X) }{2}  \partial_X V .  W + \frac{1}{2}  V .  W .
\end{eqnarray}
 Since the bilinear form $B_\sigma $ is continuous (thanks to the cut-off $\chi_\sigma$) and coercive:
 \begin{eqnarray}
 \label{coercive}
B_\sigma (V, V) = \int_{   \R^d   \times \R } a_{0}  | \partial_X V |^2
 +  \frac{1}{2} (1 + \frac{\chi'_\sigma (X)}{2})  | V |^2 
\end{eqnarray}
we infer from Lax-Milgram theorem that
 there exists a unique weak/variational 
  solution of problem  (\ref{jeu20b})-(\ref{jeu21})-(\ref{jeu22}).

\textbf{  Step $3$:  $V \in  L^\infty ( \R^d , H^1 ( \R )) $}.
We now turn our attention to regularity restricting ourselves for shortness to establish  a-priori estimates.
First  multiplying the equations (\ref{jeu20b}) by  $V$ and
     integrating   (only) in $X$  yields for any $x \in  \R^d$
 \begin{eqnarray}
  \label{1410}
 \int_{\R } a_{0} | \partial_X V |^2
 +   \frac{1}{2} (1 + \frac{\chi'_\sigma (X)}{2})   | V |^2 dX
=
- << f,W>>
 \end{eqnarray}
where  $<< .,.>>$ denotes this time the bracket of duality between $  H^{-1} ( \R)$ and $ H^1 ( \R)$
so that we get 
 \begin{eqnarray}
  \label{1758}
 \| V \|_{ L^\infty ( \R^d , H^1 ( \R )) }
\lesssim
  \| f \|_{ L^\infty ( \R^d , H^{-1} ( \R )) }.
\end{eqnarray}

\textbf{  Step $4$:  $V \in C^{s, r} ( \mathcal{O}_{\pm,0}  , H^1 ( \R )) $}.
We follow the same strategy than  in section  \ref{pte}. 
By Rychkov's Theorem \ref{rychkov} 
there exists some 
 extensions  $a_{\pm,0}  $   in  $ C^{s,r}  (\R^d    )$  and
 $$f_{\pm} \in  C^{s, r} \big( \R^d ,  H^{-1} ( \R)\big)  \cap  C^{s, r} \big(\R^d  , p-\mathcal{S}  (\R)\big)   $$
of the restrictions of $a_0$ and $f$ to $ \mathcal{O}_{\pm,0}$.
Then  we perform estimates of the solutions $V_{\pm}$ of the problems
 \begin{eqnarray}
  \label{jeu20f}
a_{\pm,0}  \partial_X^2  V_{\pm}  +  \frac{\chi_\sigma (X) }{2}  \partial_X V_{\pm}  - \frac{1}{2}  V_{\pm} &=& f_{\pm}  \quad \text{when }  \quad  \pm X > 0,
  \\  \label{jeu21f} V_{\pm} |_{X=0^+  } - V_{\pm} |_{X=0^- } &=& 0,
\\  \label{jeu22f}   \partial_X V_{\pm} |_{X=0^+  }  -  \partial_X  V_{\pm} |_{X=0^-  } &=& 0 ,
\end{eqnarray}
 for $x$ running in the full range $\R^d$
through a Fourier analysis.
 Finally we observe that $V |_{\mathcal{O}_{\pm,0}}$ and $V_{\pm}$ satisfy both 
 the equations  (\ref{jeu20b})-(\ref{jeu21})-(\ref{jeu22}) for $x$ in $\mathcal{O}_{\pm,0}$.
 Proceeding as in step $3$ we get that they are equal.

 We are therefore reduced to the case where the functions  $a_0$ and $g_0$ are respectively in
 $C^{s, r} ( \R^d)$ and
$ C^{s, r} \big( \R^d ,  H^{-1} ( \R)\big)  \cap  C^{s, r} \big(\R^d  , p-\mathcal{S}  (\R)\big)$.
We are going to prove by iteration for $-1  \leqslant  l  \leqslant s$ that
 \begin{eqnarray*}
 \| V \|_{  C^{l, r}  ( \R^d , H^1 ( \R )) }
\lesssim
  \| f \|_{  C^{l, r} ( \R^d , H^{-1} ( \R )) }.
\end{eqnarray*}
To do this  we make use of a spectral localization with respect to $x$ that is we apply the operators $\Delta_j $ to the equations (\ref{jeu20b}) to get for $j \geq -1$
  the equations:
 \begin{eqnarray}
  \label{eqr1}
 a_{0}  \partial^2_X  \Delta_j V  +  \frac{\chi_\sigma}{2}  \partial_X \Delta_j V  - \frac{1}{2}\Delta_j  V =
 \Delta_j f + |\lbrack a_{0} ,   \Delta_j   \rbrack    \partial^2_X V
 \end{eqnarray}
for $ \pm X>0  $  and to the interface condition  (\ref{jeu21})-(\ref{jeu22}) to get at $X=0$ 
  \begin{eqnarray}
  \label{eqr3}
   \left\{
\begin{array}{c}
   \Delta_j  V |_{X=0^+  } - \Delta_j  V |_{X=0^-  } = 0,
  \\   \Delta_j    \partial_X V|_{X=0^+  } -   \Delta_j   \partial_X  V|_{X=0^-  } =   0 .
   \end{array}
  \right.
    \end{eqnarray}
    We want to show that 
  \begin{eqnarray*}
      \text{sup}_{j  \geqslant -1}  \    2^{j (l+r)} ( \| \Delta_j V   \|_{ L^\infty (\R^3  , H^1  (  \R)) }  < \infty .
     \end{eqnarray*}
  To do this we multiply the equation (\ref{eqr1}) by  $\Delta_j V$ 
    we integrate in $X$ over  $ \R_\pm $, 
    and we sum the two resulting equations, noticing that the boundary term produced by the integration by parts of the sum of the respective first terms vanish.
     We thus get for any $x \in  \R^d$
   the identity 
  \begin{eqnarray*}
 \int_{ \R} a_{0}  | \partial_X  \Delta_j V |^2 
 +   \frac{3}{4}   | \Delta_j  V |^2
 =  - <<  \Delta_j f  + |\lbrack a_{0} ,   \Delta_j   \rbrack    \partial^2_X V , \Delta_j V >>.
 \end{eqnarray*}
       Now we also have 
        $$ |\lbrack a_{0} ,   \Delta_j   \rbrack    \partial_X V|_{X=0^+  }  =  |\lbrack a_{0} ,   \Delta_j   \rbrack      \partial_X  V|_{X=0^-  }  $$
        so that
  \begin{eqnarray}
  \label{1741}
  \int_{ \R } a_{0}  | \partial_X  \Delta_j V  |^2
 +    \frac{1}{4}  | \Delta_j  V |^2
 \lesssim \|  \Delta_j f  \|^2_{ H^{-1}  (  \R) } +  I_{j} (x) ,
   \end{eqnarray}
        where $ I_{j} (x) $ denote
  \begin{eqnarray*}
  I_{j} (x) :=   \int_{  \R }   |  \partial_X \Delta_j V |. |\lbrack a_{0} ,   \Delta_j   \rbrack    \partial_X V | .
         \end{eqnarray*}
        We will prove  the following commutator estimate:
  \begin{lem}
   \label{commumu}
         For any $x \in  \R^d$ there holds 
  \begin{eqnarray*}
   \text{sup}_{j  \geqslant -1}  \  2^{2j (l+r)}  I_{j}(x)  \lesssim 
      \|  \partial_X V \|_{ C^{l, r} \big(  \R^d ,   L^2 ( \R) \big) } .   \|  \partial_X  V \|_{ C^{l-1, r}  \cap L^\infty \big(  \R^d ,   L^2 ( \R) \big) }
   \end{eqnarray*}
          \end{lem}
     Let us admit for a while Lemma \ref{commumu} and  infer from the estimate  (\ref{1741})  
       that 
  \begin{eqnarray*}
 \|    V      \|_{ C^{l, r} \big(  \R^d ,   H^1 ( \R) \big) } 
  \lesssim
\|   f  \|_{ C^{l, r} \big(  \R^d ,   L^2 ( \R) \big) } + 
  \|   \partial_X V      \|_{ C^{l-1, r} \cap L^\infty   \big(  \R^d ,   L^2 ( \R) \big) }
   \end{eqnarray*}
       so   that
       starting with the case $l=-1$ -which is a consequence of the estimate  (\ref{1758})-
        the iteration 
       can be done till we get  $V \in C^{s, r} ( \R^d , H^1 ( \R )) $. 
       
  \begin{proof}[Proof of Lemma \ref{commumu}]
  We will consider only     $j   > 0$, the case $j=0$ corresponding to minor  modifications of notation and being actually easier.
     We make use of the paraproduct writing
  \begin{eqnarray*}
  \lbrack a_{0} ,   \Delta_j   \rbrack    \partial_X V = 
   \lbrack T_{a_{0}} ,   \Delta_j   \rbrack    \partial_X V
   + (a_{0} - T_{a_{0}} )  \Delta_j  \partial_X V
   -  \Delta_j  ( a_{0} - T_{a_{0}} )  \partial_X V 
         \end{eqnarray*}
   so that  $ I_{j}= I^1_{j}  +  I^2_{j} + I^3_{j} $ where  
  \begin{eqnarray*}
  I^1_{j} (x) &:=&   \int_{  \R }   |  \partial_X \Delta_j V |. |\lbrack  T_{a_{0}} ,   \Delta_j   \rbrack    \partial_X V | ,
  \\     I^2_{j} (x) &= &  \int_{  \R }   |  \partial_X \Delta_j V |. | (a_{0} - T_{a_{0}} )  \Delta_j \partial_X V | ,
    \\     I^3_{j} (x) &= &  \int_{  \R }   |  \partial_X \Delta_j V |.  |\Delta_j  (a_{0} - T_{a_{0}} ) \partial_X V | .
         \end{eqnarray*}
    Referring to the definitions (\ref{parap0}) and  (\ref{convol})
      we have
 \begin{eqnarray*}
   \lbrack T_{a_{0}} ,   \Delta_j   \rbrack    \partial_X V &= &
    \sum_{k   \geqslant 1}  \lbrack S_{k-1} a_{0} , \Delta_j    \rbrack     \Delta_k  \partial_X V 
    \\  &= & 
     \sum_{k   \geqslant 1, | k-j|  \leqslant 5 } 
    2^{3j} \int_{\R^3  } 
    (S_{k-1} a_{0} (x) - S_{k-1} a_{0} (y)) 
     \tilde{h} (2^{j}  (x-y) )   \Delta_k  \partial_X V  (y) dy .
   \end{eqnarray*}
         Hence 
         \begin{eqnarray*}
|   \lbrack T_{a_{0}} ,   \Delta_j   \rbrack    \partial_X V   |
      \leqslant  2^{-j}   \sum_{k   \geqslant 1, | k-j|  \leqslant 5}  \| S_{k-1} a_{0} \|_{Lip} \int_{\R^3  } g ( 2^{j}  (x-y) )   | \Delta_k  \partial_X V  (y) | 2^{3j} dy 
           \end{eqnarray*}
         where 
         $g$ denotes the function $g(x) :=  |x | . | \tilde{h}(x)|$ which is in $L^1 ( \R^3 )$.
         Using Fubini theorem we get that
 for any $x \in  \R^3$
         \begin{eqnarray*}
         I^1_{j} (x)     \lesssim    \sum_{k   \geqslant 1, | k-j|  \leqslant 5}  2^{-j}  
         \int_{\R^3  }     g ( 2^{j}  (x-y) ) 
                 \big(      \int_{  \R } 
           |  \partial_X \Delta_j V (x,X) |. |  \partial_X  \Delta_k V  (y,X) |  dX  \big) 2^{3j} dy .
           \end{eqnarray*}
We now use Cauchy-Schwarz inequality to get for any $x \in  \R^3$
   \begin{eqnarray*}
          I^1_{j} (x)       &\lesssim &   \sum_{k   \geqslant 1, | k-j|  \leqslant 5}  2^{-j} 
              \|  \partial_X \Delta_j V (x,.) \|_{L^2 (\R) }
              \int_{\R^3  }  g ( 2^{j}  (x-y) )
                \|  \partial_X  \Delta_k V  (y,.) \|_{L^2 (\R) }   2^{3j} dy 
                \\      & \lesssim&   \sum_{k   \geqslant 1, | k-j|  \leqslant 5}  2^{-j}   
                  \|  \partial_X \Delta_j V (x,.) \|_{L^2 (\R) }
                \|  \partial_X  \Delta_k V  (y,.) \|_{ L^\infty (\R^3  ,L^2 (\R)) } 
           \end{eqnarray*}
hence
  \begin{eqnarray*}
   \text{sup}_{j  \geqslant -1}  \  2^{2j (l+r)}  I^1_{j}(x)  \lesssim 
      \|  \partial_X V \|_{ C^{l, r} \big(  \R^d ,   L^2 ( \R) \big) } .   \|  \partial_X  V \|_{ C^{l-1, r} \big(  \R^d ,   L^2 ( \R) \big) }
   \end{eqnarray*}
     Similar bounds for $I^2_{j}$ and $I^3_{j}$ can be obtained, mixing once again classical paradifferential arguments with a Fubini argument.

\end{proof}

\textbf{  Step $4$:  $V \in C^{s, r} ( \R^d ,  p-\mathcal{S} ( \R )) $}. 
Mimicking  Step $3$ we can prove that 
  \begin{eqnarray}
   \label{1453}
 \|    V      \|_{ C^{s, r} \big(  \R^d ,   H^1 ( \R) \big) } 
  \lesssim
\|   f  \|_{ C^{s, r} \big(  \R^d ,   p-L^2 ( \R) \big) } .
   \end{eqnarray}
    Now we are going to prove by induction on $l$ and then by induction on $k$  that  for all $k$ and $l$ in $\N$
\begin{eqnarray}
 \label{miraclekl}
 V_{k,l,0} := X^k \partial_X^l V   \in   C^{s,r}  \big(\R^3 , p-H^1  (  \R)  \big).
 \end{eqnarray}
Thanks to Sobolev embeddings  we will infer  (\ref{miracle0}).
The point is that $ V_{k,l}$, let us simply write $ \tilde{V}$, verify a problem of the form
 \begin{eqnarray}
  \label{1905a} a_{0} \partial_X^2 \tilde{V}_{}  +  \frac{X}{2}  \partial_X \tilde{V}_{}  - \frac{1}{2} \tilde{V}_{} = \tilde{F}_{} \quad \text{when }  \quad  \pm X > 0,
\\  \label{1905b}  \tilde{V} |_{X=0^+  } - \tilde{V} |_{X=0^- } = J,
\\  \label{1905c}  a_{0}  \partial_X \tilde{V} |_{X=0^+  }  -  a_{0}  \partial_X  \tilde{V} |_{X=0^-  } &=& \tilde{J},
\\ \label{1905d}   \tilde{V} (x,X)   \rightarrow 0 \quad \text{when }  \quad  X \rightarrow \pm \infty. 
     \end{eqnarray}
where $\tilde{F}$ (respectively $J$ and $ \tilde{J}$) is in $C^{s, r} \big(  \R^d , p-L^2 (  \R) \big)$ (resp. in  $C^{s, r} (  \R^d )$) thanks to the previous steps.
 Once again we can reduce the transmission conditions to the homogeneous case  by
  defining the functions $ V$ by
 \begin{eqnarray}
V :=    \tilde{V}  \pm \frac{1}{2} \big( (2J + \frac{\tilde{J}}{a}) e^{\mp X} -  (J + \frac{\tilde{J}}{a}) e^{\mp 2X} \big) \quad \text{when }  \quad  \pm X > 0.
     \end{eqnarray}
 which satisfy  a problem of the form
 \begin{eqnarray}
\label{1906a} a_{0} \partial_X^2 V  +  \frac{X}{2}  \partial_X V  - \frac{1}{2} V =  F ,
\\  \label{1906b}  V |_{X=0^+  } - V_- |_{X=0^- } = 0,
\\  \label{1906c} a_{0} \partial_X V |_{X=0^+  }  - a_{0} \partial_X  V |_{X=0^-  } &=& 0,
\\ \label{1906d}   V (x,X)   \rightarrow 0 \quad \text{when }  \quad  X \rightarrow \pm \infty. 
   \end{eqnarray}
  where $F$ is in $C^{s, r} \big(  \R^d , L^2 (  \R) \big)$ so that we can apply  (\ref{1453}).
  
  \textbf{  Step $5$:  $V \in C^{0, r} (   \R^d ,  p-\mathcal{S} ( \R )) $}.
  It is actually given by the analysis of step $4$. 
  
   \textbf{  Step $6$:   Passing to the limit in $\sigma$. }
      We finally let $\sigma$ goes to infinity. 
      The  estimate  (\ref{weaksigmap0})  is uniform with respect to  $\sigma$.
      Using weak compactness we can pass to the weak limit so that we get a solution to the uncut  equation.
  
 \textbf{  Step $7$: proof  of (\ref{perp})}.
It suffices to use uniqueness argument  based on equality (\ref{1410}).

    \end{proof}

\subsection{Existence and uniqueness of the layer profile}
\label{wp}

Let us now study  the time-dependent equation (\ref{heu2r}).
We denote by $\mathcal{D}$ the domain  $\mathcal{D}:= (0,T) \times \R^d  \times \R$, by
 $\mathcal{D}_\pm$ the restrictions $\mathcal{D}_\pm := (0,T) \times \R^d  \times \R_\pm $ on each side of the boundary $\Gamma  := (0,T) \times \R^d  \times \{0 \} $
  and by $A$, $ \tilde{L}$ and $L$  the operators of respective order $0$, $1$ and $2$  acting formally on functions $V(t,x,X)$ as follows:
  \begin{eqnarray*}
A V :=   V \cdot \nabla_x v^0
  - 2  \frac{( V \cdot\nabla_x v^0). n }{a  } n
\\ \tilde{L} V:=  \frac{X}{2}  \partial_X V  - \frac{1}{2}  V -
 t( \partial_t  V + v^0 \cdot\nabla_x V   + A V ),
 \\  L V :=   a  \partial_X^2  V  +    \tilde{L} V  .
\end{eqnarray*}
The profile problem now reads as follows
  \begin{equation}
   \label{weak0}
  L  V = f \  \text{on }  \  \mathcal{D}_\pm ,
\  ( \lbrack  V \rbrack , \lbrack \partial_X V \rbrack ) = (0,g)   \ \text{on }  \   \Gamma,
\end{equation}
the decreasing of $V$ for large $X$  being encoded in the choice of the space $E_1 := L^2 ( (0,T) \times \R^d , H^1 ( \R) )$.
\begin{theo}
\label{weakeq}
For any $f \in E'_1$, for any $g \in  L^2 ((0,T) \times \R^d )$
there exists exactly one solution $V \in E_1$ of (\ref{weak0}). 
In addition the function 
$ \sqrt{t} \|  V(t,.,.)  \|_{L^2 ( \R^d \times \R) }$ is continuous on $(0,T)$.
\end{theo}
The first equation of (\ref{weak0}) is satisfied  in the sense of distributions on both $\mathcal{D}_\pm $ and the sense given to the jump conditions is explained  in the proof below.
\begin{proof}
We consider $\sigma > 0$ and a smooth function $\chi_\sigma$ such that 
$\chi_\sigma (X) = X $ for $| X | < \sigma $, $\chi_\sigma (X) = 3\sigma /2 $ for $| X | >2 \sigma $ and $\| \chi_\sigma  '  \|_{L^\infty (\R) } < 1$.
We will work with the modified operators
$\tilde{L}_\sigma :=  \tilde{L} +  \frac{\chi_\sigma (X) - X }{2} \partial_X  V$ and
 $L_\sigma V :=   a  \partial_X^2  V  +    \tilde{L}_\sigma V  $ whose coefficients are bounded.
Even more the coefficients of the first order part are Lipschitz.

We now explain the meaning of the jump conditions in the equation (\ref{weak0}).
Since  $V$ is in $E_1$ the jump $ \lbrack  V \rbrack |_{ \Gamma} $ is in $L^2  (\Gamma)$.
To give a sense to the jump of the derivative $ \lbrack \partial_X V \rbrack$ we will use the equation.
For any $V$ in the space $E_2 :=  \{ V \in C_0 (\mathcal{D}) / \ V|_{ \mathcal{D}_\pm} \in C^\infty \}$ and  $W$ in $H^{1} (\mathcal{D})$ we have, integrating by parts, the following Green  identity:
 \begin{equation}
  \label{greentrace}
\sum_\pm \int_{ \mathcal{D}_\pm } L_\sigma V . W =
-  \int_{ \mathcal{D} } a \partial_X  V . \partial_X  W 
+  \int_{ \mathcal{D} }  \tilde{L}_\sigma^* W.V
- \int_{ \Gamma}  a  \lbrack \partial_X V \rbrack . W 
- T \int_{ \tilde{\Gamma}} W.V 
\end{equation}
where $\tilde{\Gamma}:=  \{T \} \times \R^d  \times \R $ and where $ \tilde{L}_\sigma^*$  denotes the operator (the adjoint of  $\tilde{L}_\sigma$)
 \begin{equation}
  \label{heu2green}
  \tilde{L}_\sigma^* V:=  -  \frac{\chi_\sigma (X)  }{2}  \partial_X V 
   - \frac{1  }{2} (1 + \chi_\sigma  ' (X)) V +
 t( \partial_t  V + v^0 \cdot\nabla_x V   +  (1+ \dive v^0 -A ) V ) .
\end{equation}
In fact less smoothness is needed. 
Let us introduce the Hilbert space
$E_4 :=  \{ V \in E_1 /  \ L_\sigma V  \in H^{-1} (\mathcal{D}) \}$ endowed with the norm
$\| V \|_{E_4  } := \| V \|_{E_1  } +   \| L_\sigma V \|_{ H^{-1} (\mathcal{D})}$.
Thanks to a classical lemma by Friedrichs \cite{Friedrichs} the space $E_2$ is dense in $E_4$.
 \begin{lem}
  \label{green2}
The map
 \begin{equation}
  \label{trace}
  V  \in   E_2 \mapsto \tau := 
   \left\{
 \begin{array}{c}
 a  \lbrack \partial_X V \rbrack   \ \textbf{ on } \   \Gamma
\\  T V   \ \textbf{ on } \   \tilde{\Gamma} 
\end{array}
\right.
\end{equation}
extends uniquely to a continuous linear map
from $E_4 $ to $H^{- \frac{1}{2} } (\Gamma \cup  \tilde{\Gamma} ) $
 and  Green's identity  (\ref{greentrace}) is still valid for any couple $(V,W)$ in  $E_4 \times H^{1} (\mathcal{D})$ in the generalized sense that 
 \begin{eqnarray}
  \label{heu88882bis}
< L_\sigma V , W >_{ H^{-1} (\mathcal{D}) ,   H^{1} (\mathcal{D}) } =
  \int_{ \mathcal{D} }  (V.  \tilde{L}_\sigma^* W  -a \partial_X  V . \partial_X  W)
- <  \tau  , W|_{\Gamma \cup  \tilde{\Gamma}  }  >_{H^{- \frac{1}{2} } (\Gamma \cup  \tilde{\Gamma} )  , H^{ \frac{1}{2} } (\Gamma \cup  \tilde{\Gamma} )   } .
\end{eqnarray}
\end{lem}
\begin{proof}
Let $V$ be in $E_2$ and $ \tilde{W}$ be in $H^{ \frac{1}{2} } (\Gamma \cup  \tilde{\Gamma} )$.
There exists a function $W$ in $H^{1} (\mathcal{D})$ such that 
$W |_{ \Gamma \cup  \tilde{\Gamma}} =  \tilde{W} $.
From  Green's identity  (\ref{greentrace}) we  infer that
\begin{equation*}
 |  \int_{ \Gamma \cup  \tilde{\Gamma} }  \tau  \tilde{W}  | \leqslant
C \| V \|_{ E_4 }  \| W \|_{ H^{1} (\mathcal{D}) }  \leqslant
C \| V \|_{ E_4 }  \|  \tilde{W} \|_{ H^{ \frac{1}{2}  } (\Gamma \cup  \tilde{\Gamma} ) }.
\end{equation*}
Hence by Hahn-Banach theorem we get the existence of a continuous extension, which is unique because of the density stated above.
\end{proof}

We therefore have given a meaning  to the problem  (\ref{weak0}).
This meaning can seem weak but the next result said that is actually strong.
We introduce the space $ H^{1,2} (\mathcal{D})$ of the functions $V \in  H^1 (\mathcal{D})$ such that $\partial^2_{X}  V |_{ \mathcal{D}_\pm} $ are in  $L^2$.

\begin{lem}
  \label{green3}
  If $V  \in E_4$ satisfies the jump conditions
 $  \lbrack  V \rbrack = 0$ and  $\lbrack \partial_X V \rbrack  = g$ on $ \Gamma$ in the sense given by Lemma \ref{green2} 
   then there exists a sequence $V^\varepsilon$ in $H^{1,2} (\mathcal{D})$ converging to $V$ in $E_4$ and a sequence $g^\varepsilon$ converging to $g$ in $ L^2 ((0,T) \times \R^d )$ such that
   $  \lbrack  V^\varepsilon \rbrack = 0$ and  $\lbrack \partial_X V^\varepsilon \rbrack  = g^\varepsilon$ on $ \Gamma$.   
\end{lem}
\begin{proof}
As this kind of process is very classical, see for instance Rauch \cite{Rauch}, we only briefly sketch the proof.
The idea is to construct the sequence  $V^\varepsilon$ by convoluting in the variables $t,x$ only to preserve the jump conditions, to use Friedrichs lemma to prove the convergence in $E_4$ and then to gain the extra $X$ derivative, that is to prove that the $V^\varepsilon$ are in $H^{1,2} (\mathcal{D})$, thanks to the equation.
\end{proof}

  We will now prove uniqueness as a consequence of the following 
  estimate: 
   for any function $V$ in $ E_1$ satisfying
  \begin{equation}
   \label{weaksigma}
  L_\sigma  V = f \ \text{when }  \  \text{on }  \  \mathcal{D}_\pm ,
\  \lbrack  V \rbrack =0  \ \text{ and }  \  \lbrack \partial_X V \rbrack = g   \ \text{on }  \   \Gamma,
\end{equation}
there  holds 
 \begin{eqnarray}
    \label{weaksigmap}
     \|   V  \|_{E_1 } 
      \lesssim
         \|   f  \|_{E'_1 } +  \|   g  \|_{ L^2 ( \R^d \times \R)} .
  \end{eqnarray}
Let us first see the case where $V$ is in $ H^{1,2} (\mathcal{D})$.
Then using Green's identity  (\ref{greentrace}) with $V=W$ and density we get that
  \begin{eqnarray*}
   \int_{ \mathcal{D} }  |  V |^2 +  | \partial_X V |^2
   + T  \int_{ \tilde{\Gamma}  }  |  V |^2
   \lesssim
    \int_{ \mathcal{D} } t  |  V |^2 +   |f. V  |
     +   \int_{ \Gamma  }   |g. V  |,
    \end{eqnarray*}
actually not only for $T$ but for any $t$ in $(0,T)$,  so that using a Gronwall lemma we get 
  \begin{eqnarray}
    \label{weaksigmap0}
     \|   V  \|_{E_1 } +  \sqrt{T} \|  V(T,.,.)  \|_{L^2 ( \R^d \times \R) }
      \lesssim
         \|   f  \|_{E'_1 } +  \|   g  \|_{ L^2 ( \R^d \times \R)} .
     \end{eqnarray}
Now to deal with the general case we introduce the Hilbert space
$E_3 :=  \{ V \in E_1 /  \ L_\sigma V  \in E'_1 \}$ endowed with the norm
$\| V \|_{E_3  } := \| V \|_{E_1  } +   \| L_\sigma V \|_{ E'_1  }$.
From its definition (\ref{heu2green}) we see that  $\tilde{L}_\sigma^* V$  is in $E'_1 $
whenever $V \in E_3$.
Estimate  (\ref{weaksigmap}) will be deduced from:
 \begin{lem}
  \label{green1}
The map
 \begin{equation*}
  (V,W) \in E_2 \times E_2 \mapsto \rho := 
   \left\{
 \begin{array}{c}
 a  \lbrack \partial_X V \rbrack . W  \ \textbf{ on } \   \Gamma
\\  T V . W  \ \textbf{ on } \   \tilde{\Gamma} 
\end{array}
\right.
\end{equation*}
extends uniquely to a continuous bilinear map
from $E_3 \times E_3 $ to $(\text{Lip} (\Gamma \cup  \tilde{\Gamma} )'$ and  Green's identity  (\ref{greentrace}) is still valid for any couple $V,W$ in  $E_3 \times E_3 $ in the generalized sense that
 \begin{equation}
  \label{heu8888}
< L_\sigma V , W >_{ E'_1 , E_1 } =
-  \int_{ \mathcal{D} } a \partial_X  V . \partial_X  W 
+  <   \tilde{L}_\sigma^* W , V  >_{ E'_1 , E_1 }
- <  \rho  , 1 >_{ (\text{Lip} (\Gamma \cup  \tilde{\Gamma} )'  ,\text{Lip} (\Gamma \cup  \tilde{\Gamma}) } .
\end{equation}
\end{lem}
 \begin{proof}
Now let $V , W$ be in $E_3$ and let $\phi $ be in $\text{Lip} (\Gamma \cup  \tilde{\Gamma})$ there exists a function  $\Phi $  in $\text{Lip} (\mathcal{D} )$ such that 
 $\Phi |_{ \Gamma \cup  \tilde{\Gamma}} = \phi $ and such that
 $\| \Phi \|_{ \text{Lip} ( \mathcal{D} )}
  \leqslant \| \phi \|_{ \text{Lip} (\Gamma \cup  \tilde{\Gamma})}$ with $c$ independent of $\phi $. 
 From  Green's identity  (\ref{greentrace}) we infer that
 \begin{equation*}
    \int_{ \Gamma}  \rho  \phi =
-  \int_{ \mathcal{D} } L_\sigma V . \phi  W 
-  \int_{ \mathcal{D} } a \partial_X  V . \partial_X  \phi  W 
+  \int_{ \mathcal{D} }  \tilde{L}_\sigma^* \phi W .V 
\end{equation*}
  so that 
 \begin{eqnarray}
 |  \int_{ \Gamma}  \rho  \phi | &\leqslant&
\| \Phi \|_{ \text{Lip} ( \mathcal{D} )} . \|  L_\sigma V  \|_{E'_1 } . \|   W  \|_{ E_1 } 
+ \| \Phi \|_{ \text{Lip} ( \mathcal{D} )} . \|   V  \|_{E_1 } . \|  W  \|_{ E_1 } 
+  \|  \tilde{L}_\sigma^* \phi W   \|_{E'_1 } . \|  V  \|_{ E_1 } ,
\\ &\leqslant&
C \| \phi \|_{ \text{Lip} ( \Gamma \cup  \tilde{\Gamma} )} . \| V \|_{E_3  } . \|   W  \|_{ E_3 }  .
  \end{eqnarray}
Hence by Hahn-Banach theorem we get the existence of a continuous extension, which is unique because of the density of $E_2$ in $E_3$.
\end{proof}
          In order to prove the existence part of Theorem \ref{weakeq} we shall need another
         Green formula which involves the complete transposition of the operator $L$.
         At a smooth level that is for
 any $V$, $W$ in the space $E_2$  this Green identity reads:
 \begin{equation}
  \label{heu8882}
\sum_\pm \int_{ \mathcal{D}_\pm } L_\sigma V . W = 
  \int_{ \mathcal{D} }  L_\sigma^* W.V
- \int_{ \Gamma}  a  \lbrack \partial_X V \rbrack . W 
+  \int_{ \Gamma}  a  V .  \lbrack  \partial_X  W  \rbrack
- T \int_{ \tilde{\Gamma}} W.V 
\end{equation}
where  $L_\sigma^* = a \partial_X^2 + \tilde{L}_\sigma^* $ 
 denotes the adjoint of the operator  $L_\sigma$.
 \begin{lem}
  \label{green2bis}
The map
 \begin{equation}
  \label{heu8887}
  V  \in   E_2 \mapsto \tau := 
   \left\{
 \begin{array}{c}
 a  \lbrack \partial_X V \rbrack   \ \textbf{ on } \   \Gamma
\\  T V   \ \textbf{ on } \   \tilde{\Gamma} 
\end{array}
\right.
\end{equation}
extends uniquely to a continuous linear map 
from $E_4 $ to $H^{- \frac{3}{4} } (\Gamma \cup  \tilde{\Gamma})$ and Green's identity  (\ref{heu8882}) is still valid for any couple $V,W$ in $E_4 \times H^{1,2} (\mathcal{D})$ in the generalized sense that 
 \begin{eqnarray}
  \label{heu88882}
< L_\sigma V , W >_{ H^{-1} (\mathcal{D}) ,   H^{1} (\mathcal{D}) } =
  <   L_\sigma^* W , V  >_{ E'_1   , E_1 }
- <  \tau  , W|_{\Gamma \cup  \tilde{\Gamma}  }  >_{H^{- \frac{1}{2} } (\Gamma \cup  \tilde{\Gamma} )  , H^{ \frac{1}{2} } (\Gamma \cup  \tilde{\Gamma} )   } 
\\ \nonumber  +    \int_{ \Gamma}   \lbrack  \partial_X  W  \rbrack .  a  V|_{\Gamma   }  .
\end{eqnarray}
\end{lem}
\begin{proof}
By adapting classical lifting method we get that for any
 $ \tilde{W}_1$  in 
$H^{ \frac{3}{2} } (\Gamma \cup  \tilde{\Gamma} )$ and  $ \tilde{W}_2$ be in $L^2 (\Gamma)$ there exists a function $W$ in $H^{1,2} (\mathcal{D})$ such that 
$W |_{ \Gamma \cup  \tilde{\Gamma}} =  \tilde{W}_1 $, $ \lbrack  \partial_X  W  \rbrack  |_{ \Gamma}=   \tilde{W}_2$.  
The proof then follows the same lines than the one of Lemma \ref{green1}; we  use Green's identity  (\ref{heu8882}) for $V$ be in $E_2$ and a $W$ obtained by lifting  $ \tilde{W}_1$  in 
$H^{ \frac{3}{2} } (\Gamma \cup  \tilde{\Gamma} )$ and  $ \tilde{W}_2 =0$, next we use Hahn-Banach theorem and density.
\end{proof}

 Proceeding as previously we get that the estimate (\ref{weaksigmap})  also holds for the adjoint operator.
 We infer the existence of a solution to the direct problem by using Riesz's theorem and
Green's identity  (\ref{heu88882}).

      We now use again a sequence of approximation of the solution, and use the linearity of the problem together with  the 
   a-priori   estimate (\ref{weaksigmap0}) to show that this sequence is a Cauchy sequence
  in the space of the functions $V$ such that   $ \sqrt{t} \|  V(t,.,.)  \|_{L^2 ( \R^d \times \R) }$ is continuous on $(0,T)$. Completeness yields the conclusion.
     
      We finally let $\sigma$ goes to infinity. 
      The  estimate  (\ref{weaksigmap0})  is uniform with respect to  $\sigma$.
      Using weak compactness we can pass to the weak limit so that we get a solution to the uncut  equation.
            
\end{proof}

\subsection{Smoothness of the layer profile}
\label{wp'}

We are now studying the smoothness of  $V (t,x,X)$.
It is convenient to  recall here a few notations given in the introduction.
For any Frechet space $E$ of functions depending on $t,x$ and possibly on $X$ we will denote $E_D$ the set
  \begin{eqnarray*}
  E_D := \{ f \in E / \ \exists C > 0 /  \   ( \frac{ D^k f}{C^k k! }   )_{ k  \in \mathbb{N} }   \text{ is bounded in } E \} ,
  \end{eqnarray*}
where 
$D$ denotes  the material derivative $D := \partial_t  + v^0 \cdot\nabla_x$.
The spaces denote
  \begin{eqnarray*}
 \mathcal{A} &:=&  L^\infty \Big((0,T), C^{0, r} \big(  \R^d , p-\mathcal{S} (  \R_\pm) \big)\Big)
\cap  L^\infty \Big((0,T), C^{s, r} \big(  \mathcal{O}_{\pm} (t) , p-\mathcal{S} (  \R_\pm) \big)\Big) ,
\\ \mathcal{B} &:=&   L^\infty \Big((0,T), C^{0, r} \big(  \R^d \big)\Big) 
 \cap  L^\infty \Big((0,T), C^{s, r} \big(  \mathcal{O}_{\pm} (t)  \Big).
\end{eqnarray*}
Moreover we recall from section \ref{wp1} that the functions $a$ and $A$ are in $\mathcal{B}_D$.
\begin{theo}
\label{weakeqana}
For any 
$f \in  \mathcal{A}_D$,
 for any 
 $g_1 , g_2 \in\mathcal{B}_D$
  there exists  one solution 
  $V \in  \mathcal{A}_D$
  of 
  \begin{equation}
   \label{weaks}
  L  V = f \  \text{on }  \  \mathcal{D}_\pm ,
\  ( \lbrack  V \rbrack , \lbrack \partial_X V \rbrack ) = (g_1 ,g_2 )   \ \text{on }  \   \Gamma .
\end{equation}
\end{theo}

 \begin{proof}
 We first reduce  the transmission conditions on the internal boundary $ \Gamma$ to  homogeneous ones by
  defining the functions $ \tilde{V}$ by
 \begin{eqnarray}
  \label{lift}
V :=    \tilde{V}  \pm \frac{1}{2} \big( (2 g_1 + g_2) e^{\mp X} -  (g_1 + g_2) e^{\mp 2X} \big) \quad \text{when }  \quad  \pm X > 0.
     \end{eqnarray}
  so that 
  the  problem now reads as follows
  \begin{equation}
   \label{weak0h}
  L   \tilde{V} = \tilde{f} \  \text{on }  \  \mathcal{D}_\pm ,
\   \lbrack   \tilde{V} \rbrack = \lbrack \partial_X  \tilde{V} \rbrack  = 0   \ \text{on }  \   \Gamma,
\end{equation}
with  $\tilde{f}  \in  \mathcal{A}_D$.
Let us stress that there is no
 loss of regularity  in this lifting process, since the $X$-derivative has been applied innocuously to the second term in  (\ref{lift}), as well as $D$ since  $g_1 , g_2$ and $ a$ are in $ \mathcal{B}_D$

We first prove that $V $ is in  $ \mathcal{C}_D$ where we denote $ \mathcal{C} := L^\infty \big((0,T) \times  \R^d , H^1 (  \R )\big)$.
Let us first establish an a-priori estimate,
applying  for any $k \in \N$ the field $D^k$ to the problem (\ref{weak0h}) to get
  \begin{equation}
   \label{weak0hlpoidsk}
  ( \mathcal{E}-k )     \tilde{V}^{\lbrack k \rbrack} =   \tilde{f}^{\lbrack k \rbrack} + t  \tilde{V}^{\lbrack k+1 \rbrack} + t A   \tilde{V}^{\lbrack k \rbrack}   \  \text{on }  \  \mathcal{D}_\pm ,
\   \lbrack   \tilde{V}^{\lbrack k \rbrack}  \rbrack = \lbrack \partial_X  \tilde{V}^{\lbrack k \rbrack} \rbrack = 0  \ \text{on } \ \Gamma,
\end{equation}
where we  denote by $V^{\lbrack k \rbrack} := D^k  \tilde{V} $ the $k$th iterated derivative of $ \tilde{V}$ along $D$ and where
$\tilde{f}^{\lbrack k \rbrack} :=  \sum_{l=1}^{3} \ \tilde{f}^{\lbrack k \rbrack}_l ,$ where $\tilde{f}^{\lbrack k \rbrack}_1 := D^k    \tilde{f}$ whereas $ \tilde{f}^{\lbrack k \rbrack}_2$ and  $ \tilde{f}^{\lbrack k \rbrack}_3$ denote respectively the commutators:
  \begin{eqnarray*}
 \tilde{f}^{\lbrack k \rbrack}_2 := \lbrack D^k  , \mathcal{E} \rbrack =
  \sum_{l=0}^{k-1} \ \dbinom{k}{l} D^{k-l} a .  \partial_X^2  \tilde{V}^{\lbrack l \rbrack} ,
\ \tilde{f}^{\lbrack k \rbrack}_3 := - \lbrack D^k  ,t A  \rbrack = -   \sum_{l=0}^{k-1} \ \dbinom{k}{l} D^{k-l} t A .  \tilde{V}^{\lbrack l \rbrack} .
\end{eqnarray*}
The last sums have to be omitted when $k=0$.
Here we have used that  $\lbrack D^k,  t D \rbrack = k D^{k}$.
  We now multiply the first equation of (\ref{weak0hlpoidsk}) by $ \tilde{V}^{\lbrack k \rbrack}$ and we now integrate w.r.t. $X$ only. This yields for any $t,x \in (0,T) \times  \R^d $ the estimate:
  \begin{eqnarray}
  \label{f}
 \int_{ \R } a | \partial_X  \tilde{V}^{\lbrack k \rbrack} |^{2}
 + (k +  \frac{3}{4} ) \int_{ R }  | \tilde{V}^{\lbrack k \rbrack} |^{2}
 \leqslant 
  | \int_{ \R }  \tilde{f}^{\lbrack k \rbrack} .  \tilde{V}^{\lbrack k \rbrack}  |
  +  \int_{ \R } t | \tilde{V}^{\lbrack k+1 \rbrack} .  \tilde{V}^{\lbrack k \rbrack}  |
  + C_1  \int_{ \R } t | \tilde{V}^{\lbrack k \rbrack}  |^2 .
    \end{eqnarray}
Using the condition (\ref{ell}) the l.h.s. of (\ref{f}) is larger than
  \begin{eqnarray*}
  c \int_{ \R }  | \partial_X  \tilde{V}^{\lbrack k \rbrack} |^{2}
 +  \frac{3}{4} (k + 1 ) \int_{ R }  | \tilde{V}^{\lbrack k \rbrack} |^{2} .
  \end{eqnarray*}
We now bound  the r.h.s. of (\ref{f}).
Using Cauchy-Schwarz and Young  inequalities we have 
  \begin{eqnarray}
 \int_{ \R }  | \tilde{f}^{\lbrack k \rbrack}_1 .  \tilde{V}^{\lbrack k \rbrack}  |
 \leqslant 
  \frac{4}{k+1}   \int_{ \R }  | \tilde{f}^{\lbrack k \rbrack} |^{2}+   \frac{k+1}{4}   \int_{ \R }  | \tilde{V}^{\lbrack k \rbrack} |^{2}  .
 \end{eqnarray}
Integrating by parts yields
  \begin{eqnarray}
| \int_{ \R }  \tilde{f}^{\lbrack k \rbrack}_2 .  \tilde{V}^{\lbrack k \rbrack}  |
 & \leqslant&   \sum_{l=0}^{k-1} \ \dbinom{k}{l}    \int_{ \R } |  D^{k-l} a  \partial_X  \tilde{V}^{\lbrack l \rbrack} \partial_X  \tilde{V}^{\lbrack k \rbrack} | .
 \end{eqnarray}
Since $a$ is analytic, there exists $C_a > 0$ s.t. for any $l \in \N$ $|| D^l a ||_\mathcal{B}  \leqslant C_a^l l!$ so that using Cauchy-Schwarz and Young  inequalities yield
  \begin{eqnarray}
 | \int_{ \R }  \tilde{f}^{\lbrack k \rbrack}_2 .  \tilde{V}^{\lbrack k \rbrack}  |  & \leqslant& C_2 (   \sum_{l=0}^{k-1} \  \frac{k !}{l !} C_a^{k-l}  ||  \partial_X  \tilde{V}^{\lbrack l \rbrack} ||)^2  +  \frac{c}{2}  \int_{ \R }   | \partial_X  \tilde{V}^{\lbrack k \rbrack} |^{2} ,
\end{eqnarray}
where we denote here $|| f || :=  (\int_{ \R }   |f(t,x,X)|^2 dX )^\frac{1}{2}  $.
In a similar way there exists $C_3, C_A > 0$ s.t. 
  \begin{eqnarray}
| \int_{ \R }  \tilde{f}^{\lbrack k \rbrack}_3 .  \tilde{V}^{\lbrack k \rbrack}  |
 & \leqslant& C_3 (   \sum_{l=0}^{k-1} \   \frac{k !}{l !} C_A^{k-l}  ||   \tilde{V}^{\lbrack l \rbrack} ||)^2  +  \frac{1}{8}  \int_{ \R }   |  \tilde{V}^{\lbrack k \rbrack} |^{2} .
\end{eqnarray}
Finally  for $0<t<\underline{T} $ we have
\begin{eqnarray}
| \int_{ \R }  t  \tilde{V}^{\lbrack k+1 \rbrack}  |.  \tilde{V}^{\lbrack k \rbrack}  |
 & \leqslant& 
 \frac{ k+1}{8}   \int_{ \R }   |  \tilde{V}^{\lbrack k \rbrack} |^{2} +
  \frac{ 4\underline{T}  }{k+1}  \int_{ \R }   |  \tilde{V}^{\lbrack k+1 \rbrack} |^{2}  .
\end{eqnarray}
Hence
  \begin{eqnarray*}
  \frac{ c}{2}  \int_{ \R }  | \partial_X  \tilde{V}^{\lbrack k \rbrack} |^{2}
 + \frac{k +1}{4} \int_{ \R }  | \tilde{V}^{\lbrack k \rbrack} |^{2}
 \leqslant  \frac{4}{k+1}  \int_{ \R }  | \tilde{f}^{\lbrack k \rbrack}  |^{2} 
  +
 C_2 (   \sum_{l=0}^{k-1} \  \frac{k !}{l !} C_a^{k-l}  ||  \partial_X  \tilde{V}^{\lbrack l \rbrack} ||)^2
\\  +  C_3 (   \sum_{l=0}^{k-1} \   \frac{k !}{l !} C_A^{k-l}  ||   \tilde{V}^{\lbrack l \rbrack} ||)^2 
   +    \frac{ 4 \underline{T}  }{k+1}  \int_{ \R }  | \tilde{V}^{\lbrack k+1 \rbrack} |^{2}
   + C_1 \underline{T}   \int_{ \R }  | \tilde{V}^{\lbrack k \rbrack} |^{2}
    \end{eqnarray*}
thus we infer -keeping the notations $C_1$-$C_3$ for their squareroots- that
  \begin{eqnarray}
   \label{heure}
  \frac{ c}{4}   || \partial_X  \tilde{V}^{\lbrack k \rbrack} ||
 + \frac{ \sqrt{k+1}}{8}  || \tilde{V}^{\lbrack k \rbrack} ||
 \leqslant  \frac{2}{\sqrt{k+1}}   || \tilde{f}^{\lbrack k \rbrack}  || 
  +
 C_2  \sum_{l=0}^{k-1} \  \frac{k !}{l !} C_a^{k-l}  ||  \partial_X  \tilde{V}^{\lbrack l \rbrack} ||
 \\  \nonumber +  C_3   \sum_{l=0}^{k-1} \   \frac{k !}{l !} C_A^{k-l}  ||   \tilde{V}^{\lbrack l \rbrack} ||
   +   \sqrt{\frac{4\underline{T}  }{k+1}}  || \tilde{V}^{\lbrack k+1 \rbrack} || 
   + C_1 \sqrt{\underline{T} }  || \tilde{V}^{\lbrack k \rbrack}  ||  .
    \end{eqnarray}
We introduce the functions
  \begin{eqnarray*}
a_k (t,x) :=  \frac{ || \tilde{V}^{\lbrack k \rbrack}  ||   }{ k!  C^k }  , \ b_k (t,x) :=  \frac{ ||  \partial_X  \tilde{V}^{\lbrack k \rbrack}  ||   }{ k!  C^k  \sqrt{k+1}}   \text{ and } f_k (t,x) := \frac{ || \tilde{f}^{\lbrack k \rbrack}  ||   }{ (k+1)!  C^k } ,
   \end{eqnarray*}
where $C$ is a positive real which will be chosen in a few lines.
Dividing the estimate (\ref{heure}) by $k!  C^k  \sqrt{k+1}$ yields
  \begin{eqnarray}
   \label{heure2}
 \frac{ c}{4}  b_k +  \frac{1}{8}  a_k  \leqslant 2 f_k + C_2   \sum_{l=0}^{k-1} \  (\frac{C_a}{C})^{k-l} b_l + C_3   \sum_{l=0}^{k-1} \  (\frac{C_A}{C})^{k-l} a_l +   \sqrt{4\underline{T}  } C  a_{k+1}  + C_1 \sqrt{\underline{T} } a_k  .
   \end{eqnarray}
We choose $C$ large enough so that
$ \max (\frac{C_2}{\frac{C}{C_A} -1 },\frac{C_3}{\frac{C}{C_a} -1 })   \leqslant  \min (\frac{c}{8}, \frac{1}{16} )$ and then $T$ is chosen small enough so that $ \sqrt{4\underline{T}  } C  \leqslant  \frac{ 1}{64}  $ and that $ C_1 \sqrt{\underline{T} }  \leqslant  \frac{ 1}{64}  $.
Hence summing over $k \in \N$ the  estimates (\ref{heure2}) yield the a-priori estimate:   for any $t,x \in (0,\underline{T}) \times  \R^d $ 
  \begin{eqnarray}
   \label{heure3}
  \sum_{k \in \N}  (\frac{ c}{8}  b_k +  \frac{1}{32}  a_k ) \leqslant 2 \sum_{k \in \N} f_k .
  \end{eqnarray}
We now define the iterative scheme  $ (\tilde{V}^n )_{n \in \N  } $ by setting $\tilde{V}^0$ as the solution of 
  \begin{equation*}
 \mathcal{E}   \tilde{V}^0 = \tilde{f} \  \text{on }  \  \mathcal{D}_\pm ,
\   \lbrack   \tilde{V}^0 \rbrack = \lbrack \partial_X  \tilde{V}^0 \rbrack  = 0   \ \text{on }  \   \Gamma,
\end{equation*}
and $\tilde{V}^{n+1} $ as the solution of 
  \begin{equation*}
 \mathcal{E}   \tilde{V}^{n+1} = \tilde{f} +  t( D   + A  )\tilde{V}^{n}  \  \text{on }  \  \mathcal{D}_\pm ,
\   \lbrack   \tilde{V}^{n+1} \rbrack = \lbrack \partial_X  \tilde{V}^{n+1} \rbrack  = 0   \ \text{on }  \   \Gamma .
\end{equation*}
A momentÕs thought convinced that  $\tilde{V}^0$  is in $ \mathcal{C}_D$. 
Then  proceeding as in the proof of  the estimate (\ref{heure3}) 
we infer the convergence of  the iterative scheme for any $t,x \in (0,\underline{T}) \times  \R^d $  to a solution $\tilde{V}$ of  the problem (\ref{weak0h}).
Using several time slices yields 
 that $V $ is in  $ \mathcal{C}_D$. 

Now to prove Theorem \ref{weakeqana}, we increase the smoothness with respect to $x$ thanks to  the operators $\Delta_j $ of spectral localization. 
We proceed as previously dealing with the additional spectral commutators as in section 
\ref{wp4} (in particular using the commutator estimate (\ref{commumu})).

 If we denote by $\tilde{\mathcal{B}}$ the space of the functions $f(t,x,X)$ with $  \mathcal{B}$ smoothness in $t,x$ with values in $H^1 (\R)$, this yields that $ \tilde{V}$ is in $\tilde{\mathcal{B}}_D$. Then we prove by induction that $X^k \tilde{V}$ is in $\tilde{\mathcal{B}}_D$ for all $k$  in $\N$. Finally we use the equation to increase by induction the number of derivatives with respect to $X$ and get that  $ \tilde{V}$ is in $\tilde{\mathcal{A}}_D$.
 
  \end{proof}

  \subsection{Other properties of the profile}
   \label{other}

 Let us now prove that the normal projection $V \cdot n$  vanishes.
 We multiply  the equation (\ref{heu}) by $n$ and take into account the equation (\ref{neq}) for $n(t,x)$ so that we get
 \begin{equation}
  \label{heu2red}
 |n|^2  \partial_X^2  (V \cdot n)  +  \frac{X}{2}  \partial_X (V \cdot n)  - \frac{1}{2}  (V \cdot n) =
 t( \partial_t   + v^0 \cdot\nabla_x ) (V \cdot n) .
\end{equation}
Moreover taking also the scalar product of the transmission conditions on the internal boundary $X=0$ with $n$ we get 
 \begin{eqnarray}
 \label{heu21red}  (V \cdot n) |_{X=0^+  } &=& (V \cdot n) |_{X=0^- } ,
\\  \label{heu22red}   \partial_X  (V \cdot n) |_{X=0^+  }  &=& \partial_X   (V \cdot n) |_{X=0^-  } .
  \end{eqnarray}
 Proceeding as in section \ref{wp} we get an energy estimate for the problem 
  (\ref{heu2red})-(\ref{heu21red})-(\ref{heu22red}) from which we infer that  $V \cdot n$  vanishes.
 
 When $s   \geq 2  $ the
 first and second time derivatives of the  functions $n$ and $\omega_0$ are in 
   $ L^\infty([0,T], C^{0, r} (  \R^d ))  $.
   This allows to get some estimates for the
 second time derivative $\partial_t^2 V $ by commuting $\partial_t^2  $ with
  the equation  (\ref{heu2r})-(\ref{heu21r}). 
Hence we can define a trace at $t=0$ of all the terms in the equations  (\ref{heu2r})-(\ref{heu21r}). We hence get that the trace $V  |_{t=0  }$ satisfy  (\ref{ieu20})-(\ref{ieu21})-(\ref{ieu22}).

 
 \subsection{Higher order expansions}
\label{hoprofile}

We are now concerned with the following terms in the expansion with respect to $\ep t$ of the solutions of the Navier-Stokes equations.
Actually in this section we show that if the initial data is piecewise smooth on each side of the interface  $\{ \phi^{0} = 0  \}$
-that is if  $s = + \infty$- then it  is possible to write a complete  asymptotic expansion of the vorticity of the form:
\begin{eqnarray}
\label{complete1}
 \omega^\ep (t,x) =   \sum_{j \geqslant 0 }  \sqrt{\ep t}^j   \  \Omega^{j  }  (t,x,\frac{\phi^{0}(t,x)}{\sqrt{\ep t} } )  +O(\sqrt{\ep t}^\infty)  ,
 \end{eqnarray}
where the first  profile $\Omega^{0 } $ is the one
 constructed in the previous section:  $\Omega^{0 } := \Omega $.

Let us explain the underlying intuition to guess the expansion (\ref{complete1}): 
when the initial data is globally smooth,
 the vorticity $\omega^\ep$ given by the NS equations 
 admits a regular (Taylor-like) expansion:
\begin{eqnarray}
\label{formul0}
 \omega^\ep (t,x) :=  \sum_{j \geqslant 0 } ( \ep t)^j   \  \omega_{j  }  (t,x )  +O((\ep t)^\infty) ,
\end{eqnarray}
so that  we expect that when the initial data is piecewise smooth on each side of the interface  $\{ \phi^{0} = 0  \}$,
the same should still hold true when we incorporate the fast scale $\frac{\phi^{0}(t,x)}{\sqrt{\ep t} }$ so that 
 we should write $\omega^\ep$ through the $\ep$-dependent profile:
\begin{eqnarray}
\label{formul1}
 \omega^\ep (t,x) =   \Omega^\ep (t,x,\frac{\phi^{0}(t,x)}{\sqrt{\ep t} }),
\end{eqnarray}
where $\Omega^\ep (t,x,X)$ admits a regular expansion:
\begin{eqnarray}
\label{formul2}
 \Omega^\ep (t,x,X) :=  \sum_{j \geqslant 0 }  \sqrt{\ep t}^j   \  \Omega_{j  }  (t,x,X )  +O(\sqrt{\ep t}^\infty) .
\end{eqnarray}
Once again to prove this we will consider the velocity formulation, looking first for a determination of the velocity  profiles of the expansion
\begin{eqnarray}
  \label{complete2} v^\ep (t,x) =  v^0 (t,x)  + \sum_{j \geqslant 1 }  \sqrt{\ep t}^j   \  V^{j  }  (t,x,\frac{\phi^{0}(t,x)}{\sqrt{\ep t} } )  +O(\sqrt{\ep t}^\infty) ,
 \end{eqnarray}
where the profile $V^{1}$  is  the one constructed in the previous section: $V^{1} := V$.
Then we will recover the vorticity profiles by
 plugging  (\ref{complete1}) and (\ref{complete2}) into  the  relation (\ref{curlep}) and equalling the terms of order  $\sqrt{\ep t}^{j}$; what leads to:
\begin{eqnarray}
\label{rotrotj}
\Omega^j = \rot_x V^j +
n \wedge \partial_X  V^{j+1}  . 
\end{eqnarray}
To guess the amplitudes in the expansion of the pressure we have to remember that 
we follow a  Rankine-Hugoniot(-BKW) approach looking for profiles such that the expansions  solve -formally- the equations (\ref{R1})-(\ref{R2}) 
  on each side  $\{ \pm \phi^{0} > 0  \}$ plus the  continuity of the velocity, of the pressure and of the vorticity on the  interface  $\{ \phi^{0} = 0  \}$.
  Taking the divergence and the normal scalar product  yields  the pressure problem: 
\begin{align}
\label{pressure1} \Delta p^\ep 
&= - \text{div} \ (v^\ep \cdot \nabla v^\ep )
\\ \label{pressure2} \lbrack   p^\ep    \rbrack
&=0,
\\ \label{pressure3} \lbrack  \partial_n  \ p^\ep    \rbrack
&=   \lbrack  ( -   v^\ep \cdot \nabla v^\ep    + \ep \Delta v^\ep). n  \rbrack,
\end{align}
where  the notation $\lbrack    \cdot   \rbrack $ stands for the jump 
across  the  interface  $\{ \phi^{0} = 0  \}$ (that is for a function $f(t,x)$: 
we denote $\lbrack f   \rbrack 
:= f |_{ \phi^{0} =0^+   } -  f  |_{ \phi^{0} =0^-   }$).
If we plug the expansion  (\ref{complete2}) into the r.h.s. of the equations  (\ref{pressure1})-(\ref{pressure3})  we are led to look for a pressure expansion of the form
\begin{eqnarray}
  \label{complete3} p^\ep (t,x) =  p^0 (t,x)  + \sum_{j \geqslant 2 } \frac{ \sqrt{\ep t}^j }{t}  \  P^{j  }  (t,x,\frac{\phi^{0}(t,x)}{\sqrt{\ep t} } )  +O(\sqrt{\ep t}^\infty) ,
\end{eqnarray}

The  profiles above are of the  following  form: for $ \pm X > 0 $,
\begin{eqnarray}
\label{zyg}
U  (t,x,X) &:=&  \underline{U}  (t,x) + \tilde{U}  (t,x,X) ,
\end{eqnarray}
 where the function $\tilde{U}  (t,x,X)$ is rapidly decreasing  when $ \pm X  \rightarrow \infty$, and the letter $U$ is the placeholder for the $\Omega^{j  }$, the  $V^{j  }$ and the $P^{j  }$.
 We will refer to the term $\underline{U} $ as the regular part and to the term $ \tilde{U}$ as the layer part.
 The  layer part  $ \tilde{P}^2$ is equal to  $ \tilde{P}^2 = t P$ where
  $P$ is the profile of the  previous section.
  This possibility to be smoothly factorized by $t$ is very particular to the order $2$.
  It comes from the fact that the (a-priori)  main contribution given by the Laplacian in the r.h.s. of the equation (\ref{pressure3}) vanishes thanks to the orthogonality property (\ref{orthov}).
  It is therefore fair to acknowledge that the heuristic argument given in section
 \ref{amplitudes} to guess the first order expansion (\ref{formu2p})  is at less not very robust and leads to a correct amplitude almost by chance !
  Furthermore to satisfy the pressure continuity we will have to add to this  layer part  $ \tilde{P}^2$
   a regular part $\underline{P}^{2}$, as anticipated in Remark \ref{ppp}.

To determine  the velocity and pressure profiles  we proceed iteratively, determining at the step $j$
 the velocity profile $V^{j }$, the  regular part of the pressure profile $\underline{P}^{j }$ and the layer part of the following pressure profile $ \tilde{P}^{j+1}$,
  from the profiles   already known by the previous steps.
The step  $j=1$ was done in the previous section.
We now explain how to do a step $j \geqslant 2$ when the previous ones are done.

 \begin{rem}
    \rm
  Regarding the transmission aspect of our strategy
we refer  to  \cite{fourier},  in the setting of the approximation of semi-linear symmetric hyperbolic systems of PDEs by the vanishing viscosity method. 
We also mention   \cite{guesrauchX},  \cite{guesrauch} by Gu\`es and Rauch  in the context of  internal waves for semi-linear symmetric hyperbolic systems
and  \cite{gueswilliams} by Gu\`es and Williams in the context of  viscous shocks  profiles.
 \end{rem}
 \subsubsection{Jump conditions}
\label{hoprofile1}

The transmission conditions read:
  \begin{eqnarray}
  \label{heuho2}
  V^{j} |_{X=0^+ , \phi^{0} = 0^+ } &=&  \ V^{j} |_{X=0^- , \phi^{0} = 0^- }  ,
   \\ \label{wp720}  ( a \partial_X \tilde V^{j} +  \partial_n  V^{j-1} ) |_{X=0^+ , \phi^{0} = 0^+ }  &=& (a \partial_X \tilde V^{j}  + \partial_n  V^{j-1}) |_{X=0^- , \phi^{0} = 0^- }  ,
   \\ \label{heu211}  
   P^{j} |_{X=0^+ , \phi^{0} = 0^+ } &=&  \ P^{j} |_{X=0^- , \phi^{0} = 0^- } 
 \end{eqnarray}
 We use  the vectorfields $w^k$ defined by (\ref{imitating}) to decompose  the vectorfields  as $u =  \frac{1}{a} (u. n) n + u_\text{tan} $ where $u_\text{tan} = \sum_k  \ c_k  (  V .w^k ) w^k $.
In what follows we will use the notation $\lbrack    \cdot   \rbrack $ for two different meanings:
\begin{eqnarray}
\lbrack    \underline{U}    \rbrack 
:=
  \underline{U} |_{ \phi^{0} =0^+   } -    \underline{U}  |_{ \phi^{0} =0^-   }  \text{ and }
\lbrack   \tilde{U}   \rbrack 
:=
\tilde{U}
 |_{X=0^+ , \phi^{0} =0^+   } - \tilde{U}  |_{X=0^+ , \phi^{0} =0^+ } .
 \end{eqnarray}
By assumption the previous velocity profile $V^{j}$ satisfies
  \begin{eqnarray}
 V^{j-1} |_{X=0^+ , \phi^{0} = 0^+ } &=&  \ V^{j-1} |_{X=0^- , \phi^{0} = 0^- } .
 \end{eqnarray}
 As a consequence there holds
  \begin{eqnarray*}
 \lbrack  \partial_n  V^{j-1}_\text{tan}  \rbrack =  \lbrack ( \partial_n  V^{j-1} )_\text{tan} \rbrack,
 \text{ and }    \lbrack ( \partial_n  V^{j-1} ) . n  \rbrack n =  \lbrack  \partial_n  \big( (V^{j-1}  . n) n  .\big)\rbrack ,
 \end{eqnarray*}
and the jump conditions  (\ref{heuho2})-(\ref{wp720})-(\ref{heu211}) reduce to
  \begin{eqnarray}
   \label{heu210} 
\lbrack  \underline{V}^{j}  . n    \rbrack &=& - \lbrack  \tilde{V}^{j} . n   \rbrack ,
  \\     \label{press5}
 \lbrack  \tilde{V}^{j}_\text{tan}  \rbrack &=& -  \lbrack  \underline{V}^{j}_\text{tan}  \rbrack ,
  \\    \label{press6}
  \lbrack a  \partial_X  \tilde{V}^{j}_\text{tan}  \rbrack &=&  -  \lbrack   \partial_n V^{j-1}_\text{tan}  \rbrack ,
   \\   \label{press7}
    \lbrack a  \partial_X ( \tilde{V}^{j} . n) \rbrack &=&  -  \lbrack  \partial_n (V^{j-1} . n )  \rbrack ,
    \\   \label{press8}
      \lbrack  \underline{P}^{j}     \rbrack &=&  - \lbrack  \tilde{P}^{j}    \rbrack .
  \end{eqnarray}

 \subsubsection{Divergence equation}
\label{hoprofile2}

First plugging  the ansatz (\ref{complete2})  into the equation   (\ref{R2}), equalling the terms  at order $\sqrt{\ep t}^j$ yields:
\begin{eqnarray}
\label{nonoj}
  \partial_X  (V^{j} . n) = -  \mathrm{div}_x \,   V^{j-1} .
  \end{eqnarray}
   We first begin to take the limit $ \pm X  \rightarrow \infty$ of the equation to get
   \begin{eqnarray}
\label{nonojbar}
 \mathrm{div}_x \,   \underline{V}^{j-1} = 0.
  \end{eqnarray}
  Then subtracting the equation (\ref{nonojbar}) to the  equation (\ref{nonoj}) yields
\begin{eqnarray}
\label{nonojj}
  \partial_X  (  \tilde{V}^{j} . n) = -  \mathrm{div}_x \,    \tilde{V}^{j-1} .
  \end{eqnarray}
 For  $\pm X > 0$, we 
 integrate between $X$ and $\pm \infty$
 to find
 \begin{eqnarray}
  \label{lerayprofile2bis}
  \tilde{V}^{j} . n :=  -  \int_X^{\pm \infty}  \mathrm{div}_x \,    \tilde{V}^{j-1}  .
 \end{eqnarray}
This already yields 
 that (\ref{press7}) is satisfied.
 To do this we will use  that the divergence can be rewrited as
  \begin{eqnarray*}
\mathrm{div}_x \,  V =  \frac{1}{a} \partial_n (V. n) - \sum_k a_k w^k . \nabla_x (  w^k . V)
  \end{eqnarray*}
where the $w^k$ are tangential to the hypersurface and the $a_k$ are scalar.
Since $\mathrm{div}_x \,   \underline{V}^{j-1} = 0$ there holds
  \begin{eqnarray*}
  \lbrack  \partial_n ( \underline{V}^{j-1} . n  )  \rbrack &= & \lbrack  a \sum_k a_k w^k . \nabla_x (\underline{V}^{j-1} .w^k ) \rbrack   \\  &= & -  \lbrack a \sum_k a_k w^k . \nabla_x (\tilde{V}^{j-1} . w^k ) \rbrack 
   \end{eqnarray*}
because of (\ref{heuho2}). Then
  \begin{eqnarray*}
  \partial_n (V^{j-1} . n  ) &= & 
  \lbrack   \partial_n ( \tilde{V}^{j-1} . n  ) 
 +  a \sum_k a_k w^k . \nabla_x (\tilde{V}^{j-1} . w^k ) \rbrack 
   \\  &= &  \lbrack  a \mathrm{div}_x \,  \tilde{V}^{j-1}  \rbrack 
     \\  &= & -  \lbrack a \partial_X  \tilde{V}^{j} . n  \rbrack 
 \end{eqnarray*}
because of (\ref{nonojj}). 

 Moreover proceeding in the same way at the next order yields
    \begin{eqnarray}
\label{nonojbar3}
 \mathrm{div}_x  \underline{V}^{j} = 0.
  \end{eqnarray}

\subsubsection{Velocity equation}
\label{hoprofile3}
 
 When plugging the ansatz (\ref{complete2}) and (\ref{complete3})  into the equation  (\ref{R1})  the term 
$$ (\partial_t  + v^0 \cdot\nabla_x ) \phi^{0} . \partial_X  \tilde V^{j} $$
 that should a priori appears in the equation, vanishes thanks to the eikonal equation satisfied by $ \phi^{0}$,
 and equalling the terms of order  $\sqrt{\ep t}^{j}$ one get one equation of the form
  \begin{equation}
  \label{heuho}
  (L- \frac{j-1}{2})   {V}^{j} =  f^{j} + \nabla_x P^{j}
+  \partial_X P^{j+1}  n   ,
\end{equation}
where 
$f^{j}$ would be decomposed into 
 $f^{j} = \underline{f}^{j} + \tilde{f}^{j}$
  -as the profiles $U$ were in  (\ref{zyg})- and
 contains terms already determined so that it could be considered as a source term at this step.
 
\subsubsection{Regular part}
\label{hoprofile4}
 
 We first begin to take the limit $ \pm X  \rightarrow \infty$ of the equation  (\ref{heuho}) in order to kill the layer parts and to obtain the following equation for the regular part:
 \begin{equation}
  \label{heuhoreg}
\frac{j}{2}  \underline{V}^{j}  + t (\partial_t  \underline{V}^{j} + v^0 \cdot\nabla_x   \underline{V}^{j}    +   \underline{V}^{j} \cdot\nabla_x v^0 ) +  \nabla_x \underline{P}^{j} 
 =   \underline{f}^{j} ,
\end{equation}
  Since the hypersurface $\{t=0\}$ is characteristic  the equation (\ref{heuhoreg}) does not need any initial condition.
 \begin{lem}
 There is only one couple $(\underline{V}^{j}, \underline{P}^{j} ) $ piecewise smooth satisfying the equations (\ref{heuhoreg})-(\ref{nonojbar3}) on each side $\mathcal{O}_{\pm} (t) $ of the interface, and  the transmission conditions (\ref{heu210})-(\ref{press8}). 
  \end{lem}
 \begin{proof}
 The idea is to recast the problem into an equation of the form
 \begin{eqnarray}
  \label{recast}
\frac{j}{2} \underline{V}^{j}  +  (\partial_t  + v^0 \cdot\nabla_x ) t  \underline{V}^{j}  
 =  \Pi (\underline{V}^{j})        +   \underline{f}^{j} .
  \end{eqnarray}
 on each side $\mathcal{O}_{\pm} (t) $, where $ \Pi $ is a pseudodifferential operator of order $0$.
For instance
if we denote 
$$g := - \frac{j}{2}  \underline{V}^{j}  - t (\partial_t  \underline{V}^{j} + v^0 \cdot\nabla_x   \underline{V}^{j}    +   \underline{V}^{j} \cdot\nabla_x v^0   )
 +  \underline{f}^{j}.$$
Taking the divergence and the scalar product by $n$ of the equation (\ref{heuhoreg}) we get the following problem for the pressure:
 \begin{eqnarray}
 \label{sa1}   \triangle \underline{P}^{j} &=&  \dive g   \quad \text{for}   \quad  x  \in \mathcal{O}_{\pm} (t),
\\  \label{sa2}
\lbrack    \partial_n \underline{P}^{j}     \rbrack =
 &=&  \lbrack  g.n    \rbrack  ,
\\    \label{sa3} 
\lbrack  \underline{P}^{j}    \rbrack  &=& -
   \lbrack  \tilde{P}^{j}   \rbrack.
 \end{eqnarray}
 Using that the regular parts of the velocity profiles -that is here $v^0 $, $  \underline{V}^{j}$ and $  \underline{V}^{j-2}$- are divergence free
(see  the conditions (\ref{R2}) and  (\ref{nonojbar3})) we simplify the equation (\ref{sa1}) into
 \begin{eqnarray}
  \triangle \underline{P}^{j} &=&  \dive ( \underline{f}^{j}  -  \underline{V}^{j} \cdot\nabla_x v^0   ) -   \sum_{i,k}   \partial_k v^0_i . \partial_i \underline{V}^j_k .
  \end{eqnarray}
 To simplify the equation (\ref{sa2}) 
  we apply the operator $\partial_t  + v^0 \cdot\nabla_x$ to the condition (\ref{heu210}) and simplify thanks to the equation of $n$.
   This leads to 
 \begin{eqnarray}
 \lbrack   \partial_n \underline{P}^{j}      \rbrack = \lbrack
 \frac{j}{2} \underline{V}^{j}
 + t\big( ( \partial_t  + v^0 \cdot\nabla_x ) \tilde{V}^{j} +   \tilde{V}^{j} \cdot\nabla_x v^0
  \big) +    \underline{f}^{j}  \rbrack .n
    \end{eqnarray}
 These simplifications done it becomes clear that the operator 
 $\Pi (\underline{V}^{j})  := -   \nabla_x \underline{P}^{j} - \underline{V}^{j}  \cdot\nabla_x v^0 $
  is of order $0$.
 
  We now set $\gamma := \frac{j}{2}-1 $ and we denote 
   $\mathcal{O}_\pm$  the space-time domains  
   $\mathcal{O}_\pm :=  \{ (t,x)  \in (0,T) \times \R^d / \ x \in \mathcal{O}_\pm (t) \} $
   and $\mathcal{H}$ the operator:
 \begin{eqnarray}
  \label{recast8}
\mathcal{H}  := \gamma + 
  (\partial_t  + v^0 \cdot\nabla_x ) t  .
  \end{eqnarray}
 We have
 \begin{lem}
 Both problems
 \begin{eqnarray}
 \mathcal{H} \underline{V} =  \underline{f} \text{ for } \mathcal{O}_{\pm} 
   \end{eqnarray}
   are well-posed: for any $f $ in $L^2 ( \mathcal{O}_{\pm} ) $ there exists only one solution $\underline{V}$ in $L^2 ( \mathcal{O}_{\pm} ) $. 
   Moreover
    $\underline{V}$ in $C  \big((0,T),L^2 ( \mathcal{O}_{\pm}) \big) $. 
    If $f$ is in $L^\infty ((0,T), C^{s,r} ( \mathcal{O}_{\pm} (t) )) $ then
     $\underline{V}$ in $L^\infty ((0,T), C^{s,r} ( \mathcal{O}_{\pm} (t) )) $.
    \end{lem}
   \begin{proof}
   Existence and uniqueness in $L^2$ are based as in   section  \ref{wp} on a Green identity: for any smooth function $ \underline{W} $ there holds
 \begin{eqnarray}
 \int_{\mathcal{O}_\pm} \mathcal{H} \underline{V} .  \underline{W}
 = \gamma  \int_{\mathcal{O}_\pm} \underline{V} .  \underline{W}
   -   \int_{\mathcal{O}_\pm} t \underline{V} .  (\partial_t  + v^0 \cdot\nabla_x ) \underline{W}
   + T   \int_{\mathcal{O}_\pm (T)}  \underline{V} .  \underline{W}  .
       \end{eqnarray}
   Here we have used that 
 \begin{eqnarray}
  \partial_t  \int_{\mathcal{O}_\pm}  =
   \int_{\mathcal{O}_\pm} (\partial_t  + v^0 \cdot\nabla_x ) .
   \end{eqnarray}
   This identity leads to the a-priori estimate:
 \begin{eqnarray}
  (\gamma + \frac{1}{2})  \int_{\mathcal{O}_\pm} \underline{V}^2   +
   \frac{t}{2}  \int_{\mathcal{O}_\pm (t)} \underline{V}^2
   \lesssim  \int_{\mathcal{O}_\pm}   \underline{f}^2 .
   \end{eqnarray}
This estimate, by a weak-strong argument, leads to uniqueness. 
Besides by Riesz theorem we get the existence part.
Using Rychkov's extension theorem \ref{rychkov} the estimates reduce to estimate in the full plane for globally smooth coefficients and source. 
Then we apply the operators $\Delta_j$ to get
 \begin{eqnarray}
 \mathcal{H} \Delta_j \underline{V} =  \underline{f} +
 t  \lbrack v^0 \cdot\nabla_x , \Delta_{j} \rbrack \underline{V} .
   \end{eqnarray}
We multiply this equation by  $p |\Delta_j V |^{p-2} . \Delta_j V$, where $p >2$, 
and we integrate by parts  in $t,x$ over $(0,T)  \times \R^d$.
The commutators are estimated through paraproduct estimates which follow the same lines as the one of the previous sections.
Finally we let $p$ goes to infinity.
  \end{proof}
  From the previous lemma we infer the existence and uniqueness in $L^2$ for the equation  (\ref{recast}). To get  piecewise smoothness we notice that  the operator 
 $\Pi$ satisfies the transmission property so that it preserves piecewise smoothness at any order (we recall that in this section the interface is assumed to be smooth so that we only need here to refer to classical works as for instance \cite{rempel}).
   \end{proof}

\subsubsection{Layer part}
\label{hoprofile5}

 The tangential part of  $\underline{V}^{j} $ is in general discontinuous at $\{\phi^{0} =0\}$. 
 This jump will have to be compensated by the layer part $ \tilde{V}^{j}$, which we now look for.
 We subtract  the equation (\ref{heuhoreg}) to the equation  (\ref{heuho})  to get for  $ \tilde{V}^{j}$ an equation of the form 
  \begin{equation}
  \label{heuho26}
  (L- \frac{j-1}{2} )  \tilde{V}^{j} =   \partial_X  \tilde{P}^{j+1} . n  + \tilde{f}^{j} .
\end{equation}
 \begin{lem}
 There exists a couple $(\tilde{V}^{j}, \tilde{P}^{j+1} ) $ piecewise smooth satisfying the equations (\ref{heuho26})-(\ref{lerayprofile2bis}) on each side $\pm X >0$ of the interface, and  the transmission conditions  (\ref{press5})-(\ref{press6}).
  \end{lem}
 \begin{proof}
 Since the equation (\ref{lerayprofile2bis}) already determines the normal part $\tilde V^{j} . n$ of  the velocity profile $\tilde{V}^{j} $ we will reduce the problem into a problem for the tangential part  $\tilde{V}^{j}_\text{tan} $. The pressure $\tilde{P}^{j+1} $ can be determined in function of  $\tilde{V}^{j} $:
  \begin{eqnarray}
    \label{press}
| n|^2 \partial_X \tilde{P}^{j+1} = (L- \frac{j-1}{2} )  (\tilde{V}^{j}  . n)  + 2 t \tilde{V}^{j} \cdot\nabla_x v^0 . n - \tilde{f}^{j} . n .
   \end{eqnarray}
Here we have  taken the scalar product  of the equation (\ref{heuho2}) with $n$ and use the equation of $n$.
 Taking (\ref{press}) into account  the equation (\ref{heuho26}) now reads 
  \begin{eqnarray}
    \label{press2}
(L- \frac{j-1}{2} )  \tilde{V}^{j} = \frac{1}{a}  \big((L- \frac{j-1}{2} )  (\tilde{V}^{j}  . n)\big)  n 
   + 2 t (\tilde{V}^{j} \cdot\nabla_x v^0 . n) n + \tilde{f}^{j}_\text{tan} . 
  \end{eqnarray}
But  using again the equation of $n$ we have
  \begin{eqnarray}
\nonumber (L- \frac{j-1}{2} )  \tilde{V}^{j}_\text{tan} &=&
(L- \frac{j-1}{2} )  \tilde{V}^{j} -  \frac{1}{a} \big( (L- \frac{j-1}{2} )  (\tilde{V}^{j}  . n)  \big) n
\\   \label{press3}
&&+ \frac{t}{a} (\tilde{V}^{j}  . n)(  \frac{1}{a} (n. \nabla_x v^0 .n ) n -    {}^{t} \nabla_x v^0 .n )   )
 \end{eqnarray}
Mixing  (\ref{press2}) and (\ref{press3}) leads to
  \begin{eqnarray}
   \label{press4}
(L- \frac{j-1}{2} )( \tilde{V}^{j}_\text{tan} ) =
f_\text{tan} + t  \tilde{A} \tilde{V}^{j} 
 \end{eqnarray}
where $ \tilde{A} \tilde{V}^{j} $ denotes the local $0$-order operator
  \begin{eqnarray}
   \tilde{A} \tilde{V}^{j}  :=  2 ( \tilde{V}^{j} \cdot\nabla_x v^0 . n) n 
   + \frac{1}{a} (\tilde{V}^{j}  . n)(  \frac{1}{a} (n. \nabla_x v^0 .n ) n -    {}^{t} \nabla_x v^0 .n )   ).
 \end{eqnarray}
Thanks to the  analysis of section \ref{wp}  there exists $ \tilde{V}^{j}_\text{tan}$ satisfying the equations (\ref{press4})-(\ref{press5})-(\ref{press6}).

   \end{proof}

\subsubsection{On the smoothing of the level function}
\label{hoprofile6}
 
We should be tempted to take into account the viscous smoothing of the level function $\phi^0$ that is
instead of the solution $\phi^0$ of
\begin{eqnarray*}
D v^0 := \partial_t  \phi^0 + v^0 \cdot \nabla_x \phi^0 = 0 .
\end{eqnarray*}
to consider the solution $  \phi^\ep$ of the "viscous eikonal equation":
\begin{eqnarray}
\label{viscouseik}
  \partial_t  \phi^\ep + v^\ep \cdot \nabla_x \phi^\ep = \ep  \Delta_x \phi^\ep ,
  \end{eqnarray}
with, again, the initial data:
 $  \phi^\ep |_{ t =0 }  =  \phi_0$.
Actually reading again the formula  (\ref{10})-(\ref{11})-(\ref{12}) we notice that these terms are the ones in prefactor of $\frac{1}{\sqrt{\ep t} } \partial_X  U $ when one applies the transport-diffusion operator to a function of the form
 $U (t,x,\frac{\phi^0  (t,x)}{\sqrt{\ep t} })$. 

In that case proceeding as previously we have 
 that $\phi^\ep$ admits an expansion of the form 
\begin{eqnarray}
\label{illust}
\phi^\ep (t,x) \sim \phi^\ep_a (t,x) :=  \phi^0 (t,x) + \ep t \Phi (t,x,\frac{\phi^0  (t,x)}{\sqrt{\ep t} }) 
\end{eqnarray}
where he profile $ \Phi (t,x,X) $ is defined as
$ \Phi (t,x,X) :=  \underline{\Phi}  (t,x) + \tilde{\Phi} (t,x,X)$,
where the $ \underline{\Phi}$ are the solutions of 
\begin{eqnarray}
\label{az0}
(1  +  t D )  \underline{\Phi}  =  \Delta_x \phi^0 
\end{eqnarray}
and $ \tilde{\Phi}$ is the solution of the equations 
\begin{eqnarray}
\label{az1}
(  \mathcal{E}- \frac{1}{2} -  t D )  \tilde{\Phi}  = 0.
\end{eqnarray}
with the following transmission conditions across $X=0$:
 \begin{eqnarray}
 \label{heu21rq}
  \lbrack 
  \tilde{\Phi} 
   \rbrack   = -  (\underline{\Phi}_+ - \underline{\Phi}_- ) ,
\text{ and }   \lbrack  \partial_X \tilde{\Phi}  \rbrack  =  0 ,
  \end{eqnarray}
where $\underline{\Phi}_\pm$ are suitable extensions of  $\underline{\Phi}$.
Let us show that  $ \Delta_x \phi^\ep_a$ does not have a jump.
We have, at least in a neighborhood of $\{ \phi^0 \}$, the following identity:
\begin{eqnarray}
 \Delta_x \phi^\ep_a =  \Delta_x \phi^0 + a \partial^2_X \tilde{\Phi}
 + \sqrt{\ep t}  (  \Delta_x \phi^0 .  \partial_X \tilde{\Phi}
 + 2 n  \cdot\nabla_x   \partial_X \tilde{\Phi} )
 + \ep t \Delta_x  \Phi .
  \end{eqnarray}
In the expression above the profiles are evaluated in $X = \frac{\phi^0 (t,x)}{\sqrt{\ep t} }$.
Using the equation (\ref{az1}) we have 
\begin{eqnarray}
\label{az2}
\lbrack  a  \partial_X^2  \tilde{\Phi} \rbrack =
 \lbrack   (1  +  t D ) \tilde{\Phi} \rbrack.
\end{eqnarray}
Since $D:=  \partial_t   + v^0 \cdot\nabla_x$ is tangent to the hypersurface $\{\phi^0 (t,x) = 0 \}$ we infer from  (\ref{heu21rq}) that 
\begin{eqnarray}
\label{az3}
\lbrack  a  \partial_X^2  \tilde{\Phi}  \rbrack = -
 \lbrack (1  +  t D ) \underline{\Phi} \rbrack.
\end{eqnarray}
Using  the equation (\ref{az0}) we get $\lbrack  a  \partial_X^2  \tilde{\Phi} \rbrack = -
 \lbrack   \Delta_x  \underline{\phi}^0 \rbrack$, so that 
\begin{eqnarray}
 \lbrack  \Delta_x \phi^\ep_a  \rbrack =  \sqrt{\ep t}  \lbrack (  \Delta_x \phi^0 .  \partial_X \tilde{\Phi}
 + 2 n  \cdot\nabla_x   \partial_X \tilde{\Phi} )  \rbrack
 + \ep t  \lbrack \Delta_x  \Phi  \rbrack.
  \end{eqnarray}
Besides differenting with respect to $x$ the transmission conditions (\ref{heu21rq}) we get $\lbrack \Delta_x \phi^\ep_a  \rbrack = 0 $.

One can see the link between the two points of view by a Taylor expansion: for any smooth profile $U$ there exists a smooth profile $U^\flat$ such that 
\begin{eqnarray}
U (t,x,  \frac{\phi^\ep_a  (t,x)}{\sqrt{\ep t} }) =
U (t,x,  \frac{\phi^0  (t,x)}{\sqrt{\ep t} }) 
+ \sqrt{\ep t} U^\flat (t,x, \frac{ \phi^0  (t,x)}{\sqrt{\ep t} }) .
\end{eqnarray}
In effect it is sufficient to define  $U^\flat$ by 
 $U^\flat (t,x,X) := U^\dagger (t,x,X, \Phi  (t,x,X))$ where $U^\dagger  (t,x,X,h) := h^{-1} \big(U^\flat (t,x,X+h)-U^\flat (t,x,X) \big) $.

 We believe that the point of view of the phase smoothing adopted in this section could be interesting in view of future extensions to singular vortex patches.
   Let say for instance that $\nabla_x \phi_{0}$ has a jump on the hypersurface $\psi_{0}$. 
  Then  the corresponding solutions $ \phi^\ep $ given by (\ref{viscouseik}) admit an expansion of the form:
\begin{eqnarray}
\label{phabis}
   \phi^\ep (t,x)  \sim  \phi^0 (t,x) + \sqrt{\ep t}  \phi^1 (t,x,\frac{\psi^\ep  (t,x)}{\sqrt{\ep t} }))
  \end{eqnarray}
where $\psi^\ep $ also solve the eikonal eiquation:
\begin{eqnarray}
\label{veik2}
  \partial_t  \psi^\ep + v^0 \cdot \nabla_x \psi^\ep = \ep  \Delta_x \psi^\ep ,
  \\  \psi^\ep |_{t=0  }  = \psi_{0} .
\end{eqnarray}
Let assume that $\psi_{0}$ is smooth so that $\psi^\ep$ is well-approximated by the solution $\psi^0$ of the inviscid 
 eikonal equation:
\begin{eqnarray}
\label{veik3}
  \partial_t  \psi^0 + v^0 \cdot \nabla_x \psi^0 = 0 ,
   \\  \psi^0 |_{t=0  }  = \psi_{0} .
\end{eqnarray}
 As a consequence  the expansion (\ref{phabis}) can be simplified into 
\begin{eqnarray}
\label{pha2}
   \phi^\ep (t,x)  \sim  \phi^0 (t,x) + \sqrt{\ep t}  \phi^1 (t,x,\frac{\psi^0  (t,x)}{\sqrt{\ep t} })) .
  \end{eqnarray}
Now we plug this later expansion  (\ref{pha2})  into  the expansion (\ref{viscousprofile}) to get 
\begin{eqnarray}
\label{viscousprofile}
 \omega^\ep (t,x)  \sim  \tilde\Omega (t,x,\frac{\phi^0  (t,x)}{\sqrt{\ep t} }, \frac{\psi^0  (t,x)}{\sqrt{\ep t} }  ) ,
\end{eqnarray}
where
\begin{eqnarray}
 \tilde\Omega (t,x,X_1 , X_2 ) := \Omega (t,x,X_1 +  \phi^1 (t,x,X_2) ) 
\end{eqnarray}


  \subsection{Tangent parallel layers do not interact }
\label{close}

 As mentioned in Remark \ref{strat} the proof \`a la Chemin  of Theorem \ref{compendium} also succeed to cover the case where the initial vorticity $\omega_0$ is discontinuous across two hypersurfaces   $ \{ \phi_0  = 0  \}$ and  $ \{\phi_0  =  \eta \}$, where $ \eta > 0$: it yields that the corresponding solution of the Euler equations has -at time $t$- a vorticity $\omega_0$ piecewise smooth and discontinuous across the two hypersurfaces $ \{ \phi^0  (t,.)  = 0  \}$ and  $ \{\phi^0  (t,.)  =  \eta \}$, where $\phi^0 $ is again the solution of the transport equation  (\ref{eik1})-(\ref{eik2}). For the corresponding solution of the Navier-Stokes equations two layers of width $ \sqrt{\ep t} $ develop around the hypersurfaces $ \{ \phi^0  (t,.)  = 0  \}$ and  $ \{\phi^0  (t,.)  =  \eta \}$. When $t$ proceeds it comes a time when the layers overlap.
 However they do not  interact and the NS velocities can be described by the superposition of the two layers that is by an expansion of the form 
\begin{eqnarray}
\label{2formu}
 v^\ep (t,x) &\sim & v^0  (t,x) + \sqrt{\ep t}\,  \Big( V (t,x,\frac{\phi^{0}(t,x)}{\sqrt{\ep t}}
 + W (t,x,\frac{\phi^{0}(t,x)- \eta}{\sqrt{\ep t}} )  \Big),
  \end{eqnarray}
  where the profile $V  $ is  again the solution of  the equation  (\ref{heu2r}) with the transmission conditions (\ref{heu21r}), whereas  the extra profile $W$  is the solution of 
   the equation  (\ref{heu2r}), with the transmission conditions:
 \begin{eqnarray*}
 \lbrack W  \rbrack = 0
  \text{ and }    \lbrack  \partial_X W    \rbrack   = - \frac{n \wedge  (\check{\omega}^0_+ -  \check{\omega}^0_-  )}{a} ,
  \end{eqnarray*}
    where  the brackets denote the jump of $W(t,x,X)$  across  $\{ X=0 \}$ that is $ \lbrack W  \rbrack :=  W  |_{X=0^+  } - W |_{X=0^- }$ and
$ \omega^0_\pm$ are some well-chosen extensions
 of the restriction of the Euler vorticity to both sides $ \{ \pm (\phi^{0}(t,.)- \eta) > 0 \} $.

 The point is that both profiles $V$ and $W$ satisfy the orthogonality condition 
 $V \cdot n  = W \cdot n  = 0$ so that the Burgers term $$(V+W)  \cdot n  \partial_X (V+W)$$ which should induce a nonlinear coupling of the two layers identically vanishes (as the self-interaction did in section \ref{Transparency}).

 In the same vein -using locally this argument- we can infer that 
 there is no non-linear  interaction between two vortex patches  tangent in one point.

\subsection{Well-prepared expansions}
\label{bienprep}

Let us also stress that it is also possible to construct some asymptotic expansions of the form:
\begin{eqnarray}
\label{complete1wp}
 \omega^\ep (t,x) =   \sum_{j \geqslant 0 }  \sqrt{\ep}^j   \  \Omega^{j  }  (t,x,\frac{\phi^{0}(t,x)}{\sqrt{\ep } } )  +O(\sqrt{\ep }^\infty)  ,
 \\ \label{complete2wp} v^\ep (t,x) =  v^0 (t,x)  + \sum_{j \geqslant 1 }  \sqrt{\ep }^j   \  V^{j  }  (t,x,\frac{\phi^{0}(t,x)}{\sqrt{\ep } } )  +O(\sqrt{\ep }^\infty) ,
  \\ \label{complete3wp} p^\ep (t,x) =  p^0 (t,x)  + \sum_{j \geqslant 2 }  \sqrt{\ep }^j   \  P^{j  }  (t,x,\frac{\phi^{0}(t,x)}{\sqrt{\ep } } )  +O(\sqrt{\ep }^\infty) ,
\end{eqnarray}
for both the Euler and the NS equations. 
Let us see how the construction is modified, at least for the first velocity profile.

For the Euler equation one get for $V^{1}$ the profile equation 
\begin{eqnarray}
\label{eqpro1wp}
(D  +  A ) V^{1}  = 0,
\end{eqnarray}
which is a transport equation along the flow of $ v^0$ with an extra term, local,  of order $0$.
Notably this equation involves $X$ only as a parameter.

For the NS equations one get the profile equation:
\begin{eqnarray}
\label{eqpro2wp}
(D  +  A - a \partial^2_X ) V^{1}  = 0
\end{eqnarray}
which is hyperbolic in $t,x$ and parabolic in $t,X$.
In particular   the hypersurface $\{t=0\}$ is now non-characteristic.

The equations (\ref{eqpro1wp})-(\ref{eqpro2wp}) are therefore both well-posed when set for $X$ in the whole real line. 
Let us explain an analogy which makes this sound natural.
The expansions  (\ref{complete1wp})-(\ref{complete3wp}) could be seen as  viscous and local counterparts of the asymptotic expansions  of weakly nonlinear geometric optics (cf. for instance \cite{metivieroptics}).
In this latter case the profiles are periodic with respect to the fast variable $X$ and the small parameter $\sqrt{\ep }$ refers to short wavelengths.
In this setting of  geometric optics it  is well-known -at least from experts- that a viscosity of size $\ep $ only parabolizes the profile equations corresponding to an ansatz of the form (\ref{complete1wp})-(\ref{complete3wp}), without disturbing the well-posedness (on the contrary actually). We refer to the papers \cite{guesIMA},  \cite{CheverryCMP}, \cite{junca} which illustrate this remark with some close settings. 

Now if one looks for some solutions $v^\ep$ of the NS equations with a vortex patch $v_I$ as initial data, one has to prescribe zero initial data for the layers so that the compatibility condition between the transmission conditions and the initial condition  on the "corner" $\{t=X=0\}$ are not satisfied even at order zero (cf. Remark \ref{experts}). 
This ruins any hope for smoothness with respect to $X$, which leads to some difficulties in the analysis of the stability of the expansions.
One way to get the compatiblity conditions is to choose the initial condition, restricting ourselves to a kind of well-prepared initial data, assuming that the initial data is already of the form (\ref{complete1wp})-(\ref{complete3wp}), that we could qualified by "well-prepared"  initial data.
 It is possible to prove  the existence of such well-prepared initial data
 thanks to some Borel lemma. We refer for an insight of the method to the papers  \cite{toulouse} and  \cite{fourier}.
For such data  one observed that the transmission conditions persist when time proceeds, and  the lifetime of such expansions is the one of the solution of the Euler equation ("the ground state") which traps the main part of the nonlinearity of the problem.

  \subsection{Weaker singularities}
\label{ws}

 Let us mention that both kind of  expansions (\ref{complete1})-(\ref{complete3}) and (\ref{complete1wp})-(\ref{complete3wp})  can be useful in the case of conormal singularities weaker than  vortex patches.
 If for instance the vorticity is continuous through the initial internal boundary $\{ \phi_0 (x) = 0 \}$ but not the derivative of the vorticity, then the profiles  $\Omega^0$, $V^{1}$ and $P^{2}$  do not depend on $X$ anymore. 
 The layers appear only at the following orders (as for instance in  (\ref{illust})). 
 More generally if for $k \in \N$ the vorticity is $C^k$  through the initial internal boundary $\{ \phi_0 (x) = 0 \}$  it  is possible to write a complete  asymptotic expansion of the vorticity of the form:
\begin{eqnarray*}
 \omega^\ep (t,x) =    \omega^0 (t,x) 
+ \sum_{j = 1 }^{k-1}  (\ep t)^j   \   \underline{\Omega}^{j  }  (t,x  )  
+ \sum_{j \geqslant k}  \sqrt{\ep t}^j   \   \Omega^{j  }  (t,x,\frac{\phi^{0}(t,x)}{\sqrt{\ep t} } )  +O(\sqrt{\ep t}^\infty)  .
 \end{eqnarray*}
 Such an observation was mentioned in the setting  of the approximation of  semi-linear symmetric hyperbolic
systems of PDEs by the vanishing viscosity method in  \cite{fourier}.

      \subsection{Stronger singularities}
\label{ss}

 Let us now talk about stronger singularities than vortex patches, 
  for which layers of larger amplitude are expected. 
   Let us first deal with vortex sheets which involve some discontinuity jumps of the velocity (instead of jumps of the vorticity).
 An initial velocity piecewise smooth with discontinuity jumps across an hypersurface  $ \{ \phi_0  = 0  \}$ is an extremely instable configuration  for the Euler equations: in general the jumps of the velocity do not stay localized on a smooth hypersurface   $ \{ \phi^0 (t,.) = 0  \}$  when time proceeds. 
  A few positive results are available with analytic initial data: local in time persistence was proved by Bardos and al. in  \cite{sulem} (see also some extension to global persistence for small analytic perturbation by Caflisch and Orellana in \cite{co} and by  Duchon and Robert  in \cite{duchon}), but 
  the papers \cite{co2}, \cite{lebeau}, \cite{wu}, \cite{grenier} ruin willingness to extend out the unphysical case of analytic data.
Still a few physical phenomena are known to bring
  some stability such as compressibility (see  \cite{coulombelsecchi} in the  two dimensonal  supersonic case) or surface tension  (see  \cite{hou}, \cite{ambrose}).

 If for one of the previous reason the velocity  $v^0$ given by the Euler equation (or an appropriately modified inviscid system) stays piecewise smooth with discontinuity jumps only across an hypersurface  $ \{ \phi^0 (t,.) = 0  \}$, we expect that  the corresponding velocities given by the Navier-Stokes  equation (or a modified viscous system) should be given by  an expansion of the form:
\begin{eqnarray*}
 v^\ep (t,x) &\sim & v^0  (t,x) + V (t,x,\frac{\phi^{0}(t,x)}{\sqrt{\ep t}} ),
  \end{eqnarray*}
 since the perturbation term  $V$ should have to compensate the jumps of the inviscid solution $v^0$ through  the transmission conditions $ \lbrack V  \rbrack =  - \lbrack  v^0    \rbrack$    where  the first brackets denote the jump of $V$  across  $\{ X=0 \}$ 
    and the second one the jump of $v^0$ across $ \{ \phi^0 (t,.) = 0  \}$. 
    In addition we have to prescribe $ \lbrack  \partial_X V  \rbrack = 0$.
Actually because of some nonlinear effects we have to consider also the next term in the expansion:
\begin{eqnarray}
\label{2dformu2}
 v^\ep (t,x) &\sim & v^0  (t,x) + V^0 (t,x,\frac{\phi^{0}(t,x)}{\sqrt{\ep t}} ) + \sqrt{\ep t} V^1 (t,x,\frac{\phi^{0}(t,x)}{\sqrt{\ep t}} )  
  \end{eqnarray}
  to close the profile equation.
 Actually plugging this ansatz into the NS equations and equalling according to the orders of $\sqrt{\ep t}$ yields -with the notations of the introduction- the equations:
 \begin{equation}
  \label{prandtl}
  V^0 \cdot n  = 0,\
  (L + \frac{1}{2}) V^0 =  t \{ (V^0  \cdot  \nabla_x (v^0 + V^0) + ( V^1  \cdot n ) \partial_X V^0 \}   \text{ and }    \partial_X V^1 \cdot n  = -\dive_x V^0 .
  \end{equation}
This equation contains the same difficulty than the Prandtl equations: to get rid of $V^1$ in the second equation we use the third equation what leads to a loss of derivative. 
It is likely that here again
some positive results are possible in the case of analytic data.
Indeed Caflisch and Sammartino in \cite{WASCOM} have succeed  to prove well-posedness for analytic data in the well-prepared counterpart of the equations  (\ref{prandtl}) 
in two dimensions when the radius of curvature of the curve is much larger than the thickness of the layer.
A  solution more radical but also mathematically simpler to avoid this loss of derivative is to consider some cases without variation in the transverse directions.
For instance we can consider the following very special case: we set the phase  
$ \phi^{0}(t,x)  \equiv  x_1 $ and an inviscid  velocity $v^0$ of the form  $v^0 (t,x)  \equiv  f(x_1 ) e_2 $ where the function $f$ is $C^\infty_c$ on $\R_\pm $. 
The solution of NS equations are then simply
$$ v^\ep (t,x) =   V (\frac{x_1}{\sqrt{\ep t}} )  e_2 ,$$
 where $V$ is the solution of the elliptic equation
$$  \partial^2_X V +   \frac{X}{2} \partial_X V = 0 $$
 with the following transmission conditions on $ \{ X =0 \}$:
 \begin{equation}
  \lbrack V  \rbrack =  - \lbrack  f    \rbrack  \text{ and } \lbrack \partial_X  V  \rbrack =0,
  \end{equation}
   where $ \lbrack  f    \rbrack$ denotes the jump of $f$ across  $ \{ x_1 =0 \}$.

If now we strengthen again the amplitude looking for expansions of the form 
$$ v^\ep (t,x) =  (\ep t)^{-\alpha } \, V (t,\frac{x_2}{\sqrt{\ep t}} ) e_1 ,$$
with $\alpha >0$, we get for $V$ the equation $$  \partial^2_X V +   \frac{X}{2}  \partial_X V + \alpha V = 0 ,$$ which is still coercive in $H^1 (\R )$ for $\alpha <  \frac{1}{4}$ but admit $0$ for eigenvalue for  $\alpha =  \frac{k}{2}$, with $k \in \N^*$. The corresponding eigenfunctions are the Hermite functions $H_k  := \partial^{k-1}_X H_1  $ with $ H_1 := e^{- \frac{X^2}{4} } $.
Hence when strengthening  the amplitude of the transition layer (that is when increasing  $\alpha$) there are still some non-trivial solutions of the profile equation but the nature of the profile problem totally changes, from a elliptic problem with non-homogeneous condition transmissions to an eigenvalue problem.

 In this latter case it is then possible to consider also some viscous perturbations singular with respect to several dimensions. For instance
it is  possible to exhibit some family of solutions of the two dimensonal NS equations with velocities of the form:  
\begin{eqnarray}
\label{2dformu3}
 v^\ep (t,x) &=& \frac{1}{\sqrt{\ep t}}  V (t,\frac{x}{\sqrt{\ep t}} ),
  \end{eqnarray}
 where the profile $V(t,X) $ vanishes when the fast variable $X:= (X_1 , X_2 )$ goes to infinity. 
 The corresponding vorticities are of the form
\begin{eqnarray}
\label{2dformu4}
  \omega^\ep (t,x) = \frac{1}{\ep t}  \Omega (t,\frac{x}{\sqrt{\ep t}} )   \text{ with }     \Omega := \mathrm{curl}_X V.
  \end{eqnarray}
  Plugging these ansatz into the NS equations, whose vorticity formulation  reads as follows in  $2$d:
\begin{align}
\label{ns2d} \partial_t \omega^\ep + v^\ep \cdot \nabla \omega^\ep 
&= \ep \Delta   \omega^\ep ,
\end{align}
  and equalling the terms of order $ \frac{1}{(\ep t)^2}$ and those of order $ \frac{1}{\ep t}$, we get the following pair of equations:
\begin{eqnarray}
\label{ns2d2}
  V \cdot \nabla_X \Omega = 0  \text{ and }  (\Delta_X +  \frac{1}{2} X  \cdot \nabla_X + 1 )  \Omega = t  \partial_t \Omega ,
   \end{eqnarray}
  where we denote $\Delta_X := \partial^2_{X_1} + \partial^2_{X_2} $ and $\nabla_X := ( \partial_{X_1} , \partial_{X_2} )$.
  We then see that the first equation in (\ref{ns2d2}) is satisfied if the vorticity profile $\Omega$ is radially symmetric since in this case the corresponding velocity profile $V$ is orthoradial. The initial hypersurface  $ \{ t  = 0  \}$ is characteristic for the second equation 
   in (\ref{ns2d2}) which have therefore parasite solutions. Actually  omitting the $X$ dependence we get the ODE:  $\Omega = t  \partial_t \Omega$ whose solutions are  $\Omega = C \ln t$. However only one namely  $\Omega \equiv  0 $ has a correct behaviour for $t$ near $0$. Now setting $t=0$ in the  second equation we get the equation 
   $$ (\Delta_X +  \frac{1}{2} X  \cdot \nabla_X + 1 )  \Omega =0$$
    whose solution is 
   given by $ \Omega (X) =  C e^{-X^2 }$ ; where $C$ is determined by the conservation of the vorticity mass. The ansatz (\ref{2dformu3}) describe the
   viscous smoothing of an initial Dirac mass at $x=0$ and are usually referred as Oseen vortex. 
  We refer here for instance to the papers   \cite{gallaywayne},  \cite{gallaygallagher}, and to the references therein, for a much more precise study.
  We do not claim any novelty in the previous lines but we think that it was thought-making to incorporate a little bit of this material here. In particular we found interesting to gather the profile equation (\ref{ns2d2}) and the profile equation (\ref{heu2r}) corresponding to the vortex patches, and to observe in particular  the shift of the spectrum caused by the difference of amplitude, which make the analysis quite different.

  \section{Stability }
\label{stab}

This section is devoted to the proof of Theorem \ref{TheoStab}.
Once again we will work with the velocity formulation.
We will prove:
 \begin{theo}
 \label{TheoStab2}
 For any $k \in \N^*$ there exists $\ep_0 > 0$ such that for $0 < \ep   < \ep_0$
for any $(t,x)  \in (0,T) \times \R^d $
 \begin{eqnarray*}
 \label{velocitystab}
 v^\ep (t,x) = v^0 (t,x) +  \sum_{j = 1 }^k  \sqrt{\ep t}^j   \  V^{j  }  (t,x,\frac{\phi^{0}(t,x)}{\sqrt{\ep t} } ) +  \sqrt{\ep t}^{k+1 }  v^{\ep }_R
 \end{eqnarray*}
 where the profiles $V^{j  }$ are the ones of the previous section and 
  the sequence of the remainders $( v^{\ep }_R  )_{0 < \ep  < \ep_0}$ is in
 $\mathcal{F}$.
  \end{theo}
 Theorem \ref{TheoStab} is a corollary of this result.

 \begin{proof}
 For any $k \in \N^*$ there exists $k' >k$ such that the function $ v^\ep_a$ given by 
 \begin{eqnarray*}
 \label{velocitystabis}
 v^\ep_a (t,x) := v^0 (t,x) +  \sum_{j =  1 }^{k' } \sqrt{\ep t}^j   \  V^{j  }  (t,x,\frac{\phi^{0}(t,x)}{\sqrt{\ep t} } ) .
 \end{eqnarray*}
 is an approximated solution of the NS equations in the sense that 
\begin{eqnarray}
\label{formuR}
\partial_t v^\ep_a  + v  \cdot \nabla_x v^\ep_a  +  \nabla_x p^\ep_a = \ep   \Delta  v^\ep_a  +  \sqrt{\ep t}^{k+1 } F^\ep_R   ,
 \\  \dive v^\ep_a  =  \sqrt{\ep t}^{k+1 } \tilde{F}^\ep_R   ,
\end{eqnarray}
 with $( F^{\ep }_R  )_{0 < \ep  < \ep_0}$ and  $( \tilde{F}^{\ep }_R  )_{0 < \ep  < \ep_0}$  in $\mathcal{F}$.
 We observe that to prove Theorem \ref{TheoStab2} it is sufficient to prove that 
$  \tilde{v}^\ep_R :=  \sqrt{\ep t}^{-(k+1)} (v^\ep  - v^\ep_a )$ is in  $\mathcal{F}$.
To do that we recast the NS equations into the following equations for the remainder  $ \tilde{v}^\ep_R$:
\begin{eqnarray}
\label{formuRbis}
t (\partial_t \tilde{v}^\ep_R  + v  \cdot \nabla_x \tilde{v}^\ep_R ) +  \frac{k+1}{2  }  \tilde{v}^\ep_R   +  \nabla_x p^\ep_R= \ep t  \Delta  \tilde{v}^\ep_R  +  F^\ep_R   ,
 \\  \label{formuRbisdiv}
 \dive \tilde{v}^\ep_R  =    \tilde{F}^\ep_R   ,
\end{eqnarray}
 where the NS  pressure $p^\ep$ have been decomposed into
\begin{eqnarray*}
p^\ep  (t,x) =  p^\ep_a  (t,x) +\frac{ \sqrt{\ep t}^{k+1 }}{t} p^\ep_R  (t,x) .
\end{eqnarray*}
We immediately express the pressure remainder $p^\ep_R$ in function of the velocity $ \tilde{v}^\ep_R$ taking the divergence of the equation (\ref{formuRbis}) (and then using  the equation (\ref{formuRbisdiv})) so that 
 the problem now reads
  \begin{eqnarray}
\label{formuR2}
t (\partial_t \tilde{v}^\ep_R  + v  \cdot \nabla_x \tilde{v}^\ep_R ) +  \frac{k+1}{2  }  \tilde{v}^\ep_R - \ep t  \Delta  \tilde{v}^\ep_R  =  \Pi ( v^\ep , \tilde{v}^\ep_R ) +  \check{F}^\ep_R   ,
\end{eqnarray}
 where $ \Pi$ is bilinear of order $0$  and  $( \check{F}^{\ep }_R  )_{0 < \ep  < \ep_0}$ is  in $\mathcal{F}$.
 It is now a classical observation since the paper
 \cite{lastucedolivier} by Gu{\`e}s that -in such a process of remainder estimate- the nonlinear terms in the equation  (\ref{formuR2}) have some $ \ep t $ in factor so that  the  $ ( \tilde{v}^\ep_R  )_{0 < \ep  < \ep_0}$   live till $T$.
 As a consequence we now
 drop the tilde and the index $ \ep$ to simplify and we briefly explain how to get the needed energy estimate.
We apply the operators $\Delta_j $ of  spectral localization to the equations (\ref{formuR2}) to get for $j \geq -1$
  the equations:
 \begin{eqnarray}
\label{formuR2loc}
t (\partial_t  + v  \cdot \nabla_x ) \Delta_{j} v_R  +  \frac{k+1}{2  } \Delta_{j} v_R  - \ep t  \Delta   \Delta_{j} v_R  =  \check{F}_{R,j}
\end{eqnarray}
where
 \begin{eqnarray*}
  \check{F}_{R,j} :=
  \Delta_{j} \Pi ( v , v_R )  +   \Delta_{j} \check{F}_R   + t \lbrack v  \cdot \nabla_x ,  \Delta_{j}   \rbrack v_R 
\end{eqnarray*}
For $p > 2$ we multiply the equation  (\ref{formuR2loc}) by $p | \Delta_{j} v_R   |^{p-2} .  \Delta_{j} v_R $
and we integrate by parts  in $t,x$ on $(0,T)  \times \R^d $ 
 to get the following estimate
  \begin{eqnarray}
  \nonumber 
 (\frac{ (k+1)p}{2} -1)  \int_{ (0,T)\times \R^d  }  |   \Delta_{j} v_R |^p 
     + T ( \int_{ \R^d  }  |  \Delta_{j} v_R  |^p )(T)
     + \ep t   \int_{ (0,T)\times \R^d  }   p(p-1) a  |   \Delta_{j} v_R |^2   | \nabla_x  \Delta_{j} v_R|^{p-2}
\\      \lesssim \int_{ (0,T)\times \R^d  }   p |   \check{F}_{R,j}  |.  |  \Delta_{j} v_R  |^{p-1}
 \lesssim  p \|  \check{F}_{R,j}    \|_{ L^p ((0,T)\times \R^d )}  .  \|  \Delta_{j} v_R  \|_{ L^p ((0,T)\times \R^d ) }^{p-1}  .  \label{formuR3}
    \end{eqnarray}
 For $s'  \in \N$, for any $r'  \in (0,1)$, we multiply by $2^{j(s'+r')p}$.
 Then we use the commutator estimates  (\ref{commu}) and (\ref{commu3})
 to control the first and the third terms of  $\check{F}_{R,j}$.
 We take advantage of their factor $t$ to absorb them by the second term in  (\ref{commu3}) through Gronwall lemma. Finally 
  we take the $p$-th squareroot and we let  $p$ goes to infinity to conclude that the sequence $$( \sqrt{\ep t}^{s'+r'}  \|  \tilde{v}_R^{\ep }  \|_{L^\infty ( (0,T)  , C^{s',r'} (\R^3 )}     )_{0 < \ep  < \ep_0}$$ is bounded. 
  Since this holds for arbitrary $s'  \in \N$ and $r'  \in (0,1)$, this means that the family 
 $(  \tilde{v}_R^{\ep } )_{0 < \ep  < \ep_0}$ is in  $\mathcal{F}$.
 \end{proof}

  \begin{rem}
   Theorem  \ref{TheoStab2} proves the stability of complete expansions. 
 It would be also interesting to look at the stability of short expansions that is to understand to what extent it is possible to justify an expansion with a few terms assuming as weak tangential smoothness as possible. In this direction another approach would be to develop a multi-phase extension of what is done here.
 \end{rem}
 
 I warmly thanks  J-Y. Chemin, R. Danchin, T. Gallay and O. Glass for some useful discussions.
 
\def\cprime{$'$}

\end{document}